\documentclass{article}
\usepackage[a4paper,top=2cm,bottom=2cm,left=2cm,right=2cm]{geometry}

\usepackage{amsmath}
\usepackage{amssymb}
\usepackage{amsthm}
\usepackage{mathrsfs}
\usepackage{authblk}
\usepackage{t-angles}
\usepackage{mathtools}
\usepackage[all]{xy}
\usepackage[bbgreekl]{mathbbol}
\def\pulb{\ar@{}[dr]|(0.2){\mbox{\Large{$\lrcorner$}}}}
\allowdisplaybreaks[4]

\usepackage{graphicx}  
\usepackage[linktocpage=true]{hyperref}
\usepackage{comment}
\usepackage{tikz-cd}
\usepackage{stmaryrd}
\usepackage[toc,page]{appendix}
\usepackage{bbm}

\newtheorem{thm}{Theorem}[section]
\newtheorem{theorem}[thm]{Theorem}

\newtheorem{corollary}[thm]{Corollary}
\newtheorem{lemma}[thm]{Lemma}
\newtheorem{proposition}[thm]{Proposition}
\theoremstyle{definition}
\newtheorem{definition}[thm]{Definition}
\newtheorem{example}[thm]{Example}

\newtheorem{remark}[thm]{Remark}

\newcommand{\ud}{\underline{\mathrm{d}}}
\newcommand{\uDelta}{\underline{\Delta}}
\newcommand{\rH}{\mathrm{H}}
\newcommand{\rG}{\mathrm{G}}

\newcommand{\rGL}{{\mathrm{GL}}}
\newcommand{\uM}{{\underline{\mathcal{M}}}}

\newcommand{\beq}{\begin{equation}}
\newcommand{\eeq}{\end{equation}}
\newcommand{\co}{\mathrm{co}}
\newcommand{\uB}{\underline{B}}

\title{{\bf On Braided Differential Calculi\\ and Quantum G-structures}}
\author{{\large Antonio Del Donno$^{1}$}
} 
\author{{\large Giovanni Gava$^{1}$}
}
\author{{\large Emanuele Latini$^{2}$}
}
\author{{\large Thomas Weber$^{1}$}
}
\affil{
\centerline{\sl $^{1}$Mathematical Institute of Charles University}
\centerline{\sl Sokolovská 49/83, 18675 Prague 8, Czech Republic}

~\\

\centerline{\sl
{ $^{2}$Alma Mater Studiorum - Università di Bologna}}

\centerline{\sl  Via Zamboni 33, 40126 Bologna, Italy}
}
\date{\today}

\begin{document}
\maketitle

\begin{abstract}
\noindent We develop a theory of first order differential calculi in braided monoidal categories and classify braided covariant and bicovariant calculi on braided Hopf algebras. We show that, under certain conditions, bicovariant calculi can be transmuted to braided bicovariant calculi. For Radford--Majid biproducts, we combine bicovariant calculi on a Hopf algebra and braided bicovariant calculi on the corresponding braided Hopf algebra to covariant smash product calculi. The associated Maurer--Cartan form is shown to decompose into a direct sum of the Maurer--Cartan forms of the structure Hopf algebra of the quantum principal bundle and the braided Hopf algebra on the base. Geometrically, this construction realises the quantum affine extension of a given Hopf algebra, and we prove that the resulting quantum principal bundle is equipped with a frame resolution induced by the quantum Maurer--Cartan form. Building on this correspondence, we introduce and develop the notion of quantum $\textrm{G}$-structure, proving that quantum $\textrm{G}$-structures are quantum frame resolutions on the reduction. The theory is illustrated by examples based on transmutations of higher analogues of Sweedler's Hopf algebra and on the braided quantum plane, seen as a Yetter--Drinfeld module of $\mathcal{O}_q(\mathrm{GL}_2)$.  
\end{abstract}

\bigskip


\tableofcontents

\section*{Introduction}
G-structures are fundamental objects in differential geometry. They encode geometric structures on smooth manifolds through reductions of the structure group of the frame bundle, providing a unified language for a wide range of geometric settings. Within this framework, Cartan geometry emerges as the natural language for the study of G-structures: under suitable assumptions, a Cartan connection simultaneously encodes both the reduction of the frame bundle and a compatible affine connection on the tangent bundle \cite{sharpe}.
The basic example illustrating this correspondence is the affine homogeneous space
\[
\operatorname{Aff}(n,\Bbbk)/\operatorname{GL}(n,\Bbbk).
\]
The Maurer--Cartan form of the affine group $\operatorname{Aff}(n,\Bbbk)$, viewed as the semidirect product of the translation group $\Bbbk^n$ with the general linear group $\operatorname{GL}(n,\Bbbk)$, contains the soldering form as its translational component. Consequently, any reduction of the principal bundle
\[
\textrm{Aff}(n,\Bbbk)\longrightarrow
\textrm{Aff}(n,\Bbbk)/\textrm{GL}(n,\Bbbk)
\]
to a subgroup $\textrm{G}\subseteq \textrm{GL}(n,\Bbbk)$ accordingly determines a $\textrm{G}$-structure. This construction is closely related to the notion of a \emph{frame resolution}, whose essential ingredient is the soldering form. Indeed, the latter induces an isomorphism between the tangent bundle and an associated vector bundle, thereby encoding the geometric information carried by the underlying $\textrm{G}$-structure \cite{parabook}. In the Cartan approach to $\textrm{G}$-structures, the simplest model one can consider is provided by the affine extension of a Lie group. Since the splitting of the affine Lie algebra is equivariant, the Maurer--Cartan form of the affine group naturally encodes both the information contained in the soldering form, and hence in the frame resolution described above, and that of the affine connection.

\medskip

The aim of this paper is to generalise this picture to the setting of noncommutative differential geometry. In this context, a quantum principal bundle is described by a faithfully flat Hopf--Galois extension, while the quantum analogue of a frame bundle is provided by the notion of a \emph{quantum frame resolution}, first introduced in \cite{BrzMaj}. Our goal is to develop an appropriate notion of quantum $\textrm{G}$-structure within this framework from a Cartan-geometric perspective. To this end, we rely on the theory of quantum principal bundle reductions, originally developed in \cite{hajac} and subsequently extended in \cite{pagani}, and make use of the affine extension model.
Algebraically the affine is mirrored by the quantum principal bundle obtained from a smash product algebra (or, more generally, a crossed product algebra). More precisely, given a Hopf algebra $H$, one can construct the noncommutative-geometric analogue of a semidirect product via the smash product algebra $B\# H$, where $B$ is a left $H$-module algebra. Although $B\# H$ is, in general, only an associative algebra, it becomes a Hopf algebra when $B$ is equipped with the structure of a braided Hopf algebra in the category of Yetter--Drinfeld modules ${}_H^H\mathcal{YD}$. The resulting Hopf algebra, known as the Radford--Majid biproduct, has been described in \cite{Radford,MajidBos}.

The study of the quantum affine models and of quantum $\textrm{G}$-structures thus requires a careful treatment of differential calculi on the Radford--Majid biproduct. Indeed, as is well-known, one of the salient features of noncommutative differential geometry is the non-uniqueness of differential calculi, and in many instances it is not a priori clear which differential structure would be the one carrying meaningful geometrical information. For the Radford--Majid biproduct this issue can be addressed, in considerable generality, by means of the \emph{smash product calculus}, first introduced in \cite{PflaumSchauenburg} and subsequently generalised to crossed product algebras, and further studied in \cite{aflw} and in \cite{SciWeb}. As the Radford--Majid biproduct features a Hopf algebra $\uB$ in the braided monoidal category of Yetter--Drinfeld modules of a Hopf algebra $H$, to fully understand differentials on $\uB\# H$ one must first develop a consistent theory of first order differential calculi in braided monoidal categories and, more in depth, for braided Hopf algebras. The detailed study of such differential structures is one of the main features of the paper.

Differential calculi in the braided setting are, in fact, not yet well-understood in full generality, and only partial results are currently available in the literature \cite{BespalovCalcs,AzizMajid}. We therefore introduce the general theory of braided covariant first order differential calculi, with particular emphasis on the case of bialgebras and Hopf algebras internal to a braided monoidal category. We show the existence of a universal object in the category of braided covariant calculi on a braided bialgebra, namely the universal calculus, and that for the case of braided Hopf algebra such differential structures are classified via appropriate ideals in the augmentation ideal, thereby generalising the classification theorem of Woronowicz \cite{Woronowicz1989}. We then discuss explicit examples of braided covariant calculi; in particular, we examine in depth the case of the braided quantum plane, keeping in mind a realisation of the quantum affine extension. For the latter, we prove that no (non-trivial) braided bicovariant calculus exists once the target category is fixed to be that of Yetter--Drinfeld modules over the quantised coordinate algebra of the general linear group of rank two.

This non-existence result illustrates that bicovariance is, in general, a genuinely restrictive condition within the braided setting. It is, however, not clear a priori whether working within a braided monoidal category should be expected to imply, in comparison with the ordinary category of vector spaces, a systematic loss or a systematic gain of such structural properties. For the case of quasitriangular and coquasitriangular Hopf algebras, we show that transmutation \cite{MajidFoundation} gives a partial answer on this front: bicovariant FODCi enjoying some additional features on a quasitriangular Hopf algebra correspond to braided bicovariant FODCi on the corresponding braided Hopf algebra, while the same pattern holds in full generality for the coquasitriangular case. Beside this context, no general pattern of this kind is expected, and the answer depends sensitively on the case under investigation. Motivated by this observation, we present in this paper a surprising and, we believe, rather striking result exhibiting a case in which bicovariance is in fact gained, again by taking advantage of the theory of transmutation. We consider the transmutation of the Sweedler quasitriangular Hopf algebra $E_n$ over itself, which produces a braided Hopf algebra $\underline{E_n}$ in the category of left $E_n$-modules ${}_{E_n}\mathrm{Mod}$, together with a left-covariant differential calculus on it which is \emph{not} bicovariant. We prove that this calculus is, in fact, \emph{braided} bicovariant, thereby giving an explicit realisation of how certain differential structures may acquire additional symmetries precisely when understood within a braided monoidal category.

Having examined braided covariant calculi, we are then in the position to give a coherent description of smash product calculi on Radford--Majid biproducts. As a first result in this direction, we carefully establish the covariance properties of the latter. We then undertake a detailed study of the coinvariant 1-forms associated with a left-covariant, or bicovariant, differential structure on $\uB\# H$. In particular, we show that the corresponding cotangent space splits according to the nature of the chosen differential calculus, this splitting arising from an analogous decomposition of the quantum Maurer--Cartan form into two components. We discuss how this decomposition naturally reflects the geometric picture of a Cartan connection decomposing as the sum of a principal connection and a soldering form, and how the latter canonically induces a quantum frame resolution.

Although the present work is devoted to the quantum affine model, we expect that the ideas developed here should extend beyond this setting. A natural direction is the study of quantum analogues of Cartan geometries modeled on more general homogeneous spaces, where one would seek a local description in terms of possible suitable sheaf-theoretic techniques, such as in \cite{aflw} or more algebro-geometrically as in \cite{RosenbergKontsevich}. From the algebraic viewpoint, another possible direction would be to replace the Radford--Majid biproduct by the more general crossed product constructions, and to investigate the conditions under which such algebras admit compatible Hopf algebra structures and the corresponding theory of differential calculi.

\medskip

A short outline of the manuscript follows. In Section \ref{section:braided_hopf_algebras} we collect the algebraic preliminaries required in the sequel, mostly reviewing standard material from the theory of quantum principal bundles and of Hopf algebras in braided monoidal categories, and establishing the braided generalisation of the fundamental theorem of Hopf modules. In Section \ref{First order braided differential calculi} we develop the theory of first order differential calculi over bialgebras and Hopf algebras internal to a braided monoidal category. We establish the existence of a universal calculus, prove a braided analogue of Woronowicz's classification theorem \cite{Woronowicz1989} in terms of ideals in the kernel of the counit, and illustrate the theory through explicit examples. We prove the link between bicovariant calculi on (co)quasitriangular Hopf algebras and braided bicovariant calculi on the corresponding transmutation. We then return, in Section \ref{chapter:smash_calculi}, to the Radford--Majid biproduct, examining its interaction with the smash product calculus and, in particular, the properties of the Maurer--Cartan form in this setting. Finally, in Section \ref{section:framings_and_G_structures} we discuss the notion of quantum frame resolution from a Cartan point of view. We show how the Maurer--Cartan form for the Radford--Majid biproduct canonically induces a quantum frame resolution, and, in a general setting, that the reduction of a quantum frame resolution is again a quantum frame resolution. In Appendix \ref{section:classicalpicture} we collect some results on the classical theory of Cartan geometry and G-structures, motivating how the notion of framing, instead of the one of frame bundle, is the correct one to consider in the algebraic noncommutative setting. As already mentioned, explicit examples featuring Sweedler's Hopf algebras, the braided quantum plane, and the $q$-deformed coordinate algebra of the general linear group and the special orthogonal group of rank 2 are supplied throughout, both from the algebraic and noncommutative differential geometric point of view.

\subsection*{Acknowledgments}
The authors are thankful to Andreas \v{C}ap, Rita Fioresi, Réamonn \'{O} Buachalla and Chiara Pagani for insightful discussions. This publication is based upon work from COST Action CaLISTA CA21109 supported by COST (European Cooperation in Science and Technology) www.cost.eu, as well as HORIZON-MSCA-2022-SE-01-01 CaLIGOLA, MSCA-DN CaLiForNIA - 101119552 and INFN Sezione Bologna. T.W. is supported by GA\v{C}R PIF 24-11324I.

\subsection*{Summary of the notation and conventions}
Throughout the whole paper we denote by $\Bbbk$ either the field of real or complex numbers, and all the Hopf algebras (in the standard or braided sense) shall be understood to have a bijective antipode. In view of the multiple algebraic structures to be featured therefrom, we find it essential to collect, in the table below, a summary of the notation employed throughout the paper. We emphasize, however, that in certain specific instances departures from this notation nevertheless proves necessary; whenever this occurs, we shall make explicit the notation in use at that point. 

\begin{center}
\renewcommand{\arraystretch}{1.25}
\begin{tabular}{@{}ll@{}}

\textbf{Symbol} & \textbf{Description} \\

$\rH$, $\rG$
    & Lie groups \\

$H$
    & Hopf algebra in ${}_\Bbbk\mathrm{Vec}$ \\

$(\Delta,\varepsilon,S)$
    & Coproduct, antipode and counit of a Hopf algebra \\

$\uB$
    & Braided Hopf algebra \\

$(\underline{\Delta},\underline{\varepsilon},\underline{S})$
    & Coproduct, antipode and counit of a braided Hopf algebra \\

$B\#H$
    & Smash product algebra \\

$\uB\#H$
    & Radford--Majid biproduct  \\

$(\Delta_\#,\varepsilon_\#,S_\#)$
    & Coproduct, antipode and counit of a Radford--Majid biproduct \\

${}^H\mathrm{Mod}$, ${}_H\mathrm{Mod}$, $\mathrm{Mod}^H$, $\mathrm{Mod}_H$
    & Categories of left/right $H$-(co)modules \\

$\Delta_V$, ${}_V\Delta$
    & Right and left $H$-coactions on a $H$-comodule $V$ \\

$\underline{\Delta}_V$, ${}_V\underline{\Delta}$
    & Right and left $\uB$-coactions on a $\uB$-comodule $V$ \\
\end{tabular}
\end{center}
Throughout the paper we make extensive use of Sweedler's notation for coproducts and coactions. We distinguish between the cases in which the coalgebra (or bialgebra) lives in the category of vector spaces ${}_\Bbbk\mathrm{Vec}$ or in a general braided monoidal category $\underline{\mathcal{M}}$; the table below records the resulting conventions in the two settings.

\begin{center}
\renewcommand{\arraystretch}{1.25}
\begin{tabular}{@{}llc@{}}
\textbf{Symbol} & \textbf{Description} & \textbf{Category} \\
$\Delta_V(v)=v_{0}\otimes v_{1}$
    & Sweedler's notation for a right coaction
    & ${}_{\Bbbk}\mathrm{Vec}$ \\
${}_V\Delta(v)=v_{-1}\otimes v_{0}$
    & Sweedler's notation for a left coaction
    & ${}_{\Bbbk}\mathrm{Vec}$ \\
$\Delta(h)=h_{1}\otimes h_{2}$
    & Sweedler's notation for the coproduct
    & ${}_{\Bbbk}\mathrm{Vec}$ \\
$\underline{\Delta}_V(v)=v^{0}\otimes v^{1}$
    & Sweedler's notation for a right coaction
    & $\underline{\mathcal{M}}$ \\
${}_V\underline{\Delta}(v)=v^{-1}\otimes v^{0}$
    & Sweedler's notation for a left coaction
    & $\underline{\mathcal{M}}$ \\
$\underline{\Delta}(b)=b^{1}\otimes b^{2}$
    & Sweedler's notation for the coproduct
    & $\underline{\mathcal{M}}$ \\
\end{tabular}
\end{center}

 \section{Braided Hopf algebras and the Radford--Majid biproduct}
\label{section:braided_hopf_algebras}
In this Section, we review the basic algebraic notions underlying the constructions that appear throughout the paper.
In Section \ref{subsection:Hopf--Galois_extensions_and_smash_product_algebras}, we review the basic notions of quantum principal bundles, paying particular attention to the case of trivial bundles, which are algebraically described as trivial Hopf--Galois extensions.
In Section~\ref{subsection:braided_Hopf_algebras}, we study Hopf algebras in braided monoidal categories. After introducing the necessary framework, we present fundamental results, such as the Radford--Majid theorem, and discuss some explicit examples that are used in the sequel.
In Section \ref{subsection:braided_fundamental_theorem_of_Hopf_modules} we discuss the generalisation of the fundamental theorem of Hopf modules to the setting of braided monoidal categories. 

\subsection{Hopf--Galois extensions and smash product algebras}
\label{subsection:Hopf--Galois_extensions_and_smash_product_algebras}
The main references for this Section are \cite{DoiTak,Brz}. While here we direct our focus to the affine case, we remark that the construction we present can be generalised to incorporate a sheaf theoretic point of view that is more suitable when dealing with the quantisation of projective varieties and flag manifolds. The interested reader can consider, for example, \cite{afl,aflw}.

\medskip
Let $H$ be a Hopf algebra with bijective antipode, and $A$ a \emph{right $H$-comodule algebra},
that is, an algebra with a right $H$-coaction 
$\Delta_A: A \to  A\otimes H$, $\Delta_A(a)={a}_{0} \otimes {a}_{1}$, that
is also an algebra morphism. The space of \emph{coinvariant elements}
$$
B:=A^{\co H}:=\{a \in A ~|~ \Delta_A(a)=a \otimes 1 \}
$$
is a subalgebra of $A$.   
The algebra extension $B\subseteq A$ is called a Hopf--Galois extension
if the \emph{canonical map} 
\begin{equation}
\begin{aligned}\label{can-map}
\chi \colon A\otimes_{B} A&\to A \otimes H \, ,\\  a'\otimes_{B} a & \mapsto a'a_{0}\otimes a_{1}, 
\end{aligned}
\end{equation}
is bijective.

We say that $A$ is faithfully flat as a right $B$-module if the functor $A\otimes_B (\cdot)$  from the category of left $B$-modules to the category of vector spaces preservers and reflects exact sequences. By definition, a Hopf--Galois extension
is called faithfully flat if $A$ is faithfully flat as a right $B$-module. Faithful flatness is required for equivariant projectivity and the existence of a strong connection \cite{hajac_strong}, and it reflects the classical requirement in the definition of a principal bundle that the action of the Lie group on the total space is proper. Moreover, it ensures the categorical equivalence of $B$-modules with $H$-covariant $A$-modules \cite{Schneider}. For these reasons this requirement is included in the following definition.
\begin{definition}[{\cite[Definition 3.2]{delDLW}}]\label{def:QPB}
 A faithfully flat Hopf--Galois extension $B := A^{\mathrm{co}H} \subseteq A$ is called a \emph{quantum principal bundle (QPB)} or \emph{principal $H$-comodule algebra}. 
    \label{def:quantumprincipalbundledef}
\end{definition}
An extension $A^{\co H}\subseteq A$ is called 
\emph{trivial}  if it admits a convolution invertible $H$-comodule algebra map $j:H\to A$, known as cleaving map.
In this case, the inverse of the cleaving map is given by $ j \circ S$ with $S$ the antipode of $H$. As we see in the following, trivial extensions can be constructed by considering a \emph{left $H$-module algebra} $B$, i.e., a $\Bbbk$-algebra $B$, together with a left $H$-action 
\begin{align*}
\triangleright\colon  H \otimes B &\to B,\\
 h \otimes b \,\,&\mapsto h \triangleright b,
\end{align*}
such that $h\rhd(bb')=(h_1\rhd b)(h_2\rhd b')$ and $h\rhd 1=\varepsilon(h)1$ for all $h\in H$ and  $b,b'\in B$.
Then, the vector space $B\otimes H$ can be equipped with an algebra structure called the \emph{smash product}.

\begin{definition}
    Let $B$ be a left $H$-module algebra, with left $H$-action $\triangleright$. The product
  \begin{equation}
  \begin{aligned}
  \label{eq:smashproductalgebraproduct}
            \mu_{\#} \colon (B \otimes H) \otimes (B \otimes H) & \to B \otimes H,\\
            (b \# h) \otimes (b^{\prime} \# h^{\prime}) &\mapsto b(h_1 \triangleright b^{\prime}) \otimes h_2h^{\prime},
    \end{aligned}
  \end{equation}
structures $B\#H := (B \otimes H, \mu_{{\#}})$ as an associative unital algebra, with unit $1 \otimes 1$, called the \emph{smash product algebra}. 
\end{definition}

 We give $B\#H$ a right $H$-comodule algebra structure induced by the coproduct of $H$, i.e.,
\begin{equation}
    \Delta_{\#} := \mathrm{id}_B \otimes \Delta \colon B\#H \to B\#H \otimes H,
    \label{eq:naturalrightHcoactiononsmashproduct}
\end{equation}
with subalgebra of coinvariant elements $(B\#H)^{\mathrm{co}H} = B \otimes 1 \cong B$. Any smash product algebra $B\#H$ is a trivial extension with cleaving map

\begin{equation}\label{eq:smashproductalgebraimpliestrivialmap}
\begin{split}
    j \colon H &\to B\#H,\\
    h &\mapsto 1 \# h.
\end{split}    
\end{equation}

\noindent Conversely, given a trivial extension $B:=A^{\mathrm{co}H}\subseteq A$, one obtains a smash product algebra $B\#H$ with respect to the left $H$-action $h\rhd b:=j(h_1)\,b\,j(S(h_2))$ for all $h\in H$ and $b\in B$, giving an equivalence between smash product algebras and trivial Hopf--Galois extensions, see \cite{DoiTak}.

\subsection{Braided Hopf algebras}\label{subsection:braided_Hopf_algebras}

We assume that the reader is familiar with braided monoidal categories and refer to the textbooks \cite{KasselBook,MajidFoundation} for more details.
Throughout this section $(\underline{\mathcal{M}}, \otimes, I)$ denotes a monoidal category. By \emph{Mac Lane’s coherence theorem}, $(\underline{\mathcal{M}}, \otimes, I)$ is monoidally equivalent to a strict monoidal category and thus we are going to omit the associativity and unit constraints of $(\underline{\mathcal{M}}, \otimes, I)$. We further omit the monoidal unit $I$ whenever it is clear from the context. Recall that a \emph{braided monoidal category} $(\underline{\mathcal{M}}, \otimes,I, \sigma)$ is a monoidal category  $(\underline{\mathcal{M}},\otimes,I)$ together with a braiding $\sigma$, that is, a natural isomorphism with components
\begin{equation}\label{eq:braid}
    \begin{split}
        \sigma_{X,Y} \colon X \otimes Y &\xrightarrow[]{\cong} Y \otimes X,\\
        x \otimes y &\mapsto {}_{\alpha}y \otimes {}^{\alpha}x,
    \end{split}
\end{equation}
on any objects $X,Y \in \underline{\mathcal{M}}$, satisfying the hexagon relations. Above, ${}_\alpha y\otimes{}^\alpha x:=\sigma_{X,Y}(x\otimes y)$ is to be understood as a short notation, where $\alpha$ is the summation index of a finite sum. Omitting associators, as mentioned before, and adopting the previous short notation, the hexagon relations read    \begin{equation}\label{eq:hexagon_relations}
         \begin{aligned}
            {}_{\gamma}(z) \otimes {}^{\gamma}(x\otimes y)
            &=\sigma_{X\otimes Y,Z}((x\otimes y)\otimes z) 
            =(\sigma_{X,Z}\otimes\mathrm{id}_Y)(\mathrm{id}_X\otimes\sigma_{Y,Z})(x\otimes y\otimes z)
            ={}_{\alpha\beta}(z)\otimes {}^{\alpha}(x) \otimes {}^{\beta}(y),\\ 
            {}_{\gamma}(y\otimes z) \otimes {}^{\gamma}(x)
            &=\sigma_{X,Y\otimes Z}(x\otimes(y\otimes z))
            =(\mathrm{id}_Y\otimes\sigma_{X,Z})(\sigma_{X,Y}\otimes\mathrm{id}_Z)(x\otimes y\otimes z)
            ={}_{\alpha}(y) \otimes {}_{\beta}(z)\otimes {}^{\beta\alpha}(x)
         \end{aligned}
    \end{equation}
for all $x\in X$, $y\in Y$ and $z\in Z$, where $\beta$ is the summation index of the second braiding.
\begin{example}
The category ${}_H^H\mathcal{YD}$ of \emph{(left-left) Yetter--Drinfeld modules} over a Hopf algebra $H$ is the category with objects consisting of triples $(V,\triangleright, {}_V\Delta)$, where $(V,\triangleright \colon H \otimes V \to V)$ is a left $H$-module, $(V,{}_V\Delta \colon V \to H \otimes V)$ is a left $H$-comodule and the Yetter--Drinfeld compatibility
\[
(h \triangleright v)_{-1} \otimes (h \triangleright v)_0 = h_1v_{-1}S(h_3) \otimes h_2\triangleright v_0
\]
holds for any $h \in H$ and $v \in V$. The morphisms are left $H$-linear and left $H$-colinear maps. The category ${}_H^H\mathcal{YD}$ is monoidal with tensor product induced from the tensor product of vector spaces:
for any $V,W \in {}_H^H\mathcal{YD}$ the tensor product $V \otimes W$ becomes an object in ${}_H^H\mathcal{YD}$ via
    \[
    \begin{aligned}
        & h \triangleright_{\otimes} (v \otimes w) = h_1 \triangleright v \otimes h_2 \triangleright w, \;\; {}_\otimes\Delta(v \otimes w) = v_{-1}w_{-1} \otimes v_0 \otimes w_0. 
    \end{aligned}
    \]
The monoidal unit is the ground field $\Bbbk$ endowed with the trivial $H$-action and $H$-coaction. The monoidal category $({}_H^H\mathcal{YD},\otimes,\Bbbk)$ is braided with respect to the braiding given on objects $V,W \in {}_H^H\mathcal{YD}$ by
    \[
        \sigma_{V,W}(v \otimes w)
        = v_{-1} \triangleright w \otimes v_0.
    \]
for all $v,w \in V,W$.
\qed
\end{example}
Other examples of braided monoidal categories arise as the (co)representation categories of (co)quasitriangular bialgebras. We recall for the convenience of the reader that a bialgebra $(H,\Delta,\varepsilon,S)$ is called \emph{quasitriangular} if there is an invertible element $\mathcal{R}\in H\otimes H$, the \emph{$\mathcal{R}$-matrix}, such that $H$ is quasi-cocommutative, i.e.,
$$
\Delta^\mathrm{op}(\cdot)=\mathcal{R}\Delta(\cdot)\mathcal{R}^{-1}
$$
holds, and the hexagon relations
\begin{equation}
\begin{split}
    (\mathrm{id}\otimes\Delta)(\mathcal{R})
    &=\mathcal{R}_{13}\mathcal{R}_{12}\\
    (\Delta\otimes\mathrm{id})(\mathcal{R})
    &=\mathcal{R}_{13}\mathcal{R}_{23}
\end{split}
\end{equation}
are satisfied. If, in addition, $\mathcal{R}^{-1}=\mathcal{R}^\mathrm{op}$, we refer to $(H,\mathcal{R})$ as a \emph{triangular bialgebra}. On the other hand, we call a bialgebra $H$ \emph{coquasitriangular} if there is a convolution invertible map $\mathcal{R}\colon H\otimes H\to\Bbbk$ such that $H$ is quasi-commutative, i.e., the multiplication and the opposite multiplication $\mu,\mu^\mathrm{op}\colon H\otimes H\to H$ are related via
$$
\mu^\mathrm{op}=\mathcal{R}*\mu*\mathcal{R}^{-1}
$$
and the hexagon relations
\begin{equation}
\begin{split}
    \mathcal{R}\circ(\mathrm{id}\otimes \mu)
    &=\mathcal{R}_{13}*\mathcal{R}_{12}\\
    \mathcal{R}\circ(\mu\otimes\mathrm{id})
    &=\mathcal{R}_{13}*\mathcal{R}_{23}
\end{split}
\end{equation}
are satisfied. If, in addition, $\mathcal{R}^{-1}=\mathcal{R}^\mathrm{op}$, we call $(H,\mathcal{R})$ \emph{cotriangular}. Above, $*$ denotes the convolution product, e.g., the quasi-commutativity reads $gh=\mathcal{R}(h_1\otimes g_1)h_2g_2\mathcal{R}^{-1}(h_3\otimes g_3)$ on elements $h,g\in H$. For more information on (co)quasitriangular structures we refer to reader to \cite[Chapter 2]{MajidFoundation}.
\begin{example}
Let $H$ be a bialgebra and consider the category ${}_H\mathcal{M}$ of left $H$-modules with its monoidal structure defined on objects $M,N\in{}_H\mathcal{M}$ by
$$
h\cdot(m\otimes n)
:=\Delta(h)\cdot(m\otimes n)
=h_1\cdot m\otimes h_2\cdot n
$$
for all $h\in H$, $m\in M$, $n\in N$ and monoidal unit $\Bbbk$ endowed with the trivial $H$-action. Then, $({}_H\mathcal{M},\otimes,\Bbbk)$ is braided if and only if $H$ is quasitriangular, see \cite[Theorem 9.2.4 and paragraph thereafter]{MajidFoundation}. In this case, the braiding is given by
$$
\sigma^\mathcal{R}_{M,N}(m\otimes n)
=\mathcal{R}^\mathrm{op}\cdot(n\otimes m)
=\mathcal{R}^2\cdot n\otimes\mathcal{R}^1\cdot m
$$
for all $m\in M$, $n\in N$, where we used the short notation $\mathcal{R}=\mathcal{R}^1\otimes\mathcal{R}^2$. The category is symmetric if and only if $(H,\mathcal{R})$ is triangular.

Again, just assuming for the moment that $H$ is a bialgebra, the category $\mathcal{M}^H$ of right $H$-comodules is monoidal, with right $H$-coaction defined on the tensor product of two objects $M,N\in\mathcal{M}^H$ by
$$
\Delta_\otimes\colon M\otimes N\to M\otimes N\otimes H,\qquad
\Delta_\otimes(m\otimes n):=m_0\otimes n_0\otimes m_1n_1
$$
and with monoidal unit given by the ground field $\Bbbk$, endowed with the trivial right $H$-coaction. This category is braided if and only if $H$ is coquasitriangular, where we obtain a symmetric category if and only if $H$ is cotriangular, see \cite[Exercise 9.2.9]{MajidFoundation}. In these cases, the braiding reads
$$
\sigma^\mathcal{R}_{M,N}(m\otimes n)=n_0\otimes m_0\mathcal{R}(m_1\otimes n_1)
$$
for all $m\in M$ and $n\in N$. \qed
\end{example}

Let $(A,\underline{\mu},\underline{\eta})$ be a \emph{monoid} in a monoidal category $(\underline{\mathcal{M}},\otimes,I)$, i.e., $A$ is an object in $\underline{\mathcal{M}}$ and $\underline{\mu}\colon A\otimes A\to A$, $\underline{\eta}\colon I\to A$ are morphisms in $\underline{\mathcal{M}}$, satisfying the obvious associativity and unit relations. We informally refer to $(A,\underline{\mu},\underline{\eta})$ as an \emph{associative unital algebra object in $\underline{\mathcal{M}}$}. If $(\underline{\mathcal{M}},\otimes,I,\sigma)$ is braided, the tensor product $A \otimes A$ becomes an associative unital algebra in $\underline{\mathcal{M}}$ with respect to the unit $\underline{\eta}_\otimes=\underline{\eta}\otimes\underline{\eta}\colon I\to A\otimes A$ and the multiplication
\[
    (a \otimes b)(c \otimes d) = a{}_{\alpha}c \otimes {}^{\alpha}bd
\]
for all $a,b,c,d \in A$, where we employed the short notation introduced in \eqref{eq:braid}.

In duality, a \emph{comonoid} in a monoidal category $(\underline{\mathcal{M}},\otimes,I)$ is an object $C$ in $\underline{\mathcal{M}}$, together with morphisms $\underline{\Delta}\colon C\to C\otimes C$ and $\underline{\varepsilon}\colon C\to I$ in $\underline{\mathcal{M}}$ satisfying the obvious coassociativity and counit relations. We informally refer to $(C,\underline{\Delta},\underline{\varepsilon})$ as a \emph{coassociative counital coalgebra object in $\underline{\mathcal{M}}$}. If $(\underline{\mathcal{M}},\otimes,I,\sigma)$ is braided, the tensor product $C \otimes C$ becomes a coassociative counital coalgebra in $\underline{\mathcal{M}}$ with respect to the counit $\underline{\varepsilon}_\otimes=\underline{\varepsilon}\otimes\underline{\varepsilon}\colon C\otimes C\to I$ and the comultiplication
\[
\underline{\Delta}_\otimes(c\otimes c') = (c^1\otimes{}_{\alpha}(c'^1))\otimes({}^{\alpha}(c^2)\otimes c'^2),
\]
for all $c,c' \in C$, where we adopted the Sweedler notation $\underline{\Delta}(c)=:c^1\otimes c^2$ with upper indexes for the coproduct of an element $c\in C$, as well as the short notation \eqref{eq:braid} for the braiding.
\begin{definition}[{\cite{MajidLectureNotes}}]
A \emph{braided bialgebra} (or \emph{bimonoid}) in a braided monoidal category $(\underline{\mathcal{M}}, \otimes, I,\sigma)$ is a $5$-tuple $(\uB,\underline{\mu},\underline{\eta},\underline{\Delta},\underline{\varepsilon})$ with $(\uB,\underline{\mu},\underline{\eta})$ an algebra in $\underline{\mathcal{M}}$ and $(\uB,\underline{\Delta},\underline{\varepsilon})$ a coalgebra in $\underline{\mathcal{M}}$ such that $\underline{\Delta} \colon \uB \otimes \uB \to \uB$ and $\underline{\varepsilon} \colon \uB \to I$ are algebra maps. A braided bialgebra bialgebra $\uB$ is called a \emph{braided Hopf algebra} (or \emph{Hopf monoid}) in the category $\underline{\mathcal{M}}$ if there exists a morphism $\underline{S}\colon\underline{\mathcal{M}}\to\underline{\mathcal{M}}$, the \emph{antipode}, such that
\begin{equation*}
\begin{tikzcd}
    & \underline{B}\otimes\underline{B} \arrow{rr}{\underline{S}\otimes\mathrm{id}_{\underline{B}}} 
    & & \underline{B}\otimes\underline{B} \arrow{dr}{\underline{\mu}} & \\
    \underline{B} \arrow{ru}{\underline{\Delta}} \arrow{dr}[swap]{\underline{\Delta}} \arrow{rr}{\underline{\varepsilon}}
    & & I \arrow{rr}{\underline{\eta}}
    & & \underline{B}\\
    & \underline{B}\otimes\underline{B} \arrow{rr}{\mathrm{id}_{\underline{B}}\otimes\underline{S}}
    & & \underline{B}\otimes\underline{B} \arrow{ur}[swap]{\underline{\mu}} &
\end{tikzcd}
\end{equation*}
holds as equalities of morphisms in $\underline{\mathcal{M}}$. A morphism of braided Hopf algebras is a morphism of the underlying monoids and comonoids.
\end{definition}
Explicitly, the coproduct of a braided bialgebra is an algebra morphism via
$$
\underline{\Delta}(bb')
=b^1{}_{\alpha}(b'^1)\otimes{}^{\alpha}(b^2)b'^2
$$
for all $b,b'\in\underline{B}$. Moreover, the antipode is an anti-bialgebra morphism, i.e., it satisfies
\[
    \underline{S}(ab) = \underline{S}({}_{\alpha}b)\,\underline{S}({}^{\alpha}a) = {}_{\alpha}\underline{S}(b){}^{\alpha}\underline{S}(a), \;\;\;
    (\underline{S}(a))^{1}\otimes (\underline{S}(a))^{2} = {}_{\alpha}\underline{S}(a^{2})\otimes {}^{\alpha}\underline{S}(a^{1}) = \underline{S}({}_{\alpha}(a^2)) \otimes \underline{S}({}^{\alpha}(a^1)),
\]
for all $a,b\in \uB$ (see for example \cite{Andr-Grana}).

In the next lemma we highlight some properties following directly from the hexagon relations and from the fact that the braiding is a natural isomorphism. These equalities are useful in the rest of the paper.
\begin{lemma}
Let $(\uB,\underline{\mu},\underline{\eta},\underline{\Delta},\underline{\varepsilon}, \underline{S})$ be a braided Hopf algebra in a braided monoidal category $(\underline{\mathcal{M}}, \otimes, I, \sigma)$. Then it follows, by naturality of the braiding, that
    \[
        \begin{split}
            {}_{\alpha}(bc) \otimes {}^{\alpha}a 
            &={}_{\beta}b{}_{\gamma}c \otimes {}^{\gamma}({}^{\beta}a),\\
            {}_{\alpha}c \otimes {}^{\alpha}(ab) 
            &= {}_{\gamma}({}_{\beta}c) \otimes {}^{\gamma}a{}^{\beta}b,\\
            {}_{\alpha}(b^1 \otimes b^2) \otimes {}^{\alpha}a 
            &= ({}_{\alpha}b)^1 \otimes ({}_{\alpha}b)^2 \otimes {}^{\alpha}a,\\
            {}_{\alpha}b \otimes {}^{\alpha}(a^1 \otimes a^2) 
            &= {}_{\alpha}b \otimes ({}^{\alpha}a)^1 \otimes ({}^{\alpha}a)^2, \\
            {}_{\alpha}(\underline{S}(b)) \otimes {}^{\alpha}a &=\underline{S}({}_{\alpha}b) \otimes {}^{\alpha}a,\\
            {}_{\alpha}b \otimes {}^{\alpha}(\underline{S}(a)) 
            &= {}_{\alpha}b \otimes (\underline{S}({}^{\alpha}a)),
        \end{split}
    \]
    for all $a,b,c \in \uB$.
\end{lemma}

\medskip

In the following example, we describe the quantum plane as a braided Hopf algebra in the category of Yetter--Drinfeld modules of the $q$-deformed coordinate algebra of the complex general linear group of rank two. Constructions built upon this example occur throughout the paper.

\begin{example}
\label{ex:C_q(2)_braided}
   Fix a complex number $q\in\mathbb{C}$, where $q\neq 0$, $q^{n}\neq 1$ for all $n\in \mathbb{Z}^{+}$. Consider the quantum plane $\underline{\mathbb{C}}_q^2$, namely, the algebra 
    \begin{equation}
        \underline{\mathbb{C}}_q^2:=\mathbb{C}[x,y]/\langle xy-qyx\rangle,
        \label{eq:famigeratoquantumplane}
    \end{equation}
    together with $\mathcal{O}_q(\mathrm{GL}_{2})$, the Hopf algebra generated by $a,b,c,d$ and the formal inverse $D^{-1}$ of a central element $D = ad-qbc$,
    where $a,b,c,d$ satisfy the following relations:
    \[
        \begin{aligned}
            & ab = qba, \;\; ac=qca, \;\; bd = qdb, \;\; cd=qdc, \\
            &bc=cb, \;\; ad-da=(q-q^{-1})bc.
        \end{aligned}
    \]
    The coproduct, counit and antipode for $\mathcal{O}_q(\mathrm{GL}_2)$ are determined, on generators, as
   \begin{equation}\label{eq:GLq2operations}
    \begin{aligned}
        \Delta\colon\ 
        \begin{pmatrix} a & b \\ c & d \end{pmatrix}
        &\mapsto
        \begin{pmatrix} a & b \\ c & d \end{pmatrix}
        \dot{\otimes}
        \begin{pmatrix} a & b \\ c & d \end{pmatrix},
        &\qquad
        D^{-1} &\mapsto D^{-1}\otimes D^{-1}, \\[2mm]
        \varepsilon\colon\ 
        \begin{pmatrix} a & b \\ c & d \end{pmatrix}
        &\mapsto
        \begin{pmatrix} 1 & 0 \\ 0 & 1 \end{pmatrix},
        &\qquad
        D^{-1} &\mapsto 1, \\[2mm]
        S\colon\ 
        \begin{pmatrix} a & b \\ c & d \end{pmatrix}
        &\mapsto
        D^{-1}
        \begin{pmatrix} d & -q^{-1}b \\ -qc & a \end{pmatrix},
        &\qquad
        D^{-1} &\mapsto D.
    \end{aligned}
\end{equation}
    There exist a
    left $\mathcal{O}_q(\mathrm{GL}_2)$-action (see \cite{Hajac_plane})
    \begin{equation*}
        \begin{aligned}
            a\triangleright x & = q^{-2}x, & b\triangleright x &= 0, & c\triangleright x &= (q^{-2} -1)y, & d \triangleright x &= q^{-1}x, & D^{-1} \triangleright x &= q^3x,\\
            a \triangleright y &= q^{-1}y, & b \triangleright y&= 0, & c\triangleright y &=0, & d \triangleright y &= q^{-2}y,& D^{-1}\triangleright y & = q^3y,
        \end{aligned}
    \end{equation*}
    and a left $\mathcal{O}_q(\mathrm{GL_2})$-coaction 
    \[
    {}_{\underline{\mathbb{C}}^2_q}\Delta (x)=a\otimes x +b\otimes y, \quad {}_{\underline{\mathbb{C}}^2_q}\Delta (y)=c\otimes x +d\otimes y,
    \]
    on $\underline{\mathbb{C}}_q^2$, which, as can be readily checked, are Yetter--Drinfeld complatible. The Yetter--Drinfeld braiding, on the generators, reads 

    \begin{equation}
    \label{eq:braiding_quantum_plane}
        \begin{aligned}
            \sigma(x\otimes x)&=q^{-2}x\otimes x, &  \sigma(y\otimes y)&=q^{-2}y\otimes y,\\  \sigma(x\otimes y) &= q^{-1}y\otimes x, &  \sigma(y\otimes x) & =(q^{-2}-1)y\otimes x + q^{-1}x\otimes y. 
        \end{aligned}
    \end{equation}
    Moreover, $\underline{\mathbb{C}}_q^2$ is a braided Hopf algebra in ${}^{\mathcal{O}_q(\mathrm{GL}_{2})}_{\mathcal{O}_q(\mathrm{GL}_{2})}\mathcal{YD}$, with coproduct, counit and antipode given by
    \[
        \underline{\Delta}(z) = z \otimes 1 + 1 \otimes z, \quad \underline{\varepsilon}(z) = 0, \quad  \underline{S}(z) = -z,\,
    \]
    with $z \in \{x,y\}$ (see \cite[Example 10.2.2]{MajidFoundation} and \cite[Proposition 4.12]{AzizMajid}). We call $\underline{\mathbb{C}}_q^2$ with such braided Hopf algebra structure the \emph{braided quantum plane}.
    \qed 
    \end{example}

The next example provides an explicit description of algebra and Hopf algebra objects internal to the category of cochain complexes, together with their corresponding truncations, in a braided monoidal category.

\begin{example}\label{ex:cochain}
    Fix a braided monoidal category $(\underline{\mathcal{M}},\otimes,I,\sigma)$ and consider the category $\underline{\mathcal{M}}^\bullet$ of \emph{cochain complexes in $\underline{\mathcal{M}}$}. Explicitly, objects in $\underline{\mathcal{M}}^\bullet$ are direct sums $M^\bullet:=\bigoplus_{n\geq 0}M^n$, with $M^n\in\underline{\mathcal{M}}$, endowed with morphisms $\mathrm{d}_n\colon M^n\to M^{n+1}$ in $\underline{\mathcal{M}}$ such that $\mathrm{d}_{n+1}\circ\mathrm{d}_n=0$ for all $n\geq 0$. We often omit the index of $\mathrm{d}_n$. Morphisms $\Phi\colon M^\bullet\to N^\bullet$ in $\underline{\mathcal{M}}^\bullet$ are maps of degree $0$, i.e., $\Phi(M^n)\subseteq N^n$ for all $n\geq 0$, such that $\Phi|_{M^n}\colon M^n\to N^n$ is a morphism in $\underline{\mathcal{M}}$ for all $n\geq 0$ and such that $\Phi\circ\mathrm{d}_M=\mathrm{d}_N\circ\Phi$. Note that $\underline{\mathcal{M}}^\bullet$ is braided monoidal with respect to the tensor product and braiding induced from $\underline{\mathcal{M}}$:
    
    \begin{enumerate}
    \item[$\bullet$] its monoidal structure is $M^\bullet\otimes N^\bullet:=\bigoplus_{n\geq 0}\bigoplus_{k+\ell=n}(M^k\otimes N^\ell)$, where the latter is endowed with the differential
    \begin{equation}
    \begin{split}
        \mathrm{d}_\otimes\colon M^\bullet\otimes N^\bullet&\to M^\bullet\otimes N^\bullet,\\
        \omega\otimes\eta&\mapsto\mathrm{d}(\omega)\otimes\eta
        +(-1)^k\omega\otimes\mathrm{d}(\eta),
    \end{split}
    \end{equation}
    for homogeneous elements $\omega\in M^k$.
    
    \item[$\bullet$] the monoidal unit is given by $I$, viewed as a cochain complex concentrated in degree $0$ and endowed with the zero differential.
    
    \item[$\bullet$] its braiding, denoted by abuse of notation by the same symbol, is determined on homogeneous elements $\omega\in M^k$, $\eta\in N^\ell$ by
    \begin{equation}\label{eq:grad-braid}
    \begin{split}
        \sigma_{M^\bullet,N^\bullet}\colon M^\bullet\otimes N^\bullet&\to N^\bullet\otimes M^\bullet,\\
        \omega\otimes\eta&\mapsto(-1)^{k\cdot\ell}\sigma_{M,N}(\omega\otimes\eta).
    \end{split}
    \end{equation}
    \end{enumerate}
    Algebra objects in $\underline{\mathcal{M}}^\bullet$ are also referred to as \emph{differential graded algebras in $\underline{\mathcal{M}}$}. Similarly, we refer to bialgebra objects and Hopf algebra objects in $\underline{\mathcal{M}}^\bullet$ as \emph{differential graded bialgbras/Hopf algebras in $\underline{\mathcal{M}}$}. In particular, the structure maps of such objects have to be compatible with the degree and the differentials of the cochain complexes. For example, for a differential graded Hopf algebra $\underline{B}^\bullet$ in $\underline{\mathcal{M}}$ the coproduct and antipode make the diagrams
    \begin{equation}
    \begin{tikzcd}
        \underline{B}^\bullet\otimes\underline{B}^\bullet \arrow{r}{\mathrm{d}_\otimes} & \underline{B}^\bullet\otimes\underline{B}^\bullet\\
        \underline{B}^\bullet \arrow{u}{\underline{\Delta}^\bullet} \arrow{r}{\mathrm{d}} & \underline{B}^\bullet \arrow{u}[swap]{\underline{\Delta}^\bullet}
    \end{tikzcd}\qquad,\qquad
    \begin{tikzcd}
        \underline{B}^\bullet \arrow{r}{\underline{S}^\bullet} & \underline{B}^\bullet\\
        \underline{B}^\bullet \arrow{r}{\underline{S}^\bullet} \arrow{u}{\mathrm{d}} & \underline{B}^\bullet \arrow{u}[swap]{\mathrm{d}}
    \end{tikzcd}
    \end{equation}
    commute. We recall that $\underline{B}^\bullet\otimes\underline{B}^\bullet$ acquires its algebra structure because $\underline{\mathcal{M}}^\bullet$ is itself braided monoidal and moreover the tensor product of two algebra objects in a braided monoidal category is again an algebra object in a canonical way. As a final remark, we stress that for our case of interest the differential graded Hopf algebras arising from bicovariant calculi are \emph{cutoff}, meaning concentrated in finitely many degrees: there exists $N\geq 0$ such that $M^{n}=0$ for all $n>N$. We denote by $\underline{\mathcal{M}}^{\bullet\leq N}\subset\underline{\mathcal{M}}^\bullet$ the full subcategory of such cutoff complexes. It is important to note that $\underline{\mathcal{M}}^{\bullet\leq N}$ is \emph{not} monoidal with respect to the tensor product $\otimes$ of $\underline{\mathcal{M}}^\bullet$: if $M^\bullet,N^\bullet$ are cutoff at degree $N$, their tensor product $M^\bullet\otimes N^\bullet$, as defined above, is in general concentrated in degrees up to $2N$, since $(M^\bullet\otimes N^\bullet)^n=\bigoplus_{k+\ell=n}M^k\otimes N^\ell$ and non-zero summands can occur for any $n\leq 2N$. For this reason we endow $\underline{\mathcal{M}}^{\bullet\leq N}$ with the monoidal structure $M^{\bullet\leq N}\otimes N^{\bullet\leq N}:=\bigoplus_{0\leq n\leq N}\bigoplus_{k+\ell=n}M^k\otimes N^\ell$. Note that, in case $\underline{\mathcal{M}}$ is braided monoidal, the braiding in \eqref{eq:grad-braid} (co)restricts to a braiding on $\underline{\mathcal{M}}^{\bullet\leq N}$. Algebra, bialgebra and Hopf algebra objects in $\underline{\mathcal{M}}^{\bullet\leq N}$ are defined as in the case of $\underline{\mathcal{M}}^\bullet$. We refer to \cite{SchauenburgDG} for more information on cutoff differential graded Hopf algebra objects.
    \qed
\end{example}

A class of braided Hopf algebras is induced from (co)quasitriangular Hopf algebras via the so-called \emph{transmutation}, cf. \cite{Majid1991,MajidLectureNotes}.

\begin{proposition}{\cite{MajidBraid}}
\label{prop:transmutations} Consider a Hopf algebra $(H,\mu,\eta,\Delta,\varepsilon,S) \in {}_{\Bbbk}\mathrm{Vec}$.
\begin{enumerate}
\item  Let $H_{1}$ be a quasitriangular Hopf algebra with $\mathcal{R}$-matrix $\mathcal{R}=\mathcal{R}^1\otimes\mathcal{R}^2$, and let $f\colon H_{1}\to H$ be a Hopf algebra map. There is a braided Hopf algebra $\underline{H}=(H,\mu,\eta,\underline{\Delta},\varepsilon,\underline{S})$ in the category $({}_{H_{1}}\mathrm{Mod},\otimes_{\Bbbk},\sigma^\mathcal{R})$, which coincides with $H$ as an algebra and admits the same counit, but which is endowed with the comultiplication and antipode
\begin{equation}
    \underline{\Delta}(h):=h_1S(f(\mathcal{R}^2))\otimes\mathcal{R}^1\rhd h_2,\qquad
    \underline{S}(h):=f(\mathcal{R}^2)S(\mathcal{R}^1\rhd h),
\end{equation}
for all $h\in H$, where $H$ is viewed as an object in ${}_{H_{1}}\mathrm{Mod}$ via the action $\mathrm{Ad}_{\mathrm{L},f}:=k\rhd_{f} h:=f(k_1)hf(S(k_2))$.

\item Let $H_{2}$ be a coquasitriangular Hopf algebra with $\mathcal{R}$-form $\mathcal{R}\colon H\otimes H\to\Bbbk$, and let $g\colon H\to H_{2}$ be a Hopf algebra map. There is a braided Hopf algebra $\underline{H}=(H,\underline{\mu},\eta,\Delta,\varepsilon,\underline{S})$ in the category $(\mathrm{Mod}^{H_{2}},\otimes_{\Bbbk},\sigma^\mathcal{R})$, which coincides with $H$ as a coalgebra and which admits the same unit, but which is endowed with the multiplication and antipode
\begin{equation}
    \underline{\mu}(h\otimes k):=h_2k_2\mathcal{R}(g(S(h_1)h_3)\otimes g(S(k_1))),\qquad
    \underline{S}(h):=S(h_2)\mathcal{R}(g(S^2(h_3)S(h_1))\otimes g(h_4))
\end{equation}
for all $h,k\in H$. Above, $H$ is viewed as an object in $\mathrm{Mod}^{H_{2}}$ via the coaction $\mathrm{coAd}_{\mathrm{R},g}\colon H\to H\otimes H_{2}$, $h\mapsto h_2\otimes g(S(h_1)h_3)$.
\end{enumerate}
\end{proposition}
As a special case of the above proposition, we may consider the identity itself as the Hopf algebra map $ H\to H$ responsible for the transmutation of a (co)quasitriangular Hopf algebra $H$ into its own category of (co)modules with respect to the left adjoint $H$-action (right adjoint $H$-coaction). A relevant example of such construction, that we now illustrate for the quasitriangular setting, is the transmutation of Sweedler's Hopf algebra $E_{1}$ for a selected choice of $\mathcal{R}$-matrix (see also \cite{MajidFoundation}). 

\begin{example}
\label{ex:sweedler_transmutation}
 Let us consider the $4$-dimesional Sweedler's Hopf algebra 
\[
E_{1}:=\langle x,g,1\rangle \big/ 
\big\{x^{2},g^{2}-1,xg+gx \big\},
\]
with coproduct, counit and antipode given by 

\begin{equation}
    \begin{aligned}
        \Delta(x)& =x\otimes 1 + g\otimes x, &  \Delta(g)& =g\otimes g,\\ 
        \varepsilon(x)&= 0, & \varepsilon(g)& = 1, \\ 
        S(x)&=-gx, & S(g)&=g.
    \end{aligned}
\end{equation}

\noindent For $\lambda$ an arbitrary parameter, $E_{1}$ is a quasitriangular Hopf algebra \cite{MajidFoundation} with $\mathcal{R}$-matrix 
\begin{equation}
\label{eq:Rmatrix_Sweedler}
\mathcal{R}_{\lambda} = \frac{1}{2}\left( 1\otimes 1 + 1\otimes g + g\otimes 1 - g\otimes g\right) +\frac{\lambda}{2}\left(x\otimes x + x\otimes gx + gx\otimes gx - gx\otimes x \right).
\end{equation}
 Thus, fixing for simplicity $\lambda=0$, we may consider the braided Hopf algebra $\underline{E_{1}}:=(E_{1},\mu,\eta,\underline{\Delta},\varepsilon,\underline{S})$ given by the transmutation of $E_{1}$ into its own braided monoidal category of modules $\left({}_{E_{1}}\mathrm{Mod}, \otimes,\sigma_{\mathcal{R}_{0}}\right)$ with left $E_{1}$-module action, coproduct and antipode reading 
\[
h\triangleright b=h_{1}\,b\,S(h_{2}),\quad \underline{\Delta}(b)=b_{1}\, S((\mathcal{R}_{0})^{2})\otimes (\mathcal{R}_{0})^{1}\triangleright b_{2},\quad \underline{S}(b)=(\mathcal{R}_{0})^{2}\, S((\mathcal{R}_{0})^{1}\triangleright b),
\]
and in particular 
\[
\begin{aligned}
    x\triangleright x & =0, & x\triangleright g&=2xg, &  g\triangleright x& = -x, & g\triangleright g & = g,\\ 
    \underline{\Delta}(x) & = x\otimes 1 +1\otimes x, & \underline{\Delta}(g)&=g\otimes g, & \underline{S}(x) & = -x, & \underline{S}(g)&=g.
\end{aligned}
\]
The braiding in ${}_{E_{1}}\mathrm{Mod}$, which is symmetric, is induced from the $\mathcal{R}$-matrix of $E_{1}$ as 
\begin{equation}
    \begin{aligned}
        \sigma_{\mathcal{R}_{0}} \colon \underline{E_{1}}\otimes \underline{E_{1}} & \to \underline{E_{1}}\otimes \underline{E_{1}},\\ 
        a\otimes b & \mapsto (\mathcal{R}_{0})^{2}\triangleright b \otimes (\mathcal{R}_{0})^{1}\triangleright a, 
    \end{aligned}
\end{equation} 
and on the generators it reads  
\[
\sigma_{\mathcal{R}_{0}}(x\otimes x)=-x\otimes x,\quad \sigma_{\mathcal{R}_{0}}(x\otimes g)= g\otimes x,\quad \sigma_{\mathcal{R}_{0}}(g\otimes g) = g\otimes g.
\]
\qed 
 \end{example}

Otherwise, from the datum of a braided Hopf algebra in the category of Yetter--Drinfeld modules ${}^{H}_{H}\mathcal{YD}$ of a Hopf algebra $H\in{}_{\Bbbk}\mathrm{Vec}$ it is possible to realise a Hopf algebra structure on the tensor product $B\otimes H$. This is the content of the Radford--Majid biproduct construction. 
\begin{theorem}[\cite{Radford,MajidFoundation}]
\label{thm:bosonisation}
        Let $H$ be a Hopf algebra and $\underline{B}$ be a braided Hopf algebra in ${}_H^H\mathcal{YD}$. Then we can build a Hopf algebra $A:=\underline{B}\#H$, called \emph{Bosonisation} or \emph{Radford's biproduct}, with multiplication given by the smash product algebra multiplication, that is,
        \begin{equation}
            (b\#h)(b'\#h') = b(h_{1} \triangleright b')\#h_{2}h',
            \label{eq:smashmultiplication}
        \end{equation}
        coproduct $\Delta_{\#} \colon A \to A \otimes A$, counit $\varepsilon_{\#} \colon A \to \Bbbk$ and antipode $S_{\#} \colon A \to A$ given in terms of the braided coproduct counit and antipode of $\uB$ and in terms of the coproduct, counit and antipode of $H$ as
        \begin{equation}
            \begin{aligned}
                &\Delta_{\#}(b\#h) = \big(b^{1}\#(b^{2})_{-1}h_{1}\big)\otimes\big((b^{2})_{0}\#h_{2}\big);\\
                &\varepsilon_{\#}(b\#h) = \underline{\varepsilon}(b)\varepsilon(h);\\
                &S_{\#}(b\#h) = \big(1\#S(b_{-1}h)\big)\big(\underline{S}(b_{0})\#1\big).
            \end{aligned}
            \label{eq:Hopfbosonisation}
        \end{equation}
\end{theorem}
\begin{example}
    Building on Example \ref{ex:C_q(2)_braided}, we provide a natural realisation of a Radford--Majid biproduct.  We define $\mathrm{Aff}_q(2):=\underline{\mathbb{C}}_{q}^{2}\# \mathcal{O}_{q}(\textrm{GL}_{2})$, whose Hopf algebra structure, according to Theorem \ref{thm:bosonisation}, is outlined below in matrix notation
    
    {\footnotesize $$
        \begin{array}{rcl}
        \Delta_{\#} \begin{pmatrix}
            1\#1 &0 &0\\
            x\#1 &1\#a &1\#b\\
            y\#1 &1\#c &1\#d
        \end{pmatrix} &=& 
        \begin{pmatrix}
            1\#1 &0 &0\\
            x\#1 &1\#a &1\#b\\
            y\#1 &1\#c &1\#d
        \end{pmatrix} \dot{\otimes} \begin{pmatrix}
            1\#1 &0 &0\\
            x\#1 &1\#a &1\#b\\
            y\#1 &1\#c &1\#d
        \end{pmatrix},\\[8mm]
       \varepsilon_{\#}\begin{pmatrix}
            1\#1 &0 &0\\
            x\#1 &1\#a &1\#b\\
            y\#1 &1\#c &1\#d
        \end{pmatrix} &=&\begin{pmatrix}
            1 &0 &0\\ 0 &1 &0 \\0 &0 &1
        \end{pmatrix},\\[8mm]
          S_{\#}\begin{pmatrix}
            1\#1 &0 &0\\
            x\#1 &1\#a &1\#b\\
            y\#1 &1\#c &1\#d
        \end{pmatrix}& =&(1\#D^{-1}) \begin{pmatrix}
            (1\# D)(1\#1) &0 &0\\ q^{-1}(1\#b)(y\#1) -(1\#d)(x\#1)  &1\#d &-q^{-1}(1\# b) \\  q(1\#c)(x\#1) - (1\#a)(y\# 1) & - q(1\#c) &1\#a
        \end{pmatrix}.
        \end{array}
    $$}

    \medskip 
    \noindent
    This construction provides a noncommutative generalisation of the affine extension of the $2$-dimensional plane. In fact, notice how the coproduct in Equation \ref{eq:Hopfbosonisation} can be written on the generators of $\mathrm{Aff}_q(2)$ in terms of matrix multiplication of the affine extension Lie group (see Equation \ref{eq:affineextension} of the appendix), as well as the antipode being given by the inverse matrix.
    \qed
\end{example}

\subsection{Braided fundamental theorem of Hopf modules}\label{subsection:braided_fundamental_theorem_of_Hopf_modules}
The aim of the following Section is to generalise the fundamental theorem of Hopf modules to the general setting of Hopf modules of a Hopf algebra in a braided monoidal category. The discussion is based on \cite{BespalovBraidedFundThm}. 

\medskip

Let $(\mathcal{C},\otimes,I)$ be a monoidal category, where we omit the associativity and unit constraints, as argued before. Fix an associative unital algebra $(A,\mu,\eta)$ in $(\mathcal{C},\otimes,I)$. Internal to $(\mathcal{C},\otimes,I)$ we define a \emph{left $A$-module} as an object $M$ in $\mathcal{C}$ endowed with a morphism $\cdot\colon A\otimes M\to M$ in $\mathcal{C}$, the \emph{left $A$-action}, such that $a\cdot(b\cdot m)=(ab)\cdot m$ and $1\cdot m=m$ for all $a,b\in A$ and $m\in M$. A \emph{morphism of left $A$-modules} $M,N$ internal to $(\mathcal{C},\otimes,I)$ is a morphism $\phi\colon M\to N$ in $\mathcal{C}$ such that $\phi(a\cdot m)=a\cdot\phi(m)$ for all $a\in A$ and $m\in M$. We frequently call $\phi$ a \emph{left $A$-linear map}. The category of left $A$-modules internal to $(\mathcal{C},\otimes,I)$ is denoted by ${}_A\mathcal{C}$. Similarly, one defines the categories $\mathcal{C}_A$ of right $A$-modules internal to $(\mathcal{C},\otimes,I)$ and ${}_A\mathcal{C}_A$ of $A$-bimodules internal to $(\mathcal{C},\otimes,I)$.

Fix a coassociative counital coalgebra $(C,\Delta,\varepsilon)$ in $(\mathcal{C},\otimes,I)$. Internal to $(\mathcal{C},\otimes,I)$ we define a \emph{right $C$-comodule} as an object $M$ in $\mathcal{C}$ endowed with a morphism $\underline{\Delta}_M\colon M\to M\otimes C$ in $\mathcal{C}$, the \emph{right $C$-coaction}, such that $(m^0)^0\otimes(m^0)^1\otimes m^1=m^0\otimes(m^1)^1\otimes(m^1)^2$ and $m^0\varepsilon(m^1)=m$ for all $m\in M$, where we employ the Sweedler-type notation
$$
\underline{\Delta}_M(m)=:m^0\otimes m^1,\qquad(\mathrm{id}\otimes\underline{\Delta}_M)(\underline{\Delta}_M(m))=:m^0\otimes m^1\otimes m^2
$$
for all $m\in M$ and similarly for higher iterations of the coaction. A \emph{morphism of right $C$-comodules} $M,N$ internal to $(\mathcal{C},\otimes,I)$ is a morphism $\psi\colon M\to N$ in $\mathcal{C}$ such that $\psi(m)^0\otimes\psi(m)^1=\psi(m^0)\otimes m^1$ for all $m\in M$. We frequently refer to $\psi$ as a \emph{right $C$-colinear map}. The category of right $C$-comodules internal to $(\mathcal{C},\otimes,I)$ is denoted by $\mathcal{C}^C$. Similarly, one defines the categories ${}^C\mathcal{C}$ of left $C$-comodules internal to $(\mathcal{C},\otimes,I)$ and ${}^C\mathcal{C}^C$ of $C$-bicomodules internal to $(\mathcal{C},\otimes,I)$. Left $C$-coactions of a left $C$-comodule $M$ are usually denoted by ${}_M\underline{\Delta}\colon M\to C\otimes M$ and we employ the Sweedler-like notation ${}_M  \underline{\Delta}(m)=:m^{-1}\otimes m^0$ for all $m\in M$.

From now on we assume, in addition, that $(\underline{\mathcal{M}},\otimes,I,\sigma)$ is braided monoidal and that $\underline{\mathcal{M}}$ be a preadditive category such that equalizers and coequalizers exist. If $(\underline{B},\underline{m},\underline{\eta},\underline{\Delta},\underline{\varepsilon})$ is a braided bialgebra in $(\underline{\mathcal{M}},\otimes,I,\sigma)$ and $(M,\Delta_M)$ a right $\underline{B}$-comodule in $(\underline{\mathcal{M}},\otimes,I)$, we denote by

\begin{equation}
    M^{\underline{B}}:=\{m\in M \mid \underline{\Delta}_M(m)=m\otimes1_{\underline{B}}\}
\end{equation}

\noindent the space of \emph{right $\underline{B}$-coinvariant elements} of $M$. Note that, by definition, $M^{\underline{B}}=\ker(\underline{\Delta}_M-\mathrm{id}\otimes\underline{\eta})$, i.e., $M^{\underline{B}}$ is the equalizer of the morphisms $\underline{\Delta}_M,\mathrm{id}\otimes\underline{\eta}\colon M\to M\otimes\underline{B}$ in $\underline{\mathcal{M}}$:
$$
\begin{tikzcd}
    M^{\underline{B}} \arrow{r}{\iota} & M \arrow[rr, shift left, "\underline{\Delta}_M"] \arrow[rr, shift right, "\mathrm{id}\otimes\underline{\eta}"'] & & M\otimes\underline{B}
\end{tikzcd}
$$
In the next definition we endow $M$ with an additional right $\underline{B}$-module structure, compatible with the right $\underline{B}$-coaction. 
\begin{definition}
Let $\underline{B}$ be a braided bialgebra in $\underline{\mathcal{M}}$. A \emph{right-right Hopf module} is a right $\underline{B}$-module $M$ endowed with a right $\underline{B}$-coaction $\underline{\Delta}_M\colon M\to M\otimes\underline{B}$ that is also a morphism of right $\underline{B}$-modules, where $M\otimes\underline{B}$ is structured as a right $\underline{B}$-module via $(m\otimes b)\cdot c:=m\cdot {}^{\alpha}(c^1)\otimes{}_\alpha(b)c^2$ for all $m\in M$ and $b,c\in\underline{B}$. Explicitly,
\begin{equation}
    (m\cdot b)^0\otimes(m\cdot b)^1
    =m^0\cdot{}^\alpha(b^1)\otimes{}_\alpha(m^1)b^2
\end{equation}
for all $m\in M$ and $b\in\underline{B}$.
A morphism of right-right Hopf modules is a right $\underline{B}$-linear and right $\underline{B}$-colinear map. The category of right-right Hopf modules is denoted by $\underline{\mathcal{M}}_{\underline{B}}^{\underline{B}}$.
\end{definition}
In case $\underline{B}$ is even a braided Hopf algebra in $\underline{\mathcal{M}}$ and $M$ a right-right Hopf module, the space $M^{\underline{B}}$ of right $\underline{B}$-coinvariant elements can also be understood as the coequalizer of the morphisms $\cdot,\mathrm{id}\otimes\underline{\varepsilon}\colon M\otimes\underline{B}\to$ in $\underline{\mathcal{M}}$:
$$
\begin{tikzcd}
    M\otimes\underline{B} \arrow[rr, shift left, "\cdot"] \arrow[rr, shift right, "\mathrm{id}\otimes\underline{\varepsilon}"'] & & M \arrow[r, "\mathrm{pr}"] & M^{\underline{B}}
\end{tikzcd}
$$
where $\mathrm{pr}$ is the corestriction of the idempotent
$$
\omega_M\colon M\to M,\qquad m\mapsto m^0\cdot \underline{S}(m^1).
$$
In fact, one proves that $M^{\underline{B}}\cong M/M\underline{B}^+$ as objects in $\underline{\mathcal{M}}$, see \cite[Proposition 3.2.1]{BespalovBraidedFundThm} for the left-left Hopf module versions of these statements. 
Then, we structure $M^{\underline{B}}\otimes\underline{B}$ as an object in $\underline{\mathcal{M}}^{\underline{B}}_{\underline{B}}$ via
\begin{equation}
\begin{split}
    (M^{\underline{B}}\otimes\underline{B})\otimes\underline{B}&\to M^{\underline{B}}\otimes\underline{B}\\
    (m\otimes b)\otimes b'&\mapsto m\otimes bb'
\end{split},\qquad
\begin{split}
    M^{\underline{B}}\otimes\underline{B}&\to(M^{\underline{B}}\otimes\underline{B})\otimes\underline{B}\\
    m\otimes b&\mapsto m\otimes\underline{\Delta}(b)
\end{split}.
\end{equation}
It turns out that, up to isomorphism, every object in $\underline{\mathcal{M}}^{\underline{B}}_{\underline{B}}$ is of this form. This is know as the \emph{fundamental theorem of Hopf modules}, which can be phrased as the following categorical equivalence, employing the following functors, for which we fix a braided Hopf algebra $\underline{B}$ in $\underline{\mathcal{M}}$:
\begin{enumerate}
\item[$\bullet$] $\Phi\colon\underline{\mathcal{M}}^{\underline{B}}_{\underline{B}}\to\underline{\mathcal{M}}$ is a functor defined on objects $M\in\underline{\mathcal{M}}^{\underline{B}}_{\underline{B}}$ by $\Phi(M):=M^{\underline{B}}$ and on morphisms $\phi\colon M\to N$ in $\underline{\mathcal{M}}^{\underline{B}}_{\underline{B}}$ by the restriction and corestriction $\Phi(\phi):=\phi|_{M^{\underline{B}}}\colon M^{\underline{B}}\to N^{\underline{B}}$.

\item[$\bullet$] $\Psi\colon\underline{\mathcal{M}}\to\underline{\mathcal{M}}^{\underline{B}}_{\underline{B}}$ is a functor defined on objects $V\in\underline{\mathcal{M}}$ by $\Psi(V):=V\otimes\underline{B}$, seen as an object in $\underline{\mathcal{M}}^{\underline{B}}_{\underline{B}}$ via the multiplication and comultiplication of $\underline{B}$, and on morphisms $\psi\colon V\to W$ in $\underline{\mathcal{M}}$ by $\Psi(\psi):=\psi\otimes\mathrm{id}\colon V\otimes\underline{B}\to W\otimes\underline{B}$.
\end{enumerate}
These functors are inverse to each other, which is the statement of the following result (compare with \cite[Theorem 3.5.2]{BespalovBraidedFundThm}, which describes the left-left Hopf module version).

\begin{theorem}
Let $(\underline{\mathcal{M}},\otimes,I,\sigma)$ be a braided monoidal and preadditive category such that equalizers and coequalizers exist and let $\underline{B}$ be a braided Hopf algebra in $\underline{\mathcal{M}}$. Then, there is an equivalence

\begin{equation}
\begin{tikzcd}
    \underline{\mathcal{M}}_{\underline{B}}^{\underline{B}} \arrow[rr, bend left, "\Phi"] & \cong & \underline{\mathcal{M}} \arrow[ll, bend left, "\Psi"]
\end{tikzcd}
\end{equation}
of categories.
\end{theorem}

\noindent A similar version of the fundamental theorem of Hopf modules can be proven for the category of left-left Hopf modules ${}_{\underline{B}}^{\underline{B}}\underline{\mathcal{M}}$ internal to $\underline{\mathcal{M}}$. In the following we are particularly interested in tetra-modules and covariant bimodules internal to a braided monoidal category.
\begin{definition}
Let $\underline{B}$ be a braided bialgebra in $\underline{\mathcal{M}}$. A \emph{tetra-module} is a $\underline{B}$-bimodule $M$ endowed with commuting right and left $\underline{B}$-coactions $\underline{\Delta}_M\colon M\to M\otimes\underline{B}$ and ${}_M\underline{\Delta}\colon M\to\underline{B}\otimes M$ such that ${}_M\underline{\Delta},\underline{\Delta}_M$ are $\underline{B}$-bilinear. Explicitly, the bilinearity reads
\begin{equation}
\begin{split}
    \underline{\Delta}_M(b\cdot m\cdot c)
    &=b^1\cdot{}^\alpha(m^0)\cdot{}^\beta(c^1)\otimes{}_\beta({}_\alpha(b^2)m^1)c^2\\
    {}_M\underline{\Delta}(b\cdot m\cdot c)
    &=b^1~{}^\alpha(m^{-1})~{}^\beta(c^1)\otimes{}_\beta({}_\alpha(b^2)\cdot m^0)\cdot c^2
\end{split}
\end{equation}
on elements $m\in M$ and $b,c\in\underline{B}$. Morphisms of tetra-modules are $\underline{B}$-bilinear and $\underline{B}$-bicolinear maps. We write ${}_{\underline{B}}^{\underline{B}}\underline{\mathcal{M}}_{\underline{B}}^{\underline{B}}$ for the category of tetra-modules internal to $\underline{\mathcal{M}}$. When only requiring the existence of one of the compatible coactions we obtain the categories ${}_{\underline{B}}^{\underline{B}}\underline{\mathcal{M}}_{\underline{B}}$ and ${}_{\underline{B}}\underline{\mathcal{M}}_{\underline{B}}^{\underline{B}}$, respectively. 
\end{definition}
Assume that $\underline{B}$ is a braided Hopf algebra in $\underline{\mathcal{M}}$ such that its antipode $\underline{S}\colon\underline{B}\to\underline{B}$ is an isomorphism in $\underline{\mathcal{M}}$. Then, the category ${}_{\underline{B}}^{\underline{B}}\underline{\mathcal{M}}_{\underline{B}}^{\underline{B}}$ is braided monoidal with respect to the tensor product $M\otimes_{\underline{B}}N$ of tetra-modules $M,N$, where $M\otimes_{\underline{B}}N$ is endowed with the left $\underline{B}$-action on the first tensor factor and the right $\underline{B}$-action on the second tensor factor and with the $\underline{B}$-coactions

\begin{equation}
\begin{split}
    M\otimes_{\underline{B}}N&\to(M\otimes_{\underline{B}}N)\otimes\underline{B}\\
    m\otimes_{\underline{B}}n&\mapsto m^0\otimes_B{}^\alpha(n^0)\otimes{}_\alpha(m^1)n^1
\end{split}\qquad\qquad
\begin{split}
    M\otimes_{\underline{B}}N&\to\underline{B}\otimes(M\otimes_{\underline{B}}N)\\
    m\otimes_{\underline{B}}n&\mapsto m^{-1}~{}^\alpha(n^{-1})\otimes{}_\alpha(m^0)\otimes_{\underline{B}}n^0
\end{split}
\end{equation}

\noindent and with respect to the braiding

\begin{equation}
\begin{split}
    \sigma_{M,N}\colon M\otimes_{\underline{B}}N&\to N\otimes_{\underline{B}}M\\
    m\otimes_{\underline{B}}n&\mapsto m^{-2}\cdot{}^\alpha(n^0\cdot \underline{S}(n^1))\otimes_{\underline{B}}{}_\alpha(\underline{S}(m^{-1})\cdot m^0)\cdot n^2
\end{split}
\end{equation}

\noindent where we refer to \cite[Theorem 4.3.1]{BespalovBraidedFundThm}.

From the braided Hopf algebra $\underline{B}$ with invertible antipode we can build another braided monoidal category, namely, the category ${}_{\underline{B}}^{\underline{B}}\mathcal{YD}$ of \emph{left-left Yetter--Drinfel'd modules internal to $\underline{\mathcal{M}}$}. An object in ${}_{\underline{B}}^{\underline{B}}\mathcal{YD}$ is a left $\underline{B}$-module and left $\underline{B}$-comodule $M$ satisfying the Yetter--Drinfel'd compatibility

\begin{equation}
    b^1~{}^\alpha(m^{-1})\otimes{}_\alpha(b^2)\cdot m^0
    =(b^1\cdot{}^\alpha m)^{-1}~{}^\beta{}_\alpha(b^2)\otimes{}_\beta((b^1\cdot{}^\alpha m)^0),
\end{equation}

\noindent for all $m\in M$ and $b\in\underline{B}$, while a morphism in ${}_{\underline{B}}^{\underline{B}}\mathcal{YD}$ is a morphism in $\underline{\mathcal{M}}$ which is also left $\underline{B}$-linear and left $\underline{B}$-colinear, in addition. The category ${}_{\underline{B}}^{\underline{B}}\mathcal{YD}$ is monoidal with respect to the tensor product induced from $\underline{\mathcal{M}}$. Explicitly, $V\otimes W\in{}_{\underline{B}}^{\underline{B}}\mathcal{YD}$ for $V,W\in{}_{\underline{B}}^{\underline{B}}\mathcal{YD}$ via the $\underline{B}$-(co)action

\begin{equation}
\begin{split}
    b\cdot(v\otimes w)&
    :=b^1\cdot {}^\alpha v\otimes{}_\alpha(b^2)\cdot w\\
    (v\otimes w)^{-1}\otimes(v\otimes w)^0
    &:=v^{-1}~{}^\alpha(w^{-1})\otimes{}_\alpha(v^0)\otimes w^0,
\end{split}
\end{equation}

\noindent for all $b\in\underline{B}$, $v\in V$ and $w\in W$. The braiding of ${}_{\underline{B}}^{\underline{B}}\mathcal{YD}$ is determined by

\begin{equation}
\begin{split}
    \sigma_{V,W}\colon V\otimes W&\to W\otimes V\\
    v\otimes w&\mapsto {}^\alpha(v^{-1})\cdot{}_\alpha w\otimes v^0
\end{split}
\end{equation}

\noindent and we refer to \cite{BespalovCrossed} for proofs and more information about Yetter--Drinfel'd modules internal to braided monoidal categories.

According to the fundamental theorem of Hopf modules, see \cite[Theorem 4.3.2]{BespalovBraidedFundThm}, the categories of tetra-modules and Yetter--Drinfel'd modules internal to a braided monoidal category are equivalent as braided monoidal categories.

\begin{theorem}
Let $(\underline{\mathcal{M}},\otimes,I,\sigma)$ be a braided monoidal and preadditive category such that equalizers and coequalizers exist and let $\underline{B}$ be a braided Hopf algebra with invertible antipode in $\underline{\mathcal{M}}$. Then, there is an equivalence

\begin{equation}
\begin{tikzcd}
    ({}_{\underline{B}}^{\underline{B}}\underline{\mathcal{M}}_{\underline{B}}^{\underline{B}},\otimes_{\underline{B}},\underline{B},\sigma^\mathcal{W}) \arrow[rr, bend left, "\Phi"] & \cong & ({}_{\underline{B}}^{\underline{B}}\mathcal{YD},\otimes,I,\sigma^{\mathcal{YD}}) \arrow[ll, bend left, "\Psi"]
\end{tikzcd}
\end{equation}

\noindent of braided monoidal categories.
\end{theorem}

\section{Braided first order differential calculi} 
\label{First order braided differential calculi}
In this Section we generalise the notion of differential calculus from $\Bbbk$-vector spaces to general braided monoidal categories. We shall always assume the category to be abelian, so as to retain most of the fundamental results of the classical theory. Unless otherwise stated, we denote such a category by $\underline{\mathcal{M}}:=(\mathcal{M},\otimes,\sigma)$.
In Section \ref{section:definition_examples_and_constructions} we introduce the main definitions to be used for the remainder of the section. In particular, we give the notion of first order braided covariant differential calculus over a bialgebra in a braided monoidal category, and we show the existence of such a calculus having a universal property.  
In Section \ref{section:classification_theorem} we discuss how first order braided differential calculi on a braided Hopf algebra are classified by (braided) ideals in the kernel of the counit. The result is conceptually  analogous to the celebrated Woronowicz's classification theorem, originally appearing in \cite{Woronowicz1989} for the case of Hopf algebras in the category of vector spaces. In Section \ref{section:examples_of_braided_calculi} we discuss some examples of first order braided differential calculi on braided Hopf algebras, notably on the Sweedler's $E_{n}$ Hopf algebras and on the braided group realised as the transmutation of $\mathcal{O}_{q}(\mathrm{SL}_{2})$.

\subsection{Definition, examples and constructions}
\label{section:definition_examples_and_constructions}
We start recalling the notion of a first order differential calculus for an algebra in an abelian monoidal category $(\mathcal{C},\otimes)$. 

\begin{definition}\label{def:braided_fodci}
    Let $B$ be an algebra in $(\mathcal{C},\otimes)$.  A \emph{first order differential calculus} (FODC) over $B$ is the datum of: 
    \begin{enumerate}
        \item a $B$-bimodule $\Omega^{1}(B) \in \mathcal{C}$, 
        \item a derivation $\left(\,\mathrm{d}\colon B \to \Omega^{1}(B)\,\right)\in\mathrm{Hom}_{\mathcal{C}}(B,\Omega^{1}(B))$, i.e., a morphism $\mathrm{d}\colon B \to \Omega^{1}(B)$ in $\mathcal{C}$ satisfying the Leibniz rule $\mathrm{d}(bb')=\mathrm{d}(b)b'+b\mathrm{d}(b')$ for all $b,b'\in B$, 
        \item such that the morphism in $\mathcal{C}$
        $$
        B\otimes B\to\Omega^1(B),\qquad b\otimes b'\mapsto b\mathrm{d}(b')
        $$
        is surjective.
    \end{enumerate}
 A morphism $\varphi\colon \Omega^{1}(B)\to \widetilde\Omega^{1}(B)$ of FODCi over $B$ is a $B$-bimodule map such that the diagram 
 
\begin{equation}\label{eq:morph_of_fobdci}
\begin{tikzcd}
\Omega^{1}(B) \arrow[rr, "\varphi"] &                                                                    & \widetilde{\Omega}^{1}(B) \\
                                    &                                                                    &                           \\
                                    & B \arrow[luu, "\mathrm{d}"] \arrow[ruu, "\widetilde{\mathrm{d}}"'] &                          
\end{tikzcd}
\end{equation}
commutes. We denote the category of FODCi over an algebra $B\in\mathcal{C}$ by ${}_{\mathcal{C}}\mathrm{Diff}(B)$.
\end{definition}

\noindent One verifies that a morphism of FODCi is automatically unique and surjective, see \cite[Proposition 3.2.6]{Keegan}. 

\begin{proposition}[{\cite[Proposition 3.2.3]{Keegan}}]
\label{prop:braided_universal_is_initial}
	Let $B$	be an algebra in $(\mathcal{C},\otimes)$ with multiplication $\mu\colon B\otimes B\to B$. Consider the object $\Omega^{1}_{u}(B):=\ker\mu$ and the morphism 
    \begin{equation}
    \begin{aligned}
	    \mathrm{d}_{u}\colon B &\to \Omega^{1}_{u}(B), \\ 
        b & \mapsto 1\otimes b - b\otimes 1\,
    \end{aligned}
	\end{equation}
    in $\mathcal{C}$. Then $(\Omega^{1}_{u}(B),\mathrm{d}_{u})$ is a FODC over $B$ in $\mathcal{C}$ and it follows that $(\Omega^{1}_{u}(B),\mathrm{d}_{u})$ is the initial object in ${}_{\mathcal{C}}\mathrm{Diff}(B)$. In particular, for any FODC  $(\Omega^{1}(B),\mathrm{d})\in {}_{\mathcal{C}}\mathrm{Diff}(B)$ there exists a (necessarily unique and surjective) morphism of FODCi $\Pi\colon \Omega^{1}_{u}(B)\to\Omega^{1}(B)$ such that $\Pi\circ \mathrm{d}_{u}=\mathrm{d}$, and thus $\Omega^{1}(B)\cong \Omega^{1}_{u}(B)\big/ \ker\Pi$ in ${}_{\mathcal{C}}\mathrm{Diff}(B)$.
\begin{proof}
  It is an easy check that the pair $(\Omega^{1}_{u}(B),\mathrm{d}_{u})$ is an object in ${}_{\mathcal{C}}\mathrm{Diff}(B)$, and in fact the proof of the latter in essentially the same as in ${}_{\Bbbk}\mathrm{Vec}$. In particular, the $B$-bimodule sturcture on $\Omega^{1}_{u}(B)$ is given as left and right multiplication on the tensor product, $\mathrm{d}_{u}$ is clearly a derivation and the surjectivity condition holds.  
    Let now $(\Omega^{1}(B), \mathrm{d})$ be any object in ${}_{\mathcal{C}}\mathrm{Diff}(B)$, and consider the map defined as 
    
\begin{equation} 
\begin{aligned}
\Pi\colon \Omega^{1}_{u}(B)& \to \Omega^{1}(B) \\ 
b\,\mathrm{d}_{u}c  & \mapsto b\,\mathrm{d}c.
\end{aligned}
\end{equation}

\noindent This map is manifestly surjective, and moreover it is morphism in $\mathrm{Hom}_{{}_{\mathcal{C}}\mathrm{Diff}(B)}(\Omega^{1}_{u}(B),\Omega^{1}(B))$, since it is compatible with $\mathrm{d}_{u}$ by construction and a morphism of $B$-bimodules in $\mathcal{M}$. Thus, the isomorphism $\Omega^{1}(B)\cong \Omega^{1}_{u}(B)\big/\ker\Pi$ follows. Morphism of FODCi are unique, as previously commented.
\end{proof}
\end{proposition}

\noindent We call the object $(\Omega^{1}_u(B),\mathrm{d}_{u})$ the \emph{universal first order differential calculus} on the algebra $B\in\mathcal{C}$.

We now discuss two canonical constructions featuring FODCi.

\begin{definition}
\label{def:quotient_and_pullback_calculi}
Let $A,A_0$ and $A'$ be algebras in $(\mathcal{C},\otimes)$.
\begin{itemize} 
  \item Let $\phi\colon A\to A_{0}$ be an algebra epimorphism in $\mathrm{Hom}_{\mathcal{C}}(A,A_{0})$, and let $(\Omega^{1}(A),\mathrm{d})\in{}_{\mathcal{C}}\mathrm{Diff}(A)$. For $I:=\ker\phi$, we construct a FODC in $ {}_{\mathcal{C}}\mathrm{Diff}{(A_{0})}$, the \emph{quotient calculus}, as the object $(\Omega^{1}(A_0):=\Omega^{1}(A)/N,\mathrm{d}_0)$, where $N$ is the $A$-subbimodule generated by elements $I\,\mathrm{d}A + A\,\mathrm{d}I$, and $\mathrm{d}'\phi(a)=[\mathrm{d}a]$. 
\item Consider an algebra map $\iota\colon A' \rightarrow A\in \mathrm{Hom}_{\mathcal{C}}(A',A)$. Given a FODC $(\Omega^{1}(A),\mathrm{d})\in {}_{\mathcal{C}}\mathrm{Diff}(A)$, we canonically get a \emph{pullback calculus} 
$(\Omega^1(A'),\mathrm{d}')\in {}_{\mathcal{C}}\mathrm{Diff}(A')$ via 

\begin{equation}\label{pullback}
\Omega^1(A'):=\iota(A')\,\mathrm{d}(\iota(A'))\subseteq\Omega^1(A),
\end{equation}

\noindent where $\mathrm{d}':=\mathrm{d}\circ \iota \colon
A'\to\Omega^1(A')$. In the following we omit the $\iota$ map when dealing with the pullback calculus whenever it is clear from the context.
\end{itemize}
\end{definition}

The notion of a first order differential calculus admits a natural generalisation to that of a (higher order) differential calculus on an algebra in $(\mathcal{C},\otimes)$. Although the paper deals almost exclusively with first order differential calculi, we nevertheless find it useful to record here an explicit definition, particularly in view of some of the examples discussed below.

\begin{definition}
\label{def:higher_order_dc}
Let $B$ be an algebra in $(\mathcal{C},\otimes)$, and let $(\Omega^{1}(B),\mathrm{d})$ be a first order differential calculus over $B$ in the sense of Definition \ref{def:braided_fodci}. A \emph{differential calculus} (or \emph{higher order differential calculus}) extending $(\Omega^{1}(B),\mathrm{d})$ is the datum of:
\begin{enumerate}
    \item a graded algebra $(\Omega^{\bullet}(B)=\bigoplus_{n\geq 0}\Omega^{n}(B),\wedge)$ in $(\mathcal{C},\otimes)$, with $\Omega^{0}(B)=B$ and with multiplication morphisms
    \[
    \begin{aligned}
    \Omega^{n}(B)\otimes\Omega^{m}(B)& \to\Omega^{n+m}(B), \\ 
    \omega \otimes \gamma & \mapsto \omega \wedge \gamma,
    \end{aligned}
    \]
    restricting on $\Omega^{0}(B)\otimes\Omega^{0}(B)$ to the product of $B$ and on $\Omega^{0}(B)\otimes\Omega^{1}(B)$, respectively $\Omega^{1}(B)\otimes\Omega^{0}(B)$, to the left, respectively right, $B$-module structure of $\Omega^{1}(B)$,
    \item a degree-one morphism $\left(\,\mathrm{d}\colon\Omega^{n}(B)\to\Omega^{n+1}(B)\,\right)_{n\geq 0}\in\mathrm{Hom}_{\mathcal{C}}(\Omega^{\bullet}(B),\Omega^{\bullet+1}(B))$, restricting on $\Omega^{0}(B)=B$ to the derivation $\mathrm{d}\colon B\to\Omega^{1}(B)$ of Definition \ref{def:braided_fodci}, and satisfying
        \[
        \mathrm{d}(\omega\wedge\omega')=\mathrm{d}(\omega)\wedge \omega'+(-1)^{n}\,\omega\,\wedge \mathrm{d}(\omega'), \qquad \mathrm{d}^{2}=0,
        \]
    for all $\omega\in\Omega^n(B)$ and $\omega'\in\Omega^\bullet(B)$,
    \item such that, for every $n\geq 2$, the morphism in $(\mathcal{C},\otimes)$
    \[
    \begin{aligned}
    \Omega^{1}(B)^{\otimes n}& \to\Omega^{n}(B),\\ \omega_{1}\otimes\cdots\otimes\omega_{n}&\mapsto\omega_{1}\wedge\cdots \wedge \omega_{n},
    \end{aligned}
    \]
    is surjective, i.e., $\Omega^{\bullet}(B)$ is generated, as an algebra, by $\Omega^{0}(B)$ and $\Omega^{1}(B)$.
\end{enumerate}
A morphism $\Phi\colon\Omega^{\bullet}(B)\to\widetilde\Omega^{\bullet}(B)$ of differential calculi over $B$ is a morphism of graded algebras in $(\mathcal{C},\otimes)$, of degree zero, such that the diagram

\begin{equation}\label{eq:morph_of_higher_dc}
\begin{tikzcd}
\Omega^{\bullet}(B) \arrow[rr, "\Phi"] &  & \widetilde{\Omega}^{\bullet}(B) \\
                                        &  &                                  \\
                                        & B \arrow[luu, "\mathrm{d}"] \arrow[ruu, "\widetilde{\mathrm{d}}"'] &
\end{tikzcd}
\end{equation}

\noindent commutes.
\end{definition}

\noindent Notice that it is standard to define higher order differential calculi $(\Omega^{\bullet}(B),\wedge,\mathrm{d})$ as differential graded algebras that are generated in degree zero $\Omega^{0}(B)=B$. Here we choose to adopt a more explicit, yet equivalent, definition in terms of $(\Omega^{1}(B),\mathrm{d})$, as FODCi are the main object of discussion of the paper. 

We now introduce the notion of covariant first order differential calculus in a braided monoidal category.

\begin{definition}
Let $\underline{\mathcal{M}}$ be an abelian braided monoidal category. 
\begin{enumerate}
    \item Let $\uB$ be a bialgebra in $\underline{\mathcal{M}}$. A first order differential calculus $(\Omega^1(\uB),\ud)\in {}_{\uM}\mathrm{Diff}(\uB)$ in $\underline{\mathcal{M}}$ is called \emph{braided left covariant} (respectively \emph{braided right covariant}) if $\Omega^1(\uB)$ is a left (respectively right) $\uB$-covariant $\uB$-bimodule and the differential $\ud \colon \uB \to \Omega^{1}(\uB)$ is left (respectively right) $\uB$-colinear. We call $(\Omega^1(\uB),\ud )$ \emph{braided bicovariant} if it is both braided left and braided right covariant.
    \item Let $\uB$ be a left (right) $\underline{H}$-comodule algebra of a bialgebra  $\underline{H}\in\uM$. We say that $\Omega^{1}(\uB)$ is a braided left (right) $\underline{H}$-\emph{covariant} FODC on $\uB$ if $\Omega^{1}(\uB)$ is a left (right) $\underline{H}$-covariant $B$-bimodule and the differential is left (right) $\underline{H}$-colinear, and we say $\Omega^{1}(\uB)$ is $\underline{H}$-bicovariant if it is both a left and right $\underline{H}$-covariant $\uB$-bimodule and the differential is $\underline{H}$-bicolinear.
\end{enumerate}
\end{definition}

We denote the categories of braided left/right/bi-covariant FODCi on $\uB\in \underline{\mathcal{M}}$ by 

\begin{equation}
\label{eq:categories_of_braided_calculi}
{}^{\uB}_{\underline{\mathcal{M}}}\mathrm{Diff}(\uB),\quad  {}_{\underline{\mathcal{M}}}\mathrm{Diff}(\uB)^{\uB},\quad \text{and} \quad {}^{\uB}_{\underline{\mathcal{M}}}\mathrm{Diff}(\uB)^{\uB}.
\end{equation}

We sometimes colloquially refer to braided covariant FODCi as \emph{braided differential calculi}. The notion of braided covariant higher order differential calculus can be defined in a similar fashion, though this is not the main focus of this paper.

The universal FODC $(\Omega^{1}(\uB),\ud _{u})$ over a braided bialgebra $\uB$ in $\underline{\mathcal{M}}$ is again initial in the categories of braided covariant calculi. This is the content of the next proposition, which we state and prove at the level of a braided bialgebra, and which specialises to braided Hopf algebras (see also {\cite[Proposition 3.5.8]{Keegan}}). 

\begin{proposition}	Let $\uB$ be a bialgebra in $\underline{\mathcal{M}}$. The universal first order differential calculus over $\uB$ is a braided bicovariant FODC and it is an initial object in all the categories displayed in \eqref{eq:categories_of_braided_calculi}.
\proof  We prove the statement showing that there are maps 

	\begin{equation}
			\begin{aligned}
				{}_{\Omega^{1}_{u}(\uB)}\underline{\Delta}\colon \Omega^{1}_{u}(\uB)&\to \uB\otimes \Omega^{1}_{u}(\uB), \\ 
				b\otimes c &\mapsto b^{1}\sigma(b^{2}\otimes c^{1})\otimes c^{2}
		\end{aligned}
	\end{equation} 
    
    \begin{equation}\label{eq:right_coaction_on_universal_braided}
        \begin{aligned}
				\underline{\Delta}_{\Omega^{1}_{u}(\uB)}\colon \Omega^{1}_{u}(\uB)&\to \Omega^{1}_{u}(\uB)\otimes\uB, \\ 
				b\otimes c &\mapsto b^{1}\otimes \sigma(b^{2}\otimes c^{1}) c^{2}.
        \end{aligned}
    \end{equation} 
    
\noindent that are respectively well-defined left and right $\uB$-coactions on $\Omega^{1}_{u}(\uB)$ extending the coproduct $\underline{\Delta}\colon \uB\to \uB\otimes \uB$. Consider the following commutative diagram 

    \begin{equation*}
         \begin{tikzcd}
            \uB\otimes \uB \arrow[rr, "{}_{\Omega^{1}_{u}(\uB)}\uDelta"] \arrow[dd, "\underline{\mu}"] &  & \uB\otimes \uB\otimes \uB \arrow[dd, "(\mathrm{id}\otimes \underline{\mu})"] \\
                                                                                           &  &                                                            \\
         \uB  \arrow[rr, "\underline{\Delta}"]                                                       &  & \uB\otimes \uB                                                
        \end{tikzcd}.
    \end{equation*}
    
\noindent We have $(\mathrm{id}\otimes \underline{\mu})\,\circ \,{}_{\Omega^{1}_{u}(\uB)}\uDelta (\ker\underline{\mu}) \subseteq \uB\otimes \ker\underline{\mu} = \uB\otimes \Omega^{1}_{u}(\uB)$. Next, we explicitly show the coassociativity properties. First of all 

    \begin{equation}
        \begin{aligned}
            \label{equation:coaction_on_omega_1}
            (\mathrm{id}\otimes {}_{\Omega^{1}_{u}(\uB)}\uDelta )\circ {}_{\Omega^{1}_{u}(\uB)}\uDelta (b\otimes c) & = (\mathrm{id}\otimes {}_{\Omega^{1}_{u}(\uB)}\uDelta )(b^{1}\, \sigma(b^{2}\otimes c^{1})\otimes c^{2})\\ 
             & =  (\mathrm{id}\otimes {}_{\Omega^{1}_{u}(\uB)}\uDelta )(b^{1} {}_{\alpha}c^{1}\otimes {}^{\alpha}b^{2}\otimes c^{2}) \\ 
             & = b^{1} {}_{\alpha}c^{1}\otimes ({}^{\alpha}b^{2})^{1} {}_{\beta}((c^{2})^{1})\otimes {}^{\beta}(({}^{\alpha}b^{2})^{2})\otimes c^{2},  
        \end{aligned}
    \end{equation}
    
    whereas 
    
    \begin{equation}
    \begin{aligned}
    \label{equation:coaction_on_omega_2}
    (\underline{\Delta}\otimes \mathrm{id})\circ {}_{\Omega^{1}_{u}(\uB)}\uDelta (b\otimes c) & = (\underline{\Delta}\otimes \mathrm{id})(b^{1}{}_{\alpha}c^{1}\otimes {}^{\alpha}b^{2}\otimes c^{2})\\ 
    & = (b^{1})^{2}\sigma((b^{1})^{2}\otimes ({}_{\alpha}c^{1})^{1})({}_{\alpha}c^{1})^{2}\otimes {}^{\alpha}b^{2}\otimes c^{2} \\ 
    & = (b^{1})^{1} {}_{\beta}(({}_{\alpha}c^{1})^{1})\otimes {}^{\beta}((b^{1})^{2}) ({}_{\alpha}c^{1})^{2}\otimes {}^{\alpha}b^{2}\otimes c^{2}.
    \end{aligned}
    \end{equation}
    
\noindent As coassociativity of the coproduct $\underline{\Delta}$ on $B$ reads 
    \begin{align*}
        (\underline{\Delta}\otimes\mathrm{id})\circ \underline{\Delta}(b\otimes c) & = (\underline{\Delta}\otimes \mathrm{id})(b^{1}\otimes {}_{\alpha}(c^{1})\otimes {}^{\alpha}(b^{2})\otimes c^{2}) \\ 
        & = \underline{\Delta}(b^{1}\otimes {}_{\alpha}(c^{1})) \otimes {}^{\alpha}(b^{2})\otimes c^{2}\\ 
        & = (b^{1})^{1} \otimes {}_{\beta}(({}_{\alpha}c^{1})^{1})\otimes {}^{\beta}((b^{1})^{2})\otimes ({}_{\alpha}c^{1})^{2}\otimes {}^{\alpha}(b^{2})\otimes c^{2}\\ 
        & = b^{1}\otimes {}_{\alpha}c^{1}\otimes ({}^{\alpha}(b^{2}))^{1}\otimes {}_{\beta}((c^{2})^{1})\otimes {}^{\beta}(({}^{\alpha}(b^{2}))^{2}\otimes (c^{2})^{2}\\ 
        & = b^{1}\otimes {}_{\alpha}c^{1}\otimes \underline{\Delta}({}^{\alpha}(b^{2})\otimes c^{2}) \\ 
        & = (\mathrm{id}\otimes \underline{\Delta})\circ \underline{\Delta}(b\otimes c), 
    \end{align*}
    
\noindent we deduce that the expressions in Equations \eqref{equation:coaction_on_omega_1} and  \eqref{equation:coaction_on_omega_2} coincide. Moreover, the differential $\ud _{u}\colon  \uB\to \Omega^{1}_{u}(\uB)$ is left $\uB$-colinear with respect to $\underline{\Delta}$ and ${}_{\Omega^{1}_{u}(\uB)}\uDelta $.
Finally, we check that $\Omega^{1}_{u}(\uB)$ is a (left-left) Hopf module in $\underline{\mathcal{M}}$. 
In fact 

\begin{align*}
{}_{\Omega^{1}_{u}(\uB)}\uDelta (b\, \omega) &= {}_{\Omega^{1}_{u}(\uB)}\uDelta (b\, c \otimes e)  \\
  &= (bc)^{1}\, {}_{\alpha} e^{1} \otimes {}^{\alpha}((bc)^{2}) \otimes e^{2}  \\
   \quad  &= (b^{1}\, {}_{\beta} c^{1})\, {}_{\alpha} e^{1} \otimes {}^{\alpha}({}^{\beta}(b^{2})\, c^{2}) \otimes e^{2} 
 \\
 \quad  &= (b^{1}\, {}_{\beta}(c^{1}))\, {}_{\alpha\gamma}(e^{1}) \otimes {}^{\alpha\beta}(b^{2})\, {}^{\gamma}(c^{2}) \otimes e^{2} 
    \\ 
 \quad  & = b^{1} {}_{\beta}(c^{1}) {}_{\gamma\alpha}(e^{1}) \otimes {}^{\gamma\beta} b^{2} {}^{\alpha}(c^{2}) \otimes e^{2}  \\  
 \quad  & = b^{1} {}_{\beta}(c^{1}) {}_{\alpha}(e^{1}) \otimes {}^{\beta}(b^{2}) {}^{\alpha}(c^{2}) \otimes e^{2}  \\ 
  & = \underline{\Delta}(b) \, {}_{\Omega^{1}_{u}(\uB)}\uDelta (c\otimes e) \\ 
  & = \underline{\Delta}(b) \, {}_{\Omega^{1}_{u}(\uB)}\uDelta (\omega),
\end{align*}

\noindent where we used that the coproduct on two elements $b,c\in B$ reads 
$\underline{\Delta}(bc)= b^{1} {}_{\alpha}c^{1}\otimes {}^{\alpha} b^{2}\otimes c^{2}$, alongside the hexagon relations in Equation \eqref{eq:hexagon_relations}.

Finally, let $\Omega^1(\uB)$ be any braided  left covariant first order differential calculus on $\uB$. From Proposition \ref{prop:braided_universal_is_initial}, we have an isomorphism $\Omega^{1}_{u}(\uB)/\ker\Pi\cong \Omega^1(\uB)$ of $\uB$-bimodules. We define a map ${}_{\Omega^{1}(\uB)}\uDelta\colon \Omega^{1}(\uB)\to \uB\otimes \Omega^{1}(\uB)$ in terms of the commutativity of the following diagram

\begin{equation}\label{equation:diagram_definition_of_gamma_Delta}
    \begin{tikzcd}
        \Omega^{1}_{u}(\uB)/\ker\Pi  \arrow[rr, "{}_{\Omega^{1}_{u}(\uB)}\uDelta  "]        &  & \uB\otimes (\Omega^{1}_{u}(\uB)/\ker\Pi) \arrow[dd, "\mathrm{id}\otimes \Pi"] \\
                                                                   &  &                                                                           \\
    \Omega^1(\uB)  \arrow[uu, "\Pi^{-1}"'] \arrow[rr, "{}_{\Omega^{1}(\uB)}\uDelta", dashed] &  & \uB\otimes \Omega^1(\uB)                                                          
        \end{tikzcd},
    \end{equation}
namely 
    \[
        {}_{\Omega^{1}(\uB)}\uDelta(b\,\ud c):= b^{1}\sigma(b^{2}\otimes c^{1})\ud c^{2}:=b^{1} \, {}_{\alpha}c^{1}\otimes {}^{\alpha}b^{2}\, \ud c^{2},
    \] 
for any element $b\, \ud c\in \Omega^1(\uB)$. One can check that this assignment gives a well-defined map. 

With a completely analogous reasoning we have that the assignment in Equation \eqref{eq:right_coaction_on_universal_braided} provides a well-defined right $\uB$-coaction on $\Omega^{1}_{u}(\uB)$ that correctly descend to any braided right $\uB$-covariant FODC on $\uB$ via $\Pi$, which therefore is a morphism in ${}_{\uM}\mathrm{Diff}(\uB)$. \qed 
\end{proposition}

As a direct consequence we get the following corollary. 

\begin{corollary}\label{cor:liftofcoproductcovariance}
    Let $\uB$ be a braided  bialgebra in a braided monoidal category $\underline{\mathcal{M}}$. A FODC $(\Omega^1(\uB),\ud ) \in {}_{\underline{\mathcal{M}}}\mathrm{Diff}(\uB)$ is braided left covariant if and only if
    
    \begin{equation}
    \label{eq:left_coaction_braided_calculus}
        {}_{\Omega^{1}(\uB)}\uDelta (b\ud b') = b^1{}_\alpha b'^1 \otimes {}^{\alpha}b^2\ud  b'^2
    \end{equation}
    
    \noindent is well-defined. Similarly, $\Omega^1(\uB)$ is braided right covariant if and only if
    
    \begin{equation}
    \label{eq:right_coaction_braided_calculus}
        \uDelta_{\Omega^{1}(\uB)}(b\ud b') = b^1\ud {}_\alpha b'^1 \otimes {}^\alpha b^2b'^2
    \end{equation}
    
   \noindent is well-defined. The FODC is braided bicovariant if and only if the above coactions are well-defined.
\end{corollary}

Later on, we make use of the following lemma, which links together the notions of braided Hopf algebras and braided bicovariant first order differential calculi, also employing the monoidal structure of cutoff categories of cochain complexes (see Example \ref{ex:cochain}).

\begin{lemma}\label{lem:gradHopf}
    Let $\uB$ be a braided Hopf algebra in $\underline{\mathcal{M}}$, and let $(\Omega^1(\uB),\ud) \in {}_{\underline{\mathcal{M}}}\mathrm{Diff}(\uB)$. Then, $(\Omega^1(\uB),\ud)$ is braided bicovariant if and only if $\uB^\bullet := \uB \oplus \Omega^1(\uB) \in \underline{\mathcal{M}}^{\bullet \leq1}$ is a differential graded Hopf algebra in $\underline{\mathcal{M}}$, that is, a braided Hopf algebra in $\underline{\mathcal{M}}^{\bullet\leq 1}$. In particular, it has multiplication given by
    \begin{equation}\label{eq:mu_bullet}
        \begin{aligned}
        \underline{\mu}^\bullet \colon \uB^\bullet \otimes \uB^\bullet &\to \uB^\bullet,\\
            b \otimes b' + \omega \otimes b' + b \otimes \omega' &\mapsto bb' + (b\cdot \omega +\omega'\cdot b'),
        \end{aligned}
    \end{equation}
    unit given by the unit of $\uB$ seen in $\uB^\bullet$, comultiplication
    \begin{equation}\label{eq:delta_bullet}
        \begin{split}
            \underline{\Delta}^\bullet \colon \uB^\bullet &\to \uB^\bullet \otimes \uB^\bullet,\\
            b + \omega &\mapsto b^1 \otimes b^2 + (\omega^{-1} \otimes \omega^0 + \omega^0 \otimes \omega^1),
        \end{split}
    \end{equation}
    and counit given by counit of $\uB$ on elements of degree $0$ and vanishing on elements of degree $1$, and  antipode 
    \begin{equation}\label{S_bullet}
        \begin{split}
            \underline{S}^\bullet \colon \uB^\bullet &\to \uB^\bullet,\\
            b + \omega &\mapsto \underline{S}(b) + (-\underline{S}(\omega^{-1})\cdot \omega^0 \cdot \underline{S}(\omega^1)).
        \end{split}
    \end{equation}
 \label{lem:BbulletDGHA}
\end{lemma}
\begin{proof}
    In one direction, given $(\Omega^1(\uB),\ud) \in {}_{\underline{\mathcal{M}}}^{\uB}\mathrm{Diff}(\uB)^{\uB}$, to show that $\uB^\bullet$ is a graded Hopf algebra, one can follow the arguments used to prove the same statement in the category of vector spaces (\cite[Lemma 5.8 \& Corollary 6.3]{SchauenburgDG}). Clearly, the graded Hopf algebra structure maps are morphisms in $\underline{\mathcal{M}}$, since they are written as compositions of morphisms in $\underline{\mathcal{M}}$.
    In addition, $\underline{\mu}^\bullet$ being a morphism in $\underline{\mathcal{M}}^{\bullet\leq 1}$ is equivalent to $\ud$ being a derivation, and 
    braided bicovariance of $(\Omega^1(\uB),\ud)$ is equivalent to $\underline{\Delta}^\bullet$ being a morphism in $\underline{\mathcal{M}}^{\bullet\leq 1}$. Finally, one quickly verifies that the equality $\underline{S}^1 \circ \ud = \ud \circ \underline{S}$  holds, and thus $\underline{S}^\bullet$ is a morphism in $\underline{\mathcal{M}}^{\bullet \leq 1}$ as well. 
    
    Conversely, given a Hopf algebra in $\underline{\mathcal{M}}^{\bullet\leq 1}$, its degree-$1$ component defines a braided bicovariant FODC over its degree-$0$ component.
\end{proof}

Notice that for higher order differential calculi the above correspondence no longer holds: in general, a differential graded Hopf algebra gives rise to a braided bicovariant calculus, but not every braided bicovariant calculus arises in this way (see for example \cite{delDLW} for the analogous result in the category of vector spaces).

\subsection{A classification theorem for braided differential calculi on Hopf algebras}
\label{section:classification_theorem}

It is well-known that, for the case in which $H$ is a Hopf algebra in ${}_{\Bbbk}\mathrm{Vec}$, there is an isomorphism of FODCi between the universal first order differential calculus on $H$ and $H \otimes H^+$, where $H^{+}$ denotes the kernel of the counit of $H$ (see, e.g., \cite{Woronowicz1989}), and $H\otimes H^{+}$ is equipped with suitable $H$-bimodule structure and differential. We now show that the same isomorphism holds for the universal first order differential calculus over a Hopf algebra $(\uB,\underline{\mu},\underline{\eta},\underline{\Delta},\underline{\varepsilon},\underline{S})$ in an abelian braided monoidal category $\underline{\mathcal{M}}$. More precisely, the first order differential calculus $(\uB \otimes \uB^+,\ud)$, where $\uB^+ = \ker\underline{\varepsilon}$ and $\ud(b) = b^1 \otimes b^2 - b \otimes 1$, is isomorphic to the universal FODC introduced in Proposition \ref{prop:braided_universal_is_initial} by the following assignment:

\begin{equation}
\label{eq:universal_is_iso_to_bb+}
\begin{aligned}
\chi\colon \Omega^{1}_{u}(\uB) & \stackrel{\cong}{\to} \uB\otimes \uB^{+}, \\ 
     b\otimes c &\mapsto b\, c^{1}\otimes c^{2}, \\ 
b\, \underline{S}(c^{1})\otimes c^{2} & \mapsfrom   \, b\otimes c.
\end{aligned}
\end{equation}

\begin{proposition}
\label{prop:universal_iso_bb+}
    The differential calculus $(\uB\otimes \uB^{+},\ud )$ over a braided Hopf algebra $\uB\in\underline{\mathcal{M}}$ is  braided bicovariant, with:
    \begin{enumerate}
        \item left $\uB$-action given by the left $\uB$ multiplication $a\triangleright(b\otimes c):= ab\otimes c$; 
        \item right $\uB$-action given by the braided diagonal action $(a\otimes b)\triangleleft c  := a\, \sigma(b\otimes c^{1})\, c^{2}$; 
        \item left $\uB$-coaction  ${}_{\uB\otimes \uB^{+}}\uDelta(a\otimes b) :=a^{1}\otimes a^{2}\otimes b$; 
        \item right $\uB$-coaction $\uDelta_{\uB\otimes \uB^{+}}(a\otimes b) :=a^{1}\otimes {}_{\alpha}(b^{2})\otimes {}^{\alpha}(a^{2}\, \underline{S}(b^{1}))b^{3}$.
    \end{enumerate}
 Moreover, the map $\chi$ in Equation \ref{eq:universal_is_iso_to_bb+} is an isomorphism in ${}^{\uB}_{\underline{\mathcal{M}}}\mathrm{Diff}(\uB)^{\uB}$.
\begin{proof}
    We discuss the proof of the above statements in order.  
\begin{enumerate}
        \item   Left $\uB$-multiplication clearly gives   rise to a well-defined left $\uB$-action on $\uB\otimes \uB^{+}$ with respect to which $\chi$ is left $\uB$-linear.
         \item The right $\uB$-action on $\uB\otimes \uB^{+}$ is induced via the isomorphism in Equation \eqref{eq:universal_is_iso_to_bb+} as $(a\otimes b) \triangleleft c := a \, \sigma(b\otimes c^{1})\, c^{2}$, which is in fact an action
          \begin{align*}
         ((a\otimes b)\triangleleft c)\triangleleft d & = (a\, \sigma(b\otimes c^{1})\, c^{2})\triangleleft d \\ 
        & = (a\, {}_{\beta}c^{1} \otimes {}^{\beta}b \, c^{2})\triangleleft d  \\ 
        & = a \, {}_{\beta}c^{1} \, {}_{\alpha}d^{1} \otimes {}^{\alpha}({}^{\beta}b \, c^{2}) d^{2}\\
        & = a \, {}_\beta c^{1} {}_{\alpha\gamma}d^{1} \otimes {}^{\alpha\beta}b \, {}_{\gamma} c^{2} \\ 
        & = a \, {}_{\beta}(c^{1}\,{}_{\alpha}d^{1})\otimes {}^{\beta}b\, {}^{\alpha}c^{2}\, d^{2} \\
        & = a \, \sigma (b\otimes c^{1}\, {}_{\alpha}d^{1})({}^{\alpha}c^{2}d^{2}) \\ 
        & = a\, \sigma(b\otimes (cd)^{1})\, (cd)^{2} \\ 
         & = (a\otimes b)\triangleleft (cd) 
    \end{align*}
    with respect to which $\chi$ is right $\uB$-linear, as 
    \begin{equation*}
    \chi(a\otimes bc) =a(bc)^{1}\otimes (bc)^{2}= ab^{1}\sigma(b^{2}\otimes c^{1})c^{2}= (ab^{1}\otimes b^{2})\triangleleft c = \chi(a\otimes b)\triangleleft c.
    \end{equation*}
    \item The map
    \begin{equation*}
    \begin{aligned}
    {}_{\uB\otimes \uB^{+}}\uDelta\colon \uB\otimes \uB^{+} & \to \uB\otimes(\uB\otimes \uB^{+}),\\
    a\otimes b & \mapsto a^{1}\otimes a^{2}\otimes b\,
    \end{aligned}
    \end{equation*}
     clearly gives a well defined left $\uB$-coaction on $\uB\otimes \uB^{+}$, and it is a routine check that $\chi$ is left $\uB$-colinear with respect to the latter.  
    
     \item Finally, consider the map 
     
    \begin{equation*}
    \begin{aligned}
    \uDelta_{\uB\otimes \uB^{+}}\colon \uB\otimes \uB^{+} & \to (\uB\otimes \uB^{+})\otimes \uB,\\
    a\otimes b & \mapsto a^{1}\otimes {}_{\alpha}(b^{1})\otimes {}^{\alpha}(a^{2}\, S(b^{1}))b^{3}.
    \end{aligned}
    \end{equation*}
    
    We show that the one of the assignments in Equation \eqref{eq:universal_is_iso_to_bb+} is right $\uB$-colinear, in particular $\uDelta_{\uB\otimes \uB^{+}}\circ \chi = (\chi\otimes \mathrm{id})\circ \uDelta_{\Omega^{1}_{u}(\uB)}$. As $\sigma\colon \uB\otimes\uB \to \uB\otimes\uB$ is a natural isomorphism between the functors $\otimes\colon (\uB\otimes-) \to \uB$ and $\otimes^{\mathrm{op}}\colon (-\otimes\uB) \to \uB$, the following diagram 
    \begin{equation*}
    \begin{tikzcd}
    \uB\otimes \uB \arrow[rr, "\sigma"] \arrow[dd, "\underline{\Delta}\otimes \mathrm{id}"'] &  & \uB\otimes \uB \arrow[dd, "\mathrm{id}\otimes \underline{\Delta}"] \\
                                                                         &  &                                                    \\
    (\uB\otimes \uB)\otimes \uB \arrow[rr, "\sigma"]                               &  & \uB\otimes (\uB\otimes \uB)                             
    \end{tikzcd}\, 
    \end{equation*}
    commutes, that is 
    \begin{equation*}
    \sigma\circ (\underline{\Delta}\otimes \mathrm{id})(a\otimes b)=(\mathrm{id}\otimes \underline{\Delta})\circ \sigma (a\otimes b) \quad \Rightarrow \quad ({}_{\alpha}b)^{1}\otimes ({}_{\alpha}b)^{2}\otimes {}^{\alpha}a ={}_{\alpha}b^{1}\otimes {}_{\beta}b^{2}\otimes {}^{\beta\alpha}a.
    \end{equation*}
    
   \noindent Therefore we have 
    \begin{align*}
    \uDelta_{\uB\otimes \uB^{+}}\circ \chi(a\otimes b) & = \uDelta_{\uB\otimes \uB^{+}}(ab^{1}\otimes b^{2}) \\ 
    & = (ab^{1})^{1}\otimes {}_{\alpha}(b^{22})\otimes {}^{\alpha}((ab^{1})^{2}S(b^{21}))b^{23} = \\
    & = a^{1} \, {}_{\beta}b^{11}\otimes {}_{\alpha}b^{22}\otimes {}^{\alpha}({}^{\beta}a^{2}b^{12}S(b^{21}))b^{23} \\
    & = a^{1}\, {}_{\beta}b^{1}\otimes {}_{\alpha}b^{2}\otimes {}^{\alpha\beta}(a^{2})b^{3} \\ 
    & = a^{1}({}_{\alpha}b^{1})^{1}\otimes ({}_{\alpha}b^{1})^{2}\otimes {}^{\alpha}(a^{2})b^{2} \\ 
    & = (\chi\otimes\mathrm{id})\left((a^{1}\otimes {}_{\alpha}b^{1})\otimes {}^{\alpha}(a^{2})b^{2}\right)\\ 
    & = (\chi\otimes \mathrm{id})\circ \uDelta_{\Omega^{1}_{u}(\uB)}(a\otimes b).
    \end{align*}
    That $\uDelta_{\uB\otimes \uB^{+}}$ is a right $\uB$-coaction on $\uB\otimes \uB^{+}$ now follows from the fact that $\uDelta_{\Omega^{1}_{u}(\uB)}$ is a coaction on $\Omega^{1}_{u}(\uB)$ and $\chi\colon \Omega^{1}_{u}(\uB)\to \uB\otimes \uB^{+}$ is an isomorphism in $\underline{\mathcal{M}}$.  
\end{enumerate}
We deduce that $\chi$ is a morphism in ${}^{\uB}_{\underline{\mathcal{M}}}\mathrm{Diff}(\uB)^{\uB}$ with inverse $\chi^{-1}$.
\end{proof}
\end{proposition}

Recall that, given a braided left covariant FODC, its subspace of (left) $\uB$-coinvariant elements is 
\begin{equation}\label{eq:left_coinvariants_FOBDC}
 {}^{\mathrm{co}\uB}\Omega^{1}(\uB):=\left\{\gamma\in\Omega^{1}(\uB)\mid {}_{\Omega^{1}(\uB)}\uDelta(\gamma)=1\otimes\gamma\right\}.
\end{equation}
\noindent In the next proposition we discuss some properties of the \emph{quantum Maurer--Cartan} form $\varpi\colon\ker\underline{\varepsilon} \to {}^{\mathrm{co}\uB}\Omega^{1}(\uB)$, which plays a fundamental role in our next classification theorem. As we will see, just as in ${}_{\Bbbk}\mathrm{Vec}$, covariant calculi on a braided Hopf algebra in $\underline{\mathcal{M}}$ are classified by $\ker\underline{\varpi}$.
\begin{proposition}
    Let $\uB$ be a Hopf algebra in $\underline{\mathcal{M}}$, and let $(\Omega^{1}(\uB),\ud )$ in ${}^{\uB}_{\underline{\mathcal{M}}}\mathrm{Diff}(\uB)$. Let $\uB^{+}:=\ker\underline{\varepsilon}$. The quantum Maurer--Cartan form 
    \begin{equation}
    \begin{aligned}
\underline{\varpi}\colon \uB^{+}&\to {}^{\mathrm{co}\uB}\Omega^{1}(\uB) \\ 
 b&\mapsto \underline{S}(b^{1})\ud b^{2}
    \end{aligned}  \label{eq:quantumMCformingeneral}
    \end{equation}
    is a surjective, right $\uB$-linear morphism in $\underline{\mathcal{M}}$. 
\begin{proof}
First of all, $\underline{\varpi}$ is a well-defined morphism in $\underline{\mathcal{M}}$, as it is defined as a composition of morphisms in $\underline{\mathcal{M}}$, and moreover 
    \begin{align*}
    {}_{\Omega^{1}(\uB)}\uDelta(\underline{S}(b^{1})\ud b^{2}) & = (\underline{S}(b^{1}))^{1} \, \sigma((\underline{S}(b^{1}))^{2}\otimes (b^{2})^{1})\, \ud ((b^{2})^{2}) \\ 
    & = {}_{\alpha}\underline{S}(b^{2})\, \sigma({}^{\alpha}\underline{S}(b^{1})\otimes b^{3}) \, \ud b^{4} \\ 
    & = {}_{\alpha}\underline{S}(b^{2}) \, {}_{\beta}b^{3} \otimes {}^{\beta\alpha}\underline{S}(b^{1})\ud b^{4} \\ 
    & = {}_{\alpha}(\underline{S}(b^{2})b^{3})\otimes {}^{\alpha}\underline{S}(b^{1})\ud b^{4} \\ 
    & = {}_{\alpha}1\otimes {}^{\alpha}\underline{S}(b^{1})\ud b^{2}\\ 
    & =\sigma(\underline{S}(b^{1})\otimes 1)\ud b^{2}\\ 
    & = 1\otimes \underline{\varpi}(b),
    \end{align*}
    for any $b\in\uB^{+}$. Now, let $a\,\ud b\in {}^{\mathrm{co}\uB}\Omega^{1}(\uB)$, that is 
    \[
    {}_{\Omega^{1}(\uB)}\uDelta(a\, \ud b):=a^{1}\, {}_{\alpha}b^{1}\otimes {}^{\alpha}a^{2}\, \ud b^{2}= 1\otimes a\, \ud b.
    \]
    Applying $\underline{\mu}\circ \left(\underline{S}\otimes \mathrm{id}\right)$ to both sides of the last identity we find 
    \begin{align*}
    a\,\ud b & = \underline{S}(a^{1}\, {}_{\alpha}b^{1})\, {}^{\alpha}a^{2}\,\ud b^{2} \\ 
    & = \underline{S}({}_{\beta\alpha}b^{1}) \, \underline{S}({}^{\beta}a^{1})\,{}^{\alpha}a^{2}\,\ud b^{2}\\ 
    & = \underline{S}({}_{\beta\alpha}b^{1}) \, {}^{\beta}(\underline{S}(a^{1}))\,{}^{\alpha}a^{2}\,\ud b^{2}\\ 
    & = \underline{S}({}_{\alpha}b^{1}){}^{\alpha}(\underline{S}(a^{1})a^{2})\ud b^{2} \\ 
    & = \underline{\varepsilon}(a)\,\underline{\varpi}(b) \\
    & = \underline{\varpi}(\underbrace{\underline{\varepsilon}(a)b-\underline{\varepsilon}(ab)1}_{\in\, \uB^{+}}),
    \end{align*}
    and thus surjectivity follows. Finally, we prove that $\underline{\varpi}$ is right $\uB$-linear with respect to the right $\uB$-module structure on ${}^{\mathrm{co}\uB}\Omega^{1}(\uB)$ given by the braided right adjoint $\uB$-action $\underline{\mathrm{Ad}}_{\mathrm{R}}(b)(\omega):={}_{\alpha}\underline{S}(b^{1}){}^{\alpha}\, \omega \, b^{2}$. 
    We find 
    \begin{align*}
    \underline{\varpi}(a\triangleleft b) & = \underline{\varpi}(ab) \\ 
    & = \underline{S}((ab)^{1})\ud ((ab)^{2}) \\ 
    & = \underline{S}(a^{1}\, {}_{\alpha}b^{1})\ud ({}^{\alpha}(a^{2})\, b^{2}) \\ 
    & = \underline{S}({}_{\beta\alpha}b^{1})\underline{S}({}^{\beta}a^{1}) \ud ({}^{\alpha}a^{2})b^{2} + \underline{S}({}_{\beta\alpha}b^{1})\underline{S}({}^{\beta}a^{1}){}^{\alpha}a^{2} \ud b^{2} \\ 
    & = \underline{S}({}_{\beta\alpha}b^{1})\underline{S}({}^{\beta}a^{1}) \ud ({}^{\alpha}a^{2})b^{2} + \underline{S}({}_{\gamma}b^{1}){}^{\gamma}(\underline{S}(a^{1})a^{2})  \ud b^{2}\\ 
    & = \underline{S}({}_{\beta\alpha}b^{1})\underline{S}({}^{\beta}a^{1}) \ud ({}^{\alpha}a^{2})b^{2}  \\ 
    & = \underline{S}({}_{\beta\alpha}b^{1})\, {}^{\beta}\underline{S}(a^{1}) \, {}^{\alpha}\ud (a^{2})b^{2} \\ 
    & = \underline{S}({}_{\alpha}b^{1})\, {}^{\alpha}(\underline{S}(a^{1})\ud a^{2})b^{2} \\ 
    & = {}_{\alpha}\underline{S}(b^{1})\, {}^{\alpha}\underline{\varpi}(a) \, b^{2}\\
    & = \underline{\mathrm{Ad}}_{R}(b)(\underline{\varpi}(a)).
    \end{align*} 
    \end{proof}
\end{proposition}
\begin{remark}
    Note that the map $b \mapsto \underline{S}(b^1)\ud b^2$ in Equation \ref{eq:quantumMCformingeneral} is well-defined on the whole braided Hopf algebra $\uB$, and in particular one may consider the corresponding restriction to the kernel of the counit $\uB^+$. In fact, for all $b \in \uB$, we have that
    \[
        \underline{S}(b^1)\ud b^2 = \underline{S}(b^1)\ud b^2 - \underline{\varepsilon}(b)\underline{S}(1)\ud 1 = \underline{\varpi}(b^+),
    \]
    where $b^+=b-\underline{\varepsilon}(b)$ denotes the canonical projection to the kernel of the counit. \qed 
\end{remark}

\noindent We now prove the main theorem of this section, namely the analogue of Woronowicz's classification theorem for covariant calculi over Hopf algebras in the more general context of Hopf algebras in braided monoidal categories. In particular, we show that  FODCi in ${}^{\uB}_{\mathcal{M}}\mathrm{Diff}(\uB)$ over a braided Hopf algebra $\uB\in\underline{\mathcal{M}}$ are in 1-1 correspondence with right ideals in the kernel of the counit of $\uB$, and that  FODCi in ${}^{\uB}_{\mathcal{M}}\mathrm{Diff}(\uB)^{\uB}$ are classified by those ideals that are invariant under 
\begin{equation}
    \begin{aligned}
        \underline{\mathrm{coAd}}_{\mathrm{R}}\colon \uB&\to \uB\otimes \uB, \\ 
        b & \mapsto  {}_{\alpha}(b^{2}) \otimes \underline{S}\left({}^{\alpha}(b^{1})\right)b^{3},
    \end{aligned}
\end{equation}
i.e., the braided adjoint coaction of $\uB$ on itself. 
\begin{theorem}\label{thm:classification_theorem_braided}
Let $\uB$ be a Hopf algebra in $\underline{\mathcal{M}}$. For any right ideal $I\subseteq \uB^{+}:=\ker\varepsilon$ we obtain a braided left covariant FODC $(\Omega^{1}(\uB),\ud )$ on $\uB$  by $\Omega^{1}(\uB):=\uB\otimes (\uB^{+}/I)$ and $\ud _{I}b:=(\mathrm{id}\otimes \pi)\circ (\underline{\Delta}(b)-b\otimes 1)$, where $\pi\colon \uB^{+}\to \uB^{+}/I$ is the quotient map, and where $\Omega^{1}(\uB)$ is a $\uB$-bimodule via 
\begin{equation}
    a\triangleright (b\otimes [c]):= ab\otimes [c], \qquad (a\otimes [b])\triangleleft c:= a\, {}_{\alpha}c^{1} \otimes {}^{\alpha}[b]\,c^{2}.
\end{equation}
The left $\uB$-coaction on $\Omega^{1}(\uB)$ is ${}_{\uB\otimes \uB^{+}/I}\uDelta:=\underline{\Delta}\otimes\mathrm{id}_{\uB^{+}/I}$. If $\underline{\mathrm{coAd}}_{\mathrm{R}}(I)\subseteq I\otimes \uB$ then $\Omega^{1}(\uB)$ is a braided bicovariant FODC, with right action 
\begin{equation} \label{equation:right_coaction_on_covariant_calc}
    \uDelta_{\uB\otimes \uB^{+}/I}(a\otimes [b]):=a^{1}\otimes {}_{\alpha}[b^{2}] \otimes {}^{\alpha}(a^{2} \underline{S}(b^{1}))\, b^{3}. 
\end{equation}
Moreover, every (left) braided covariant FODC over $\uB$ is of this form.
\begin{proof} $\Omega^{1}(\uB)$ is clearly a $\uB$-bimodule with left and right actions defined as above, and the differential is well-defined as a derivation $\ud \colon \uB\to \uB\otimes (\uB^{+}/I)$ in $\underline{\mathcal{M}}$, as it can be easily checked. The surjectivity axiom holds since, given any element $a\otimes [b]\in \Omega^{1}(\uB)$, we have 
\begin{align*}
a\,\underline{\varpi}(b) = a \,  \underline{S}(b^{1})\ud b^{2}  = a \, \underline{S}(b^{1})b^{2}\otimes [b^{3}] - a\,  \underline{S}(b^{1})b^{2}\otimes 1  = a\otimes [b] - a\otimes \underline{\varepsilon}(b)  = a\otimes [b].
\end{align*}
 The map ${}_{\uB\otimes \uB^{+}/I}\uDelta \colon \Omega^{1}(\uB) \to \uB\otimes \Omega^{1}(\uB)$ is clearly a left $\uB$-coaction, and the compatibility with the $\uB$-bimodule structure of $\Omega^{1}(\uB)$ follows from the isomorphism $\chi\colon \Omega^{1}_{u}(\uB)\to \uB\otimes \uB^{+}$ in Equation \eqref{eq:universal_is_iso_to_bb+}, restricting to an isomorphism $\widetilde{\chi}\colon \Omega^{1}_{u}(\uB)\big/\ker(\mathrm{id}\otimes \pi) \to  \uB\otimes \uB^{+}/I$, where $\pi\colon \uB\to\uB^{+}/I$ is the quotient map.  
 It is a straightforward check that the differential is a morphism of left $\uB$-comodules, and therefore the FODC $(\Omega^{1}(\uB),\ud_{I})$ is braided left covariant. Assume in addition that $\underline{\mathrm{coAd}}_{\mathrm{R}}(I)\subseteq I\otimes \uB$. The map $\uDelta_{\uB\otimes \uB^{+}/I}$ in Equation \eqref{equation:right_coaction_on_covariant_calc} is well-defined and compatible with the differential. Moreover, it is a right $\uB$-coaction, since the corresponding map in \eqref{eq:universal_is_iso_to_bb+} is, and the compatibility of $\uDelta_{\uB\otimes\uB^{+}/I}$ with the $\uB$-bimodule structure of $\Omega^{1}(\uB)$ is again induced by same isomorphism. It is a straightforward check that the differential is a morphism of right $\uB$-comodules, and therefore $(\Omega^{1}(\uB),{\ud} _{I})$ is a braided right covariant FODC, and in particular is a braided FODC.

Let now $(\Omega^{1}(\uB),{\ud} )$ be a braided left covariant FODC on a braided Hopf algebra $\uB$ in $\underline{\mathcal{M}}$. Consider $I:=\ker\underline{\varpi}\subseteq \uB^{+}$, which is a right $\uB$-ideal inside $\uB^+$ as $\underline{\varpi}$ is right $\uB$-linear (see Proposition \ref{eq:quantumMCformingeneral}). Note that since $\underline{\varpi}$ is surjective, it follows that  ${}^{\mathrm{co}\uB}\Omega^{1}(\uB)\cong \uB^{+}/I$. We claim that the following map,
\begin{equation}
\begin{aligned}
    \varphi\colon \Omega^{1}(\uB) & \to \uB\otimes (\uB^{+}/I), \\ 
                \gamma & \mapsto \gamma^{-2}\otimes \underline{\varpi}^{-1}\left(S(\gamma^{-1}\right)\gamma^{0}), \\ 
       a\, \underline{\varpi}(b) &\mapsfrom a\otimes [b],
    \end{aligned}
\end{equation}
provides an isomorphism of braided FODCi. First of all, $\varphi$ and $\varphi^{-1}$ are morphisms in $\underline{\mathcal{M}}$, since defined as a composition of morphisms in $\underline{\mathcal{M}}$, and one easily checks $\varphi\circ\varphi^{-1}=\mathrm{id}_{\uB\otimes (\uB^{+}/I)}$ and $\varphi^{-1}\circ\varphi=\mathrm{id}_{\Omega^{1}(\uB)}$ in $\underline{\mathcal{M}}$. Moreover, $\varphi^{-1}$ is manifestly left $\uB$-linear and right $\uB$-linear, and it is also left $\uB$-colinear, since we have 
\begin{align*}
{}_{\Omega^{1}(\uB)}\uDelta\circ \varphi^{-1}(a\otimes [b]) &={}_{\Omega^{1}(\uB)}\uDelta(a\, \underline{\varpi}(b)) \\ 
& = {}_{\Omega^{1}(\uB)}\uDelta(a\, \underline{S}(b^{1}){\ud} b^{2}) \\ 
& = (a\, \underline{S}(b^{1}))^{1}\, {}_{\alpha}(b^{21})\otimes {}^{\alpha}(a\, \underline{S}(b^{1}))^{2}\, {\ud} b^{22} \\ 
& =a^{1}\, {}_{\beta}(\underline{S}(b^{1}))^{1}\, {}_{\alpha}(b^{21})\otimes {}^{\alpha}({}^{\beta}a^{2}(\underline{S}(b^{1}))^{2})\, {\ud} b^{22}\\ 
& = a^{1}\, {}_{\beta}(\underline{S}(b^{2}))\, {}_{\alpha}(b^{3}) \otimes {}^{\alpha\beta}(a^{2}\, \underline{S}(b^{1}))\, {\ud} b^{4}\\ 
& = a^{1}\, {}_{\alpha}(\underline{S}(b^{2})b^{3}) \otimes {}^{\alpha}(a^{2}\underline{S}(b^{1})){\ud} b^{2} \\ 
& = a^{1}\otimes a^{2}\, \underline{S}(b^{1}){\ud} b^{2}\\ 
& = (\mathrm{id}\otimes \varphi^{-1})\circ \uDelta_{\uB\otimes \uB^{+}/I} (a\otimes [b]),
\end{align*}
Finally, the differentials on $(\Omega^{1}(\uB),{\ud} )$ and $(B\otimes (B^{+}/I),{\ud} _{I})$ satisfy $\varphi^{-1}\circ{\ud} _{I} ={\ud} $, and so $\varphi^{-1}$ is a morphism in ${}_{\underline{\mathcal{M}}}\mathrm{Diff}(\uB)$. Since $\varphi$ is the inverse of $\varphi^{-1}$ it is also $\uB$-bilinear and left $\uB$-colinear, and moreover ${\ud}_{I} ={\ud}\circ \varphi$. Therefore $\varphi$ is an isomorphism in ${}^{\uB}_{\underline{\mathcal{M}}}\mathrm{Diff}(\uB)$. If, furthermore, we assume $(\uB\otimes (\uB^{+}/I),{\ud}_{I})$ to be bicovariant, then the ideal $I$ satisfies $\underline{\mathrm{coAd}}_{\mathrm{R}}(I)\subseteq I\otimes \uB$, and the isomorphism $\varphi^{-1}$ is also right $\uB$-colinear, and thus a morphism in ${}^{\uB}_{\underline{\mathcal{M}}}\mathrm{Diff}(\uB)^{\uB}$, since

\begin{align*}
\uDelta_{\Omega^{1}(\uB)}\circ \varphi^{-1}(a\otimes [b]) & = \uDelta_{\Omega^{1}(\uB)}(a\, \underline{\varpi}(b)) \\ 
& = \uDelta_{\Omega^{1}(\uB)} (a \, \underline{S}(b^{1}){\ud}b^{2}) \\ 
& = (a \, \underline{S}(b^{1}))^{1}\, {\ud}({}_{\alpha}b^{21}) \otimes {}^{\alpha}(a\, \underline{S}(b^{1}))^{2} \, b^{22} \\ 
& = a^{1}\, {}_{\beta}\underline{S}(b^{1})^{1} {\ud}({}_{\alpha}b^{21})\otimes {}^{\alpha}({}^{\beta}a^{2}\, \underline{S}(b^{1})^{2})\, b^{22} \\ 
& = a^{1} {}_{\beta\gamma} \underline{S}(b^{12}){\ud}({}_{\alpha}b^{21})\otimes {}^{\alpha}({}^{\beta}a^{2} \, {}^{\gamma}\underline{S}(b^{11}))b^{22} \\ 
& = a^{1}\, {}_{\beta}\underline{S}(b^{2}){\ud}({}_{\alpha}b^{3})\otimes {}^{\alpha\beta}(a^{2}\underline{S}(b^{1}))b^{4}  \\ 
& = a^{1}\, {}_{\alpha}(\underline{S}(b^{2}){\ud}b^{3}) \otimes {}^{\alpha}(a^{2}\underline{S}(b^{1}))b^{4}\\ 
& = a^{1}\, {}_{\alpha}\underline{\varpi}(b^{2})\otimes {}^{\alpha}(a^{2}\underline{S}(b^{1}))b^{3}\\
& = a^{1}\, \underline{\varpi}({}_{\alpha}b^{2})\otimes {}^{\alpha}(a^{2}\underline{S}(b^{1}))b^{3}\\ 
& = (\varphi^{-1}\otimes \mathrm{id})(a^{1}\otimes [{}_{\alpha}b^{2}] \otimes {}^{\alpha}(a^{2}\underline{S}(b^{1}))b^{3})\\ 
& = (\varphi^{-1}\otimes \mathrm{id}) \circ \uDelta_{\uB\otimes \uB^{+}/I}(a\otimes [b]).
\end{align*}
\end{proof}
\end{theorem}

Theorem \ref{thm:classification_theorem_braided} provides a very convenient setup to understand which differential structures on a braided Hopf algebra are covariant or even bicovariant. Essentially the problem is reduced to finding right ideals in the category of choice, eventually checking whether or not those are invariant under the braided adjoint coaction. To showcase the usefulness of this approach, we now show that it is impossible to have non-trivial braided bicovariant differential structures on the braided quantum plane.
\begin{proposition}
    The only braided bicovariant FODCi on $\underline{\mathbb{C}}_q^2$ in ${}_{\mathcal{O}_q(\mathrm{GL}_2)}^{\mathcal{O}_q(\mathrm{GL}_2)}\mathcal{YD}$ are the universal calculus and the zero calculus.
\end{proposition}
\begin{proof}
    Building on Example \ref{ex:C_q(2)_braided}, let $\uB:=\underline{\mathbb{C}}_q^2$ and let $H:=\mathcal{O}_{q}(\mathrm{GL}_{2})$. The braided Hopf algebra $\uB$ is  $\mathbb{Z}_{\geq 0}$-graded, with grading given by the degree of the corresponding monomial. As the $H$-action and $H$-coaction on $\uB$ are degree preserving, any non-trivial right $\uB$-ideal $I \subsetneq \uB^+$ in ${}_H^H\mathcal{YD}$ is generated by multiplication from the right (equivalently left) of the following homogeneous set of generators
    \[
        \{x^ay^b \mid a + b = n, \quad a,b \in \mathbb{Z}_{\geq 0}, \ n \in \mathbb{Z}_{\geq 2} \}.
    \]
    Therefore, in view of Theorem \ref{thm:classification_theorem_braided}, for any such $I$ we have a braided left covariant FODC on $\uB$ in ${}^{H}_{H}\mathcal{YD}$. In light of \cite[Example 10.2.2, page 543]{MajidFoundation}, we have a convenient formula to compactly express the coproduct of any combination of elements in $\uB$.  Furthermore, using that 
    
    \[
    \sigma(x^{s}\otimes x^{r})=q^{-2rs}x^{r}\otimes x^{s}, \quad \underline{S}(x^{r})=(-1)^{r}q^{-r(r-1)}x^r, 
    \]
    
  \noindent  we get the following expression for the braided adjoint $\uB$-coaction on the element $x^n$:
  
    \begin{equation}
    \label{eq:braided_adjoint_coaction_quantum_plane}
        \underline{\mathrm{coAd}}_{\mathrm{R}}(x^n) = \sum_{r+s+t = n}(-1)^rq^{-r(r-1)}q^{-2rs}\frac{[n]_{q^{-2}}!}{[r]_{q^{-2}}![s]_{q^{-2}}![t]_{q^{-2}}!}x^{s} \otimes x^{r+t},
    \end{equation}
    where 
    \[
        [k]_{v} := 1+v+v^2+ \dots +v^{k-1}, \quad [k]_v! := [k]_v[k-1]_v \cdots [2]_v
    \]
    for any $k \in \mathbb{Z}_{\geq 0}$. Similar expressions may be obtained for general combinations of $x$ and $y$. Now, consider $I$ as generated by homogeneous elements of degree $n\in \mathbb{Z}_{\geq 2}$, explicitly $I:=\langle x^{n},x^{n-1}y,\dots, xy^{n-1},y^{n}\rangle$. Using Equation \eqref{eq:braided_adjoint_coaction_quantum_plane}, we show that the ideal $I$ is not $\underline{\mathrm{coAd}}_{\mathrm{R}}$-invariant. Namely, the braided adjoint $\uB$-coaction does not close, for instance, on the generator $x^n$.  Setting $s=n-1$ in Equation \ref{eq:braided_adjoint_coaction_quantum_plane} isolates the unique contribution that can produce a tensor of the form $x^{n-1} \otimes x$: 
    \[
        (1-q^{-2}) x^{n-1} \otimes x\notin I\otimes \uB. 
    \]
    Thus, the trivial and the universal FODCi are the only braided bicovariant ones on the braided Hopf algebra $\underline{\mathbb{C}}_q^2$. 
\end{proof}

Notwithstanding the impossibility of having a bicovariant differential structure, plenty of braided left covariant FODCi on the braided quantum plane do exists. In the next example we explicitly describe the one sharing the same dimension with the space of differential 1-forms of the underlying classical manifold. 

\begin{example}
\label{example:2D_calculus_on_the_quantum_plane}
Let $\underline{B}:=\underline{\mathbb{C}}_{q}^{2}\in {}^{H}_{H}\mathcal{YD}$, where $H:=\mathcal{O}_{q}(\mathrm{GL}_{2})$. In accordance with Theorem \ref{thm:classification_theorem_braided}, we construct a braided left covariant FODC $(\Omega^{1}(\uB),\underline{{\ud}})$ on $\uB$. We consider $I$ to be the $\uB$-ideal in $\uB^{+}$ generated, under right multiplication, by the generators $\{x^{2},y^{2},xy\}$. A direct computation on these generators shows that $I$ is a subobject of $\uB^{+}$ in ${}^{H}_{H}\mathcal{YD}$, and hence that the associated first order differential calculus is braided left covariant; as noted in the previous proposition, it fails to be braided bicovariant. The $\uB$-bimodule structure on $\Omega^{1}(\uB)=\mathrm{span}_{\uB}\{\underline{\mathrm{d}}x,\underline{\mathrm{d}}y\}$ is obtained by pulling back the bimodule structure of $\uB\otimes \uB^{+}/I$ along the isomorphism $\chi\colon \Omega^{1}(\uB)\to\uB\otimes \uB^{+}/I$ of Proposition \ref{prop:universal_iso_bb+}. Explicitly, it reads

    \begin{equation}
    \label{eq:bimodule_strucure_2d_calc_C_q(2)}
    x\,{\ud}x = q^{-2}{\ud}x \, x, \qquad y\,{\ud}y = q^{-2}{\ud}y \, y,\qquad y\,{\ud}x = q \,{\ud}x\, y,\qquad {\ud}y\,x = q^{-1}x\,{\ud}y + (q^{-2}-1)y\,{\ud}x.
    \end{equation}

\noindent Furthermore, the left $\uB$-coaction, on the generators, reads 

    \begin{equation}
        {}_{\Omega^{1}(\uB)}\underline{\Delta}(\underline{\mathrm{d}}x)  = 1\otimes \underline{\mathrm{d}}x, \qquad {}_{\Omega^{1}(\uB)}\underline{\Delta}(\underline{\mathrm{d}}y)  = 1\otimes \underline{\mathrm{d}}y.
    \end{equation}

\noindent Finally, the comodule structure of $\Omega^{1}(\uB)$ seen as an object in ${}^{H}_{H}\mathcal{YD}$ is 

    \begin{equation}
    {}_{\Omega^{1}(\uB)}\Delta(\underline{\mathrm{d}}x) =a\otimes \underline{\mathrm{d}}x + b\otimes \underline{\mathrm{d}}y, \qquad    {}_{\Omega^{1}(\uB)}\Delta(\underline{\mathrm{d}}y) =c\otimes \underline{\mathrm{d}}x + d\otimes \underline{\mathrm{d}}y.
    \end{equation}
    
    Notice that the same differential calculus was studied with a different procedure in \cite[Example 2.79]{bm}, where, while the underlying bimodule is seen as an object in ${}^{H}_{H}\mathcal{YD}$, the calculus is not understood to be braided left $\uB$-covariant.
    \qed
\end{example}

\subsection{Braided calculi on transmuted Hopf algebras}
\label{section:examples_of_braided_calculi}
In this Section we discuss some relevant examples of braided FODCi on braided Hopf algebras.  As already outlined in Section \ref{section:braided_hopf_algebras}, given a (co)quasitriangular Hopf algebra, we may consider the corresponding transmutation in the braided monoidal category of its (co)modules. This gives a braided Hopf algebra in the latter category, on which we may study braided FODCi in terms of classifying ideals. With the intent of understanding transmutations at the level of differential calculi on (co)quasitriangular Hopf algebras, we point the reader to Lemma \ref{lem:BbulletDGHA}, where the relation between differential graded bialgebras, graded Hopf algebras and bicovariant FODCi on Hopf algebras is discussed.

It turns out that starting with a bicovariant FODC on a (co)quasitriangular Hopf algebra naturally gives rise, under certain circumstances, to a braided bicovariant FODC on the corresponding transmutation.

\begin{proposition}\label{prop:transmutationcalc}
The following statements hold.
\begin{enumerate}
\item Let $H$ be a quasitriangular Hopf algebra in ${}_{\Bbbk}\mathrm{Vec}$, let $f\in\mathrm{End}_{\mathrm{Hopf}}(H)$ and let $\underline{H}_{f}$ denote the transmutation of $H$ by $f$ in the braided monoidal category ${}_{H}\mathrm{Mod}$. For each bicovariant FODC $(\Omega^{1}(H),\mathrm{d})$ on $H$ such that $(\mathrm{d}\colon H \to \Omega^{1}(H))\in {}_{H}\mathrm{Mod}$, there is a braided bicovariant FODC $(\Omega^{1}(\underline{H}_{f}),\mathrm{d})$ on $\underline{H}_{f}$ in ${}_{H}\mathrm{Mod}$, provided that the condition 
\begin{equation}
\label{eq:auxiliary_condition_for_bicovariance_quasitriangular}
    h_{1}\mathrm{d}(f(S(\mathcal{R}^{2})))\otimes \mathcal{R}^{1}\triangleright h_{2}=0
\end{equation}
holds for any $h\in H$. 
\item Let $H$ be a coquasitriangular Hopf algebra in ${}_{\Bbbk}\mathrm{Vec}$, let $g\in\mathrm{End}_{\mathrm{Hopf}}(H)$ and let $\underline{H}_{g}$ denote the transmutation of $H$ by $g$ in the braided monoidal category $\mathrm{Mod}^{H}$. For each bicovariant FODC $(\Omega^{1}(H),\mathrm{d})$ on $H$ there is braided bicovariant FODC $(\Omega^{1}(\underline{H}_{g}),\mathrm{d})$ on $\underline{H}_{g}$ in $\mathrm{Mod}^{H}$.
\end{enumerate}
\end{proposition}
\begin{proof}
    We discuss the two cases in order.
    \begin{enumerate}
        \item Given any $f\in \mathrm{End}_{\mathrm{Hopf}}(H)$ we naturally have an induced Hopf algebra map 
        \[
        \begin{aligned}
          H & \to H^{\bullet}=H\oplus \Omega^{1}(H),\\
        h& \mapsto f(h)\oplus 0,
        \end{aligned}
        \]
        which we continue to denote by $f$, by abuse of notation. Thus, via transmutation, we have a braided graded Hopf algebra $(\underline{H}_{f}^\bullet=\underline{H}_{f}\oplus \Omega^{1}(\underline{H}_{f}),\wedge)$ in ${}_{H}\mathrm{Mod}$, whose coproduct and antipode, for any $\omega\in \underline{H}_{f}^{\bullet}$, explicitly read 
        \begin{equation}
            \underline{\Delta}^{\bullet}(\omega)= \omega_{1} f(S(\mathcal{R}^{2}))\otimes \mathcal{R}^{1}\triangleright_{f} \omega_{2}, \qquad \underline{S}^{\bullet}(\omega)=f(\mathcal{R}^{2})S^{\bullet}(\mathcal{R}^{1}\triangleright_{f} \omega),
        \end{equation}
        and whose multiplication, unit and counit are the same as for $(H^{\bullet},\wedge,\mathrm{d})$.
    Moreover, the map $\mathrm{d}\colon H\to \Omega^{1}(H)$ is left $H$-linear by assumption, and thus we automatically have a derivation $\mathrm{d}\colon \underline{H}_{f}\to \Omega^{1}(\underline{H}_{f})$ in  ${}_{H}\mathrm{Mod}$. We now show that $(\underline{H}^{\bullet}_{f},\wedge,\mathrm{d})$ is a differential graded bialgebra in ${}_{H}\mathrm{Mod}$: the only non-trivial check to perform regards the commutativity of the diagram 
    \[
    \begin{tikzcd}
\underline{H}_{f} \arrow[rr, "\underline{\Delta}^{\bullet}"] \arrow[dd, "\mathrm{d}"] &  & \underline{H}_{f}\otimes \underline{H}_{f} \arrow[dd, "\mathrm{d}_{\otimes}"]                                   \\
                                                                            &  &                                                                                                                 \\
\Omega^{1}(\underline{H}_{f}) \arrow[rr, "\underline{\Delta}^{\bullet}",swap]    &  & \Omega^{1}(\underline{H}_{f})\otimes \underline{H}_{f} \oplus \underline{H}_{f}\otimes \Omega^{1}(\underline{H}_{f}),
\end{tikzcd}
\]
namely, that $\underline{\Delta}^{\bullet}$ is a morphism in $({}_H\mathrm{Mod})^{\bullet\leq1}$. We find 
\begin{align*}
\underline{\Delta}^{\bullet}\circ \mathrm{d}(h) &  = (\mathrm{d}h)_{1}f(S(\mathcal{R}^{2}))\otimes \mathcal{R}^{1}\triangleright_{f} (\mathrm{d}h)_{2} \\ 
& = \mathrm{d}h_{1}f(S(\mathcal{R}^{2}))\otimes \mathcal{R}^{1}\triangleright_{f} h_{2} + h_{1}f(S(\mathcal{R}^{2}))\otimes \mathcal{R}^{1}\triangleright_{f} \mathrm{d}h_{2} \\   \scriptstyle{(\text{Eq.}\eqref{eq:auxiliary_condition_for_bicovariance_quasitriangular})} &=\mathrm{d}h_{1}f(S(\mathcal{R}^{2}))\otimes \mathcal{R}^{1}\triangleright_{f} h_{2} +  h_{1}\mathrm{d}f(S(\mathcal{R}^{2}))\otimes \mathcal{R}^{1}\triangleright_{f} h_{2} + h_{1}f(S(\mathcal{R}^{2}))\otimes \mathcal{R}^{1}\triangleright_{f} \mathrm{d}h_{2} \\ 
& = \mathrm{d}_{\otimes}\left(h_{1}f(S(\mathcal{R}^{2}))\otimes \mathcal{R}^{1}\triangleright_{f} h_{2}\right) \\ 
& = \mathrm{d}_{\otimes}\circ\underline{\Delta}^{\bullet}(h).
\end{align*}
Thus, $(\underline{H}_{f}^{\bullet},\wedge,\mathrm{d})$ is a differential graded bialgebra in ${}_{H}\mathrm{Mod}$ and consequently $(\Omega^{1}(\underline{H}_{f}),\mathrm{d})$ is a braided bicovariant FODC on $\underline{H}_{f}$ in ${}_{H}\mathrm{Mod}$. 
\item In a similar fashion, given $g\in \mathrm{End}_{\mathrm{Hopf}}(H)$ as above, we naturally have an induced Hopf algebra map 
\[
\begin{aligned}
H^\bullet = H\oplus \Omega^{1}(H) &\to H, \\ 
h\oplus \omega & \mapsto g(h),
\end{aligned}
\]
that again, by abuse of notation, we continue to denote the above by $g$. The corresponding braided graded Hopf algebra $(\underline{H}_{g}^{\bullet} = \underline{H}_g \oplus \Omega^1(\underline{H}_g),\underline{\wedge})  \in \mathrm{Mod}^H$, obtained via transmutation, shares the same coproduct, counit and unit with $H$, whereas the product and antipode read 
\[
\omega\underline{\wedge}\gamma =\omega_{2}\,\gamma_{2}\,\mathcal{R}(g(S(\omega_{1})\omega_{3})\otimes g(S(\gamma_{1}))), \qquad \underline{S}^{\bullet}(\omega) = S^{\bullet}(\omega_{2})\, \mathcal{R}(g(S^{2}(\omega_{3})S(\omega_{1}))\otimes g(\omega_{4}))\,.
\]
Since $(\Omega^{1}(H),\mathrm{d})$ is a bicovariant FODC, the differential $\mathrm{d}\colon H \to \Omega^{1}(H)$ is a morphism in $\mathrm{Mod}^{H}$: indeed, from bicolinearity of $\mathrm{d}$ we have
\begin{align*}
\mathrm{coAd}_{\mathrm{R},g} \circ (k + \mathrm{d}h) & = k_{2}\otimes g(S(k_{1})k_{3}) + (\mathrm{d}h)_{2}\otimes g\left(S^{\bullet}((\mathrm{d}h)_{1}) (\mathrm{d}h)_{3}\right) \\ 
& = k_{2}\otimes g(S(k_{1})k_{3}) + \mathrm{d}h_{2}\otimes g(S(h_{1})h_{3})\\
& = \mathrm{coAd}_{\mathrm{R},g}(k) + (\mathrm{d} \otimes \mathrm{id}) \circ \mathrm{coAd}_{\mathrm{R},g}(h)\,,
\end{align*}
where we used that

\[
(\mathrm{id}\otimes \Delta^{\bullet})\circ \Delta^{\bullet}(\mathrm{d}h) = \mathrm{d}h_{1}\otimes h_{2} \otimes h_{3} + h_{1}\otimes \mathrm{d}h_{2}\otimes h_{3} + h_{1}\otimes h_{2}\otimes \mathrm{d}h_{3},
\]

\noindent alongside the definition of $g\colon {H}^{\bullet}\to H$. 
Moreover, $\mathrm{d}\colon\underline{H}_{g}\to \Omega^{1}(\underline{H}_{g})$ is also a derivation, giving the structure of a differential graded bialgebra to $(\underline{H}_g^\bullet,\underline{\wedge})$: the only check to perform is that $\underline{\wedge}$ is a morphism in $(\mathrm{Mod}^H)^{\bullet \leq 1}$ and for this it is enough to show the commutativity of the following diagram 
\[
\begin{tikzcd}
\underline{H}_{f} \arrow[dd, "\mathrm{d}"'] &  & \underline{H}_{f}\otimes \underline{H}_{f} \arrow[dd, "\mathrm{d}_{\otimes}"] \arrow[ll, "\underline{\wedge}"']                                   \\
                                            &  &                                                                                                                                                  \\
\Omega^{1}(\underline{H}_{f})               &  & \Omega^{1}(\underline{H}_{f})\otimes \underline{H}_{f} + \underline{H}_{f}\otimes \Omega^{1}(\underline{H}_{f}), \arrow[ll, "\underline{\wedge}"]
\end{tikzcd}
\]
which automatically follows by the property of $\mathrm{d}\colon H\to \Omega^{1}(H)$ being a derivation of $H$, indeed 
\begin{align*}
    \mathrm{d}(h\underline{\wedge}k) & = \mathrm{d}(h_{2}k_{2})\, \mathcal{R}(g(S(h_1)h_3) \otimes g(S(k_1))) \\ 
    & = \left(\mathrm{d}(h_{2})k_{2}+ h_{2}\mathrm{d}k_{2}\right)\, \mathcal{R}(g(S(h_1)h_3) \otimes g(S(k_1)))\\ 
    & = (\mathrm{d}h)_{2}k_{2}\mathcal{R}(g(S((\mathrm{d}h)_1)(dh)_3) \otimes g(S(k_1))) + h_{2}(\mathrm{d}k)_{2}\mathcal{R}(g(S(h_1)h_3) \otimes g(S((\mathrm{d}k)_1))) \\ 
    & = \mathrm{d}h\underline{\wedge}k + h\underline{\wedge}\mathrm{d}k.
\end{align*}
\end{enumerate}
\end{proof}

\begin{remark}
There is a striking asymmetry in Proposition \ref{prop:transmutationcalc}: transmuting a bicovariant FODC via a quasitriangular structure requires linearity of the differential, as well as condition \eqref{eq:auxiliary_condition_for_bicovariance_quasitriangular} to hold, while the transmutation of a bicovariant FODC via a \emph{co}quasitriangular structures can be performed without additional assumptions. This is of course expected since, firstly, the differential is colinear but not linear by assumption and, secondly, the transmutation maps are concentrated in degree zero, thus ensuring the differentiability of the transmuted wedge product but not of the transmuted coproduct in general. There is a notion of first order \emph{co}differential calculus $(\mathcal{W},\delta)$ on a coalgebra $C$, where $\mathcal{W}$ is a $C$-bicomodule and $\delta\colon\mathcal{W}\to C$ is a coderivation, satisfying an injectivity condition \cite{Julius}, and if $C$ is a Hopf algebra one is able to define bicovariant first order codifferential calculi \cite{Borowiec}. We expect that an analog of Proposition \ref{prop:transmutationcalc} can be proven for bicovariant first order codifferential calculi, with additional conditions now arising for coquasitriangular structures, but not for quasitriangular structures. It would be interesting to combine the notions of differential and codifferential calculi, to study their compatibility, (bi)covariance and to perform simultaneous transmutation.
\end{remark}

\begin{example}
\label{ex:sweedler_bicovariant_transmutation}
Consider the Sweedler Hopf algebra $E_{1}$ as in Example \ref{ex:sweedler_transmutation}, together with the FODC $(\Omega^{1}(E_{1}),\mathrm{d})$ on $E_{1}$ given by the $E_{1}$-bimodule span of the generator $\{\mathrm{d}x\}$ with relations 
\[
x\,\mathrm{d}x=-\mathrm{d}x\,x,\quad g\,\mathrm{d}x=-\mathrm{d}x\,g.
\]
The latter is an $E_{1}$-bicovariant FODC (see also \cite{paolo}), as the coproduct $\Delta\colon E_{1}\to E_{1}\otimes E_{1}$ lifts to well-defined left and right $H$-coactions on $\Omega^{1}(E_{1})$ by the assignment given on the generator $\mathrm{d}x$ as
\begin{equation}\label{eq:coaction_sweedler_calculus_bicovariant}
    \begin{aligned}
        {}_{\Omega^{1}(E_{1})}\Delta\colon \Omega^{1}(E_{1}) & \to E_{1}\otimes \Omega^{1}(E_{1}),  & \Delta_{\Omega^{1}(E_{1})}\colon \Omega^{1}(E_{1}) &\to \Omega^{1}(E_{1})\otimes E_{1}, \\ 
        \mathrm{d}x & \mapsto g\otimes \mathrm{d}x, & \mathrm{d}x &\mapsto \mathrm{d}x\otimes 1,\\ 
    \end{aligned}
\end{equation}
that we linearly extend to $\Omega^{1}(E_{1})$ as 
\[
{}_{\Omega^{1}(E_{1})}\Delta(a\,\mathrm{d}x) = \Delta(a)(\mathrm{id}\otimes\mathrm{d})\Delta(x), \qquad  \Delta_{\Omega^{1}(E_{1})}(a\,\mathrm{d}x) = \Delta(a)(\mathrm{d}\otimes\mathrm{id})\Delta(x),\qquad a\in E_{1}. \\ 
\]
It is an easy exercise to show that $\mathrm{d}\colon E_{1}\to \Omega^{1}(E_{1})$ is $E_{1}$-linear with respect to the left adjoint $E_{1}$-action and that Equation \eqref{eq:auxiliary_condition_for_bicovariance_quasitriangular} automatically holds. Thus, considering the transmutation $\underline{E_{1}}$ as in Example \ref{ex:sweedler_transmutation}, we have a braided bicovariant FODC $(\Omega^{1}(\underline{E_{1}}),\mathrm{d})$ in $\left({}_{E_{1}}\mathrm{Mod}, \otimes,\sigma_{\mathcal{R}_{0}}\right)$, whose coactions on the generator read
\begin{equation}
    \begin{aligned}
        {}_{\Omega^{1}(E_{1})}\Delta\colon \Omega^{1}(E_{1}) & \to E_{1}\otimes \Omega^{1}(E_{1}),  & \Delta_{\Omega^{1}(E_{1})}\colon \Omega^{1}(E_{1}) &\to \Omega^{1}(E_{1})\otimes E_{1}, \\ 
        \mathrm{d}x & \mapsto 1\otimes \mathrm{d}x, & \mathrm{d}x &\mapsto \mathrm{d}x\otimes 1.\\ 
    \end{aligned}
\end{equation}
\qed 
\end{example}

The last proposition is a partial symmetry statement between the notion of bicovariant differential calculi on (co)quasitriangular Hopf algebras and braided differential calculi on the corresponding transmutations, and Example \ref{ex:sweedler_bicovariant_transmutation} provides an explicit realisation of the above construction. However, it turns out that braided bicovariant differential calculi can also be obtained while requiring less on the differential structure for the original (co)quasitriangular Hopf algebra, e.g. considering left/right covariant calculi instead of bicovariant ones. 
On this front, we find the next example, again based on Sweedler's Hopf algebra, to be particularly enlightening.

\begin{example}
\label{ex:sweedler_hopf}
Building again on Example \ref{ex:sweedler_transmutation}, consider the FODC $(\Omega^{1}(E_{1}),\mathrm{d})$ on $E_{1}$ given by the $E_{1}$-bimodule span of generators $\{\mathrm{d}x,\mathrm{d}g\}$ with relations 
\[
x\,\mathrm{d}x=-\mathrm{d}x\,x,\quad g\,\mathrm{d}g=-\mathrm{d}g\,g,\quad \mathrm{d}g\,x=-x\,\mathrm{d}g,\quad \mathrm{d}x\,g=-g\,\mathrm{d}x.
\]
The FODC $(\Omega^{1}(E_{1}),\mathrm{d})$ is left $E_{1}$-covariant FODC, as the coproduct $\Delta\colon E_{1}\to E_{1}\otimes E_{1}$ lifts to a well defined left $H$-coaction on $\Omega^{1}(E_{1})$ by the assignment given on generators as 
\begin{equation}\label{eq:coaction_sweedler_calculus}
    \begin{aligned}
        {}_{\Omega^{1}(E_{1})}\Delta\colon \Omega^{1}(E_{1})& \to E_{1}\otimes \Omega^{1}(E_{1}), \\ 
        \mathrm{d}x & \mapsto g\otimes \mathrm{d}x,\\ 
        \mathrm{d}g & \mapsto g\otimes \mathrm{d}g,
    \end{aligned}
\end{equation}
that we linearly extend to $\Omega^{1}(E_{1})$ as 
\[
{}_{\Omega^{1}(E_{1})}\Delta(a\,\mathrm{d}x + b\,\mathrm{d}g) = \Delta(a)(\mathrm{id}\otimes\mathrm{d})\Delta(x) + \Delta(b)(\mathrm{id}\otimes\mathrm{d})\Delta(g),\quad a,b\in E_{1}. 
\]
On the other hand, it a easy check that the same calculus in not right $E_{1}$-covariant. Indeed, as the kernel of the quantum Maurer--Cartan gives the classifying ideal of $\Omega^{1}(E_{1})$ we find $I=\langle x+xg\rangle=\{x+xg\}$, and $I$ does not close under the right adjoint $E_{1}$-coaction. 

We now show the existence of a braided analogue of the  differential calculus $(\Omega^{1}(E_{1}),\mathrm{d})$ for the transmutation $\underline{E_{1}}$ of $E_{1}$ in $\left({}_{E_{1}}\mathrm{Mod}, \otimes,\sigma_{\mathcal{R}_{0}}\right)$ with respect to the identity map, and in particular we argue that such calculus is braided bicovariant. Proceeding as above, consider the FODC $(\Omega^{1}(\underline{E_{1}}),\ud)\in {}_{E_{1}}\mathrm{Mod}$ on $\underline{E_{1}}$ given by the $\underline{E_{1}}$-bimodule span of generators $\{\ud x,\ud g\}$ with same relations. As it can be readily checked, $(\Omega^{1}(\underline{E_{1}}),\ud )$ is  braided left $\underline{E_{1}}$-covariant, with $\underline{E_{1}}$-coaction to be given with the same assignment as in Equation \eqref{eq:coaction_sweedler_calculus}. Moreover, it is also right $\underline{E_{1}}$-braided, with right $\underline{E_{1}}$-coaction given as 
\begin{equation}\label{eq:r_coaction_sweedler_calculus}
    \begin{aligned}
        \uDelta_{\Omega^{1}(\underline{E_{1}})}\colon \Omega^{1}(\underline{E_{1}})& \to  \Omega^{1}(\underline{E_{1}})\otimes \underline{E_{1}}, \\ 
        \ud x & \mapsto \ud x\otimes 1 ,\\ 
        \ud g & \mapsto \ud g\otimes 1.
    \end{aligned}
\end{equation}
In light of Theorem \ref{thm:classification_theorem_braided}, we prove the latter statement by showing that the classifying ideal $I$ of  $\Omega^{1}(\underline{E_{1}})= \underline{E_{1}}\otimes \left(\underline{E_{1}}^{+}\big/I\right)$ is invariant under the adjoint $\underline{E_{1}}$-coaction and thus it is a braided FODC over $\underline{E_{1}}$. Given that $I:=\ker\underline{\varpi}$, and given that 
\[
\begin{aligned}
    \underline{\varpi}(x)& = \ud x,&\underline{\varpi}(g-1)&= (g-1)\,\ud g, & \underline{\varpi}(xg)&= -\ud x,
\end{aligned}
\]
we have $I=\langle x+xg\rangle =\{x+xg\}$. Finally we find $\underline{\mathrm{coAd}}_{\mathrm{R}}(I)\subseteq I\otimes \underline{E_{1}}$, indeed 
\[
\underline{\mathrm{coAd}}_{\mathrm{R}}(x+xg) = (x+xg)\otimes 1.
\]

The construction outlined above can be generalised to the higher dimensional Sweedler's Hopf algebra
\[
E_{n}:=\langle x_{1},\dots,x_{n},g,1\rangle \big/ 
\big\{x_{i}^{2},\,g^{2}-1,\,x_{i}g+gx_{i}, x_{i}x_{j}+x_{j}x_{i} \big\}.
\]
 Indeed, as discussed in \cite{Bottegoni_Sciandra}, the $\mathcal{R}$-matrix for Sweedler's Hopf algebra $E_{1}$ in Equation \eqref{eq:Rmatrix_Sweedler} is again an $\mathcal{R}$-matrix for $E_{n}$ for the case $\lambda=0$. It is now a routine check that, for the transmuted Hopf algebra $\underline{E_{n}}$ in ${}_{E_{n}}\mathrm{Mod}$, we have the braided Hopf algebra structure
         \[
\begin{aligned}
    x_{i}\triangleright x_{j} & =0, & x_{i}\triangleright g&=2x_{i}g, &  g\triangleright x_{i}& = -x_{i}, & g\triangleright g & = g,\\ 
    \underline{\Delta}(x_{i}) & = x_{i}\otimes 1 +1\otimes x_{i}, & \underline{\Delta}(g)&=g\otimes g, & \underline{S}(x_{i}) & = -x_{i}, & \underline{S}(g)&=g,
\end{aligned}
\]
with braiding reading
\[
    \sigma_{\mathcal{R}_{0}}(x_{i}\otimes x_{j})=-x_{j}\otimes x_{i}, \qquad \sigma_{\mathcal{R}_{0}}(x_{i}\otimes g) = g\otimes x_{i}, \qquad  \sigma_{\mathcal{R}_{0}}(g\otimes g)= g\otimes g.
\]
We consider the first order differential calculus 
$(\Omega^{1}(E_{n}),\mathrm{d} )$ on $E_{n}$ given by the $E_{n}$-bimodule span of generators $\{\mathrm{d} x_{1},\dots,\mathrm{d} x_{n},\mathrm{d} g\}$ with relations 
\[
x_{i}\,\mathrm{d} x_{j}=-\mathrm{d} x_{j}\,x_{i},\quad g\,\mathrm{d} g=-\mathrm{d} g\,g,\quad \mathrm{d} g\,x_{i}=-x_{i}\,\mathrm{d} g,\quad \mathrm{d} x_{i}\,g=-g\,\mathrm{d} x_{i}.
\]
The latter is again left covariant, whereas the corresponding analogue in $_{E_{n}}\mathrm{Mod}$ is braided bicovariant, with the right coaction on the generators being analogous to the one presented in Equation \eqref{eq:r_coaction_sweedler_calculus}. \qed

\end{example}
Generally speaking, the procedure outlined in the last example fails to recover a bicovariant calculus in the braided monoidal category of modules of a quasitriangular Hopf algebra given the initial datum of a left covariant calculus. For this claim we now exhibit an explicit counterexample.
\begin{example}
Let us consider the Sweedler's Hopf algebra 
\[
E_{2}:=\langle x,y,g,1\rangle \big/ 
\big\{x^{2},g^{2}-1,xg+gx,yg+gy,xy+yx \big\},
\]
whose corresponding $R$-matrices are classified, see e.g. \cite{Bottegoni_Sciandra}, as 
\begin{equation}
\label{eq:R_matrix_Sweedler_E2}
\begin{aligned}
\mathcal{R}_{f} & = \mathcal{R}_{0} + \frac{1}{2}\Big(\alpha(x\otimes x -gx\otimes x +x \otimes gx + gx \otimes gx) \\ 
& \quad + \beta (x\otimes y + x\otimes gy - gx \otimes y +gx\otimes gy )  \\ 
& \quad + \gamma(y\otimes x + y \otimes gx -gy\otimes x +gy\otimes gx ) \\ 
& \quad + \delta(y\otimes y - gy\otimes y + y \otimes gy + gy\otimes gy) \\ & \quad  + (\beta\gamma-\alpha\delta) (xy\otimes gxy -gxy \otimes gxy + xy \otimes gy + gxy \otimes xy) \Big),
\end{aligned}
\end{equation}
for $\alpha,\beta,\gamma,\delta\in \Bbbk$, where $\mathcal{R}_{0}$ is as in  Equation \eqref{eq:Rmatrix_Sweedler}. Motivated by the last example we perform the transmutation of $E_{2}$ in $\left({}_{E_{2}}\mathrm{Mod},\otimes,\sigma\right)$ with the choice of $\mathcal{R}$-matrix corresponding, for simplicity, to the case in which the only non vanishing coefficient in Equation \eqref{eq:R_matrix_Sweedler_E2} is $\alpha$. The $E_{2}$-action, transmuted coproduct and antipode read the same as before, with the exeption of the coproduct on the group-like generator which now reads 
\[
\underline{\Delta}(g) = g\otimes g + 2\alpha( x\otimes gx).
\]
Obviously, the braiding is now non-symmetric. Indeed, we have 
\begin{align*}
\sigma(g\otimes g) & = g\otimes g +4\alpha\, xg\otimes xg,
\end{align*}
while the other combinations read as before. Considering the differential calculus on $E_{2}$ defined as in the previous example, we can consider the corresponding transmutation in ${}_{\underline{{E}_{n}}}\mathrm{Mod}$. A calculation now reveals that the kernel of the Maurer--Cartan form is not invariant under the braided right adjoint $\underline{E_{2}}$-coaction. In particular, we find  
\[
\underline{\mathrm{coAd}}_{\mathrm{R}}(y+yg)\in 2\alpha (ygx\otimes xg -2 x g\otimes xy ) + \ker\underline{\varpi}\otimes \underline{E_{2}}. 
\]
Thus, the resulting first order differential calculus is not braided bicovariant, contrarily to what we observed in the previous example for the case of the symmetric $\mathcal{R}$-matrix. \qed
\end{example}

\section{Smash product calculi on the Radford--Majid biproduct}
\label{chapter:smash_calculi}
In this Section we discuss  first order differential calculi on smash product Hopf algebras arising from the Radford--Majid biproduct. As recalled in Section \ref{section:braided_hopf_algebras}, given a Hopf algebra $H \in {}_{\Bbbk}\mathrm{Vec}$ and a braided Hopf algebra $\uB \in {}^{H}_{H}\mathcal{YD}$, the bosonisation endows the smash product $A := \uB\# H$ with a Hopf algebra structure, and a natural question is whether first order calculi on $H$ and $\uB$ induce a compatible one on $A$. This problem was addressed in \cite{PflaumSchauenburg} for smash product algebras and in \cite{SciWeb} for crossed product algebras, leading to the notion of a \emph{smash (crossed) product calculus}, which endows $A$ with a rich differential structure admitting natural analogues of several features familiar from classical differential geometry.
In Section \ref{section:smash_product_calculi} we study first order smash product differential calculi on $A=\uB\# H$ in the setting at hand, establishing their covariance properties and giving explicit formulas for the corresponding coactions.
In Section \ref{subsection:Maurer_Cartan} we examine in detail the coinvariant $1$-forms of a covariant first order smash product calculus, and show that these decompose as a direct sum of the coinvariant $1$-forms of the two constituent calculi on $\uB$ and $H$. We discuss the resulting rich $(\text{co})$module structures, and describe how the Maurer--Cartan form of the calculus splits accordingly.   

\subsection{The smash product calculus}
\label{section:smash_product_calculi}
We recall the definition and bimodule structure of the smash product calculus $\Omega^1_{\#}(B\#H)$ in the following theorem.
\begin{theorem}
Let $H$ be a Hopf algebra and $B$ a left $H$-module algebra. Let $(\Omega^1(B),\mathrm{d}_B )$ a FODC over $B$ in the category of left $H$-modules, and $(\Omega^1(H),\mathrm{d} )$ a left covariant FODC on $H$. The $\Bbbk$-module
    \begin{equation}
    \Omega^1_{\#}(B\#H) := \left(\Omega^1(B) \otimes H\right) \oplus \left(B \otimes \Omega^1(H)\right)
    \label{eq:smashproductcalculusmaybe}
    \end{equation}
    is a $B\#H$-bimodule. The left $B\#H$-action on $\Omega^1_{\#}(B\#H)$ is
    \begin{equation}
        (b\#h)\cdot (\omega_B \otimes h^{\prime} + b^{\prime} \otimes \omega) := b(h_1 \triangleright \omega_B) \otimes h_2h^{\prime} + b(h_1 \triangleright b^{\prime}) \otimes h_2\omega,
        \label{eq:leftBsmashHactiononGammasmash}
    \end{equation}
    and the right $B\#H$-action on $\Omega^1_{\#}(B\#H)$ is
    \begin{equation}
        (\omega_B \otimes h^{\prime} + b^{\prime} \otimes \omega) \cdot (b \# h) := \omega_B (h^{\prime}_1 \triangleright b) \otimes h^{\prime}_2h + b^{\prime}\big(\omega_{-1} \triangleright b\big) \otimes \omega_0h,
        \label{eq:rightBsmashHactiononGammasmash}
    \end{equation}
    for all $b,b^{\prime} \in B$, $h,h^{\prime} \in H$, $\omega_B \in \Omega^1(B)$, $\omega \in \Omega^1(H)$, where $h\mathrel{\scriptstyle \rhd}(b\mathrm{d}_Bb') = (h_1\triangleright b)\mathrm{d}_B(h_2 \triangleright b')$ denotes the left $H$-action on $\Omega^1(B)$. Moreover, $\Omega^1_{\#}(B\#H)$ is a FODC on the smash product algebra $B\#H$, with differential defined by
    \[
        \mathrm{d}_{\#}(b\#h) := \mathrm{d}_B(b) \otimes h + b \otimes \mathrm{d}(h),
    \]
    for all $b \in B, h \in H$. We call $(\Omega^1_{\#}(B\#H),\mathrm{d}_{\#})$ a smash product calculus on $B\#H$.
    \label{thm:thesmashproductcalculus}
\end{theorem}
\begin{proof}
    See \cite{PflaumSchauenburg} and \cite[Theorem 3.14]{aflw}.
\end{proof}
Given a bicovariant FODC on $H$, the right $H$-coaction on $B\#H$ defined by $b\#h \mapsto b\#h_1 \otimes h_2$ is differentiable.
\begin{corollary}
\label{cor:thesmashproductcalculusisrightHcovariant}
    Let $H$, $B$ and $(\Omega^1(B), \mathrm{d}_B)$ as above, and $(\Omega^1(H), \mathrm{d})$ a bicovariant FODC on $H$. Then the smash product calculus $(\Omega^1_{\#}(B\#H), \mathrm{d}_{\#})$ is right $H$-covariant, with right $H$-coaction
    \begin{equation}
        \begin{split}
            \Delta_{H} \colon \Omega^1_{\#}(B\#H) &\to \Omega^1_{\#}(B\#H) \otimes H,\\
            \omega_B \otimes h + b \otimes \omega &\mapsto \omega_B \otimes h_1 \otimes h_2 + b \otimes \omega_0 \otimes \omega_1.
            \end{split}\label{eq:rightHcoactiononsmashcalctomakeitrightcovariant}
    \end{equation}

\end{corollary}
\begin{proof}
    See \cite[Corollary 3.15]{aflw} and \cite[Theorem 3.7]{SciWeb}.
\end{proof}
Henceforth, let $\underline{B}$ be a braided Hopf algebra in ${}_H^H\mathcal{YD}$ and let $\uB\#H$ denote corresponding Radford's biproduct (Theorem \ref{thm:bosonisation}). We study the $\underline{B}\#H$-covariance of smash product calculi on the bosonisation $\underline{B}\#H$. Notice that we consistently adopt lower-index Sweedler's notation for $H$-coactions and upper-index Sweedler's notation for $\uB$-coactions.
\begin{proposition}
 \label{pro:leftBsmashHcovariance}
    Let $H$ be a Hopf algebra and $\underline{B}$ a braided Hopf algebra in ${}_H^H\mathcal{YD}$. Let $(\Omega^1(H),\mathrm{d})$ be a left $H$-covariant $FODC$ on $H$, and $(\Omega^1(\underline{B}),\underline{\mathrm{d}}) \in {}_H^H\mathcal{YD}$ a braided left $\underline{B}$-covariant FODC on $\underline{B}$. Then, the smash product calculus $(\Omega^1_{\#}(\underline{B}\#H),\mathrm{d}_{\#})$ on the Hopf algebra $\underline{B}\#H$ is left $\underline{B}\#H$-covariant, with coaction
    \begin{equation}
        \begin{aligned}
            {}_{\uB\#H}\Delta \colon \Omega^1_{\#}(\underline{B}\#H) & \to  \underline{B}\#H \otimes \Omega^1_{\#}(\underline{B}\#H),\\
            \underline{\omega} \otimes h + b \otimes \omega & \mapsto  \underline{\omega}^{-1}\#(\underline{\omega}^0)_{-1}h_1 \otimes (\underline{\omega}^0)_0 \otimes h_2  \\
            & \quad +b^1\#(b^2)_{-1}\omega_{-1} \otimes (b^2)_0 \otimes\omega_0,
        \end{aligned}
        \label{eq:leftBsmashHcoaction}
    \end{equation}
    for any $\omega \in \Omega^1(H)$, $\underline{\omega} \in \Omega^1(\uB)$, $h \in H$, $b \in \uB$.
\end{proposition}
\begin{proof}
    The coaction in Equation \eqref{eq:leftBsmashHcoaction} is well-defined, since it is given in terms of well-defined coactions on $\Omega^1(\underline{B})$ and $\Omega^1(H)$ by assumption. Hence, it is enough to prove that the coaction in Equation \eqref{eq:leftBsmashHcoaction} is obtained as a lift the coproduct on $\underline{B}\#H$, that is,
    \[
        {}_{B\#H}\Delta((b\#h)\mathrm{d}_{\#}(c\#k)) = (b\#h)_1(c\#k)_1 \otimes (b\#h)_2\,\mathrm{d}_{\#}(c\#k)_2
    \]
    has to hold for all $(b\#h)\mathrm{d}_{\#}(c\#k) \in \Omega^1_{\#}(\underline{B}\#H)$. Since the smash product calculus is a direct sum $\Omega^1_{\#}(\underline{B}\#H) = \Omega^1(\underline{B}) \otimes H \oplus \underline{B} \otimes\Omega^1(H)$, we split the proof into two parts.
    \begin{enumerate}
        \item Let $\underline{\omega} \otimes h := b\underline{\mathrm{d}}b' \otimes h \in \Omega^1(\underline{B}) \otimes H$ be arbitrary. Recall here that the braided left $\underline{B}$-coaction on $\Omega^1(\underline{B}) \in {}_H^H\mathcal{YD}$ reads (see Corollary \ref{cor:liftofcoproductcovariance}):
        \[
            {}_{B}\Delta(b\underline{\mathrm{d}}b') = b^1{}_{\alpha}b'^1 \otimes {}^{\alpha}b^2\underline{\mathrm{d}}b'^2 = b^1( (b^2)_{-1} \triangleright b'^1) \otimes (b^2)_0\underline{\mathrm{d}}b'^2 =: (b\underline{\mathrm{d}}b')^{-1} \otimes (b\underline{\mathrm{d}}b')^0.
        \]
        Thus, we obtain
        \[
            \begin{split}
                {}_{B\#H}\Delta(b\mathrm{d_B}b' \otimes h) & = (b\mathrm{d_B}b')^{-1}\#((b\mathrm{d_B}b')^0)_{-1}h_1 \otimes ((b\mathrm{d_B}b')^0)_0 \otimes h_2 \\
                & = b^1( (b^2)_{-1} \triangleright b'^1)\#((b^2)_0\underline{\mathrm{d}}b'^2)_{-1}h_1 \otimes ((b^2)_0\underline{\mathrm{d}}b'^2)_0 \otimes h_2\\
                &= b^1((b^2)_{-1}\triangleright b'^1)\#((b^2)_{0})_{-1}(b'^2)_{-1}h_1 \otimes ((b^2)_0)_0\underline{\mathrm{d}}(b'^2)_0 \otimes h_2.
            \end{split}
        \]
        On the other hand, one can check that by the bimodule structure of $\Omega^1_{\#}(\underline{B}\#H)$, we have that
        \[
            b\underline{\mathrm{d}}b' \otimes h = (b\#1)\mathrm{d}_{\#}(b'\#h) - (bb'\#1)\mathrm{d}_{\#}(1\#h),
        \]
        and
       {\small 
        \begin{align*}
           & (b\#1)_1(b'\#h)_1 \otimes (b\#1)_2\mathrm{d}_{\#}(b'\#h)_2 - (bb'\#1)_1(1\#h)_1 \otimes (bb'\#1)_2\mathrm{d}_{\#}(1\#h)_2 \\
            & = (b^1\#(b^2)_{-1})(b'^1\#(b'^2)_{-1}h_1) \otimes ((b^2)_0\#1)\mathrm{d}_{\#}((b'^2)_0\#h_2) \\
            &\quad -((bb')^1\#((bb')^2)_{-1})(1\#h_1) \otimes (((bb')^2)_0\#1)\mathrm{d}_{\#}(1\#h_2) \\
            & = b^1(((b^2)_{-1})_1 \triangleright b'^1)\#((b^2)_{-1})_2(b'^2)_{-1}h_1 \otimes \big( (b^2)_0\underline{\mathrm{d}}(b'^2)_0 \otimes h_2 + (b^2)_0(b'^2)_0 \otimes \mathrm{d}h_2\big) \\
            &\quad -b^1((b^2)_{-1} \triangleright b'^1)\#((b^2)_0b'^2)_{-1}h_1 \otimes ((b^2)_0b'^2)_0 \otimes \mathrm{d}h_2 \\
            & = b^1((b^2)_{-2}\triangleright b'^1)\#(b^2)_{-1}(b'^2)_{-1}h_1 \otimes \big((b^2)_0\underline{\mathrm{d}}(b'^2)_0 \otimes h_2 + (b^2)_0(b'^2)_0 \otimes \mathrm{d}h_2 \big)\\
            &\quad -b^1((b^2)_{-1} \triangleright b'^1)\#((b^2)_0)_{-1}(b'^2)_{-1}h_1 \otimes ((b^2)_0)_0(b'^2)_0 \otimes \mathrm{d}h_2 \\
            & = b^1((b^2)_{-2}\triangleright b'^1)\#(b^2)_{-1}(b'^2)_{-1}h_1 \otimes \big((b^2)_0\underline{\mathrm{d}}(b'^2)_0 \otimes h_2 + (b^2)_0(b'^2)_0 \otimes \mathrm{d}h_2 - (b^2)_0(b'^2)_0 \otimes \mathrm{d}h_2 \big)  \\
            & = b^1((b^2)_{-2}\triangleright b'^1)\#(b^2)_{-1}(b'^2)_{-1}h_1 \otimes (b^2)_0\underline{\mathrm{d}}(b'^2)_0 \otimes h_2 \\
            & =b^1((b^2)_{-1}\triangleright b'^1)\#((b^2)_{0})_{-1}(b'^2)_{-1}h_1 \otimes ((b^2)_0)_0\underline{\mathrm{d}}(b'^2)_0 \otimes h_2.
        \end{align*}}
    \item Let $b \otimes \omega = b \otimes h\mathrm{d}h' \in \underline{B} \otimes \Omega^1(H)$ be arbitrary. We have
    \[
        b \otimes h\mathrm{d}h' = (b\#h)\mathrm{d}_{\#}(1\#h'),
    \]
    thus we can write
    \[
        \begin{split}
            (b\#h)_1(1\#h')_1 \otimes (b\#h)_2\mathrm{d}_{\#}(1\#h')_2 &= (b^1\#(b^2)_{-1}h_1)(1\#h'_1) \otimes ((b^2)_0\#h_2)\mathrm{d}_{\#}(1\#h'_2)\\
            &=b^1\#(b^2)_{-1}h_1h'_1 \otimes (b^2)_0 \otimes h_2\mathrm{d}h'_2 \\
            &=b^1\#(b^2)_{-1}(h\mathrm{d}h')_{-1} \otimes (b^2)_0 \otimes (h\mathrm{d}h')_0,
        \end{split}
    \]
    completing the proof.
    \end{enumerate}
\end{proof}
Right $\uB\#H$-covariance of smash product calculi on a Radford--Majid biproduct is more subtle. In the following proposition, we determine the conditions under which the smash product calculus $\Omega^1_{\#}(\uB\#H)$ is $\uB\#H$-bicovariant.
\begin{theorem}
    Let $(\Omega^1(H),\mathrm{d})$ be a bicovariant FODC on a Hopf algebra $H$, and let $(\Omega^1(\underline{B}),\underline{\mathrm{d}}) \in {}_H^H\mathcal{YD}$ be a braided bicovariant FODC on the braided Hopf algebra $\underline{B}\in{}_H^H\mathcal{YD}$. Let $\Omega^1_{\#}(\underline{B}\#H)$ be the corresponding smash product calculus. Then, the following are equivalent:
    \begin{enumerate}
        \item $\Omega^1_{\#}(\underline{B}\#H)$ is a bicovariant FODC on $\underline{B}\#H$, with left $\underline{B}\#H$-coaction as in Equation \ref{eq:leftBsmashHcoaction}, and right $\underline{B}\#H$-coaction given by
        \begin{equation}
            \begin{aligned}
                \Delta_{\uB\#H} \colon \Omega^1_{\#}(\underline{B}\#H) & \to \Omega^1_{\#}(\underline{B}\#H) \otimes \underline{B}\#H,\\
                c\underline{\mathrm{d}}c' \otimes h + b \otimes k\mathrm{d}k' & \mapsto \big(c^1\underline{\mathrm{d}}((c^2)_{-1} \triangleright c'^1) \otimes((c^2)_0)_{-1}(c'^2)_{-1}h_1 \\
                &\quad +c^1((c^2)_{-1} \triangleright c'^1) \otimes ((c^2)_0)_{-1}(\mathrm{d}(c'^2)_{-1})h_1 \big) \otimes (c^2)_0(c'^2)_0\#h_2 \\
                &\quad + b^1 \otimes (b^2)_{-1}k_1\mathrm{d}k'_1 \otimes (b^2)_0\#k_2k'_2.
            \end{aligned}
            \label{eq:rightBsmashHcoaction}
        \end{equation}
        \item The map
        \begin{equation}
            \begin{aligned}
                \mathrm{ver} \colon \Omega^1(\underline{B}) &\to \Omega^1(H) \otimes \underline{B},\\
                \underline{\omega} := b\underline{\mathrm{d}}b' & \mapsto (\underline{\omega})_{[-1]} \otimes (\underline{\omega})_{[0]} :=b_{-1}\mathrm{d}b'_{-1} \otimes b_0b'_0,
            \end{aligned}
            \label{eq:vermap}
        \end{equation}
        is well-defined.
        \item $\underline{B} \oplus \Omega^1(\underline{B})$ is a $H \oplus \Omega^1(H)$-comodule algebra object in the category of cochain complexes ${}_\Bbbk\mathrm{Vec}^\bullet$.
    \end{enumerate}
    \label{thm:onbicovarianceofsmash}
\end{theorem}
\begin{proof} 
    The first statement implies the second. In fact, assuming $\Omega^1_{\#}(\uB\#H)$ is bicovariant, which precisely means that the right $\uB\#H$-coaction in Equation \ref{eq:rightBsmashHcoaction} is well-defined (as it is obtained as the lift of the coproduct of $\uB\#H$ at the level of $\Omega^1_\#(\uB\#H)$ by a similar procedure to the proof of Propostion \ref{pro:leftBsmashHcovariance}), we have that
    \[
    \begin{split}
        \mathrm{pr}_{\uB \otimes \Omega^1(H)} \circ \Delta_{\uB\#H}(c\underline{\mathrm{d}}c' \otimes h + b \otimes k\mathrm{d}k')& = c^1((c^2)_{-1} \triangleright c'^1) \otimes ((c^2)_0)_{-1}(\mathrm{d}(c'^2)_{-1})h_1 \big) \otimes (c^2)_0(c'^2)_0\#h_2\\
        &\quad + b^1 \otimes (b^2)_{-1}k_1\mathrm{d}k'_1 \otimes (b^2)_0\#k_2k'_2,\\
        & = (c\ud c')^{-1} \otimes ((c\ud c')^0)_{[-1]}h_1 \otimes ((c\ud c')^0)_{[0]}\#h_2 +\\
        &\quad + b^1 \otimes (b^2)_{-1}k_1\mathrm{d}k'_1 \otimes (b^2)_0\#k_2k'_2,
    \end{split}
    \]
    and thus, also $\mathrm{ver}$ is well-defined. Conversely, if $\mathrm{ver}$ is well-defined, the coaction in \ref{eq:rightBsmashHcoaction} can be written as
    \begin{equation}         
            \begin{aligned}
                \Delta_{B\#H}\colon \Omega^1_{\#}(\underline{B}\#H) & \to \Omega^1(\underline{B}\#H) \otimes \underline{B}\#H,\\
                \underline{\omega} \otimes h + b \otimes \omega & \mapsto \underline{\omega}^0 \otimes(\underline{\omega}^1)_{-1}h_1 \otimes (\underline{\omega}^1)_0\#h_2\\
                &\quad +\underline{\omega}^{-1} \otimes (\underline{\omega}^0)_{[-1]}h_1 \otimes (\underline{\omega}^0)_{[0]}\#h_2 \\
                &\quad + b^1 \otimes (b^2)_{-1}\omega_{0} \otimes (b^2)_0\#\omega_1,
            \end{aligned}
            \label{eq:rightBsmashHcoactionver}
    \end{equation}
    which is obviously well-defined.
    
    We now prove that $2)$ implies $3)$. Let $\uB^{\bullet}:=\uB\oplus \Omega^{1}(\uB)$ be viewed as an object in ${}_{\Bbbk}\mathrm{Vec}^{\bullet\leq 1}$ with the corresponding cutoff monoidal structure as in Example \ref{ex:cochain}, and let $H^{\bullet}:=H\oplus \Omega^{1}(H)$ be a Hopf algebra object in the same category (see Lemma \ref{lem:gradHopf}).  Consider the following graded map of degree $0$:
    \begin{equation}
        \begin{split}
            {}_{\uB^\bullet}\Delta^\bullet \colon \uB \oplus \Omega^1(\uB) &\to (H \oplus \Omega^1(H)) \otimes (\uB \oplus \Omega^1(\uB)),\\
            b + \underline{\omega} &\mapsto b_{-1} \otimes b_0 + (\underline{\omega})_{-1} \otimes (\underline{\omega})_0 + (\underline{\omega})_{[-1]} \otimes (\underline{\omega})_{[0]}.
        \end{split}
        \label{eq:Hbulletcoaction}
    \end{equation}
    This defines a left $H^\bullet$-coaction on $B^\bullet$. In fact, on degree $0$ elements, this is just the left $H$-coaction on $\uB$, while on degree $1$ elements we verify that
    \[
        \begin{split}
            (\mathrm{id} \otimes {}_{\uB^\bullet}\Delta) \circ {}_{\uB^\bullet}\Delta^\bullet(\underline{\omega}) &= (\underline{\omega})_{-1} \otimes ((\underline{\omega})_0)_{-1} \otimes ((\underline{\omega})_0)_0 \\
            &+(\underline{\omega})_{-1} \otimes ((\underline{\omega})_0)_{[-1]} \otimes ((\underline{\omega})_0)_{[0]} + (\underline{\omega})_{[-1]} \otimes ((\underline{\omega})_{[0]})_{-1} \otimes ((\underline{\omega})_{[0]})_0 \\
            &=((\underline{\omega})_{-1})_1 \otimes ((\underline{\omega})_{-1})_2 \otimes (\underline{\omega})_0 \\
            &+ ((\underline{\omega})_{[-1]})_{-1} \otimes ((\underline{\omega})_{[-1]})_0 \otimes (\underline{\omega})_{[0]} + ((\underline{\omega})_{[-1]})_0 \otimes ((\underline{\omega})_{[-1]})_1 \otimes (\underline{\omega})_0  \\
            &= (\Delta^\bullet \otimes \mathrm{id}) \circ {}_{\uB^\bullet}\Delta^\bullet(\underline{\omega}),
        \end{split}
    \]
    where $\Delta^\bullet(h + \omega) := h_1 \otimes h_2 + \omega_{-1} \otimes \omega_0 + \omega_0 \otimes \omega_1$ is the comultiplication of the differential graded Hopf algebra $H^\bullet$. The counit property of the coaction in Equation \ref{eq:Hbulletcoaction} is quickly verified, as the counit on $\uB^\bullet$ vanishes on elements of degree greater than $0$. Moreover, the following diagram
    \begin{center}
        \begin{tikzcd}[column sep=1cm, row sep=0.8cm]
\Omega^1(\underline{B}) \arrow[rrr, "{}_{\uB^\bullet}\Delta^\bullet"]                                                  &  &  & H \otimes \Omega^1(\underline{B}) \oplus \Omega^1(H) \otimes \underline{B}                                      \\
                                                                                                             &  &  &                                                                                                                 \\
\underline{B} \arrow[uu, "\underline{\mathrm{d}}"] \arrow[rrr, "{}_{\uB^\bullet}\Delta^\bullet"] &  &  & H \otimes \uB \arrow[uu, "\mathrm{d}_\otimes"']
\end{tikzcd}
    \end{center}
    commutes, since
    \[
        \begin{split}
            {}_{\uB^\bullet}\Delta^\bullet\circ \underline{\mathrm{d}}(b + \underline{\omega}) &= (\underline{\mathrm{d}}b)_{-1} \otimes (\underline{\mathrm{d}}b)_0 + (\underline{\mathrm{d}}b)_{[-1]} \otimes (\underline{\mathrm{d}}b)_{[0]}\\
            &=b_{-1} \otimes \underline{\mathrm{d}}b_0 + \mathrm{d}b_{-1} \otimes b_0\\
            &=\mathrm{d}_\otimes \circ {}_{\uB^\bullet}\Delta^\bullet(b+\underline{\omega}),
        \end{split}
    \]
    where we used that $\underline{\mathrm{d}}$ is left $H$-colinear as $\Omega^1(\uB) \in {}_H^H\mathcal{YD}$ and well-definedness of the $\mathrm{ver}$ map. Thus, the left $H^\bullet$-coaction ${}_{\uB^\bullet}\Delta^\bullet$ is a morphism of cochain complexes. According to Lemma \ref{lem:BbulletDGHA}, the multiplication $\underline{\mu}^\bullet$ in Equation \ref{eq:mu_bullet} is a morphism of cochain complexes in ${}_H^H\mathcal{YD}^{\bullet\leq 1}$, thus, it is a morphism in ${}_{\Bbbk}\mathrm{Vec}^{\bullet\leq 1}$. In addition, $\underline{\mu}^\bullet$ is a morphism of $H^\bullet$-comodules. In fact, consider the following diagram:
    \begin{center}
        \begin{tikzcd}[column sep=1cm, row sep=0.8cm]
\underline{B}^\bullet \otimes \underline{B}^\bullet \arrow[dd, "\underline{\mu}^\bullet"'] \arrow[rrr, "{}_{B^\bullet}\Delta^\bullet_\otimes"] &  &  & H^\bullet \otimes \underline{B}^\bullet \otimes \underline{B}^\bullet \arrow[dd, "\mathrm{id} \otimes \underline{\mu}^\bullet"] \\
                                                                                                                                       &  &  &                                                                                                                                 \\
\underline{B}^\bullet \arrow[rrr, "{}_{B^\bullet}\Delta^\bullet"]                                                                      &  &  & H^\bullet \otimes \underline{B}^\bullet                                                                                        
\end{tikzcd},
    \end{center}
    
   \noindent  where
    
    \[
        \begin{aligned}
            {}_{\uB^\bullet}\Delta^\bullet_\otimes \colon \uB^\bullet \otimes \uB^\bullet & \to H^\bullet \otimes \uB^\bullet \otimes \uB^\bullet,\\
            b \otimes b' + c \otimes \underline{\omega}' + \underline{\omega} \otimes c' & \mapsto b_{-1}b'_{-1} \otimes b_0 \otimes b'_0 \\
            &\quad + c_{-1}(\underline{\omega})_{-1} \otimes c_0 \otimes (\underline{\omega})_0 + c_{-1} \cdot (\underline{\omega}')_{[-1]} \otimes c_0 \otimes (\underline{\omega}')_{[0]}\\
            & \quad +(\underline{\omega})_{-1}c'_{-1} \otimes (\underline{\omega})_0 \otimes c'_0 + (\underline{\omega})_{[-1]}\cdot c'_{-1} \otimes (\underline{\omega})_{[0]} \otimes c'_0
        \end{aligned}
    \]
    
   \noindent is the diagonal graded coaction of $H^\bullet$ on $\uB^\bullet \otimes \uB^\bullet$. On elements of degree $0$, the above diagram commutes by $H$-colinearity of the multiplication of $\uB$. On the other hand, on elements $c \otimes \underline{\omega}' + \underline{\omega} \otimes c' \in \uB \otimes \Omega^1(\uB) \oplus \Omega^1(\uB) \otimes \uB$, the above diagram commuting is equivalent to the following:
    \[
        \begin{split}
            &(c \cdot \underline{\omega}')_{-1} \otimes (c \cdot \underline{\omega}')_0 + (\underline{\omega} \cdot c')_{-1} \otimes (\underline{\omega} \cdot c')_0 + (c \cdot \underline{\omega}')_{[-1]} \otimes (c \cdot \underline{\omega}')_{[0]} + (\underline{\omega} \cdot c')_{[-1]} \otimes (\underline{\omega} \cdot c')_{[0]} =\\
            &=c_{-1}(\underline{\omega}')_{-1} \otimes c_0 \cdot (\underline{\omega}')_0 + (\underline{\omega})_{-1}c'_{-1} \otimes (\underline{\omega})_0 \cdot c'_0 + c_{-1} \cdot (\underline{\omega}')_{[-1]} \otimes c_0(\underline{\omega}')_{[0]} + (\underline{\omega})_{[-1]}\cdot c'_{-1} \otimes (\underline{\omega})_{[0]}c'_0.
        \end{split}
    \]
    The first two terms of the left-hand side coincide with the first two terms of the right-hand side by $H$-colinearity of the bimodule actions. The remaining terms also coincide, since, letting $\underline{\omega} = a\underline{\mathrm{d}}b$ and $\underline{\omega}'=a'\underline{\mathrm{d}}b'$, we can write
    \[
        \begin{split}
            (c \cdot \underline{\omega}')_{[-1]} &\otimes (c \cdot \underline{\omega}')_{[0]} + (\underline{\omega} \cdot c')_{[-1]} \otimes (\underline{\omega} \cdot c')_{[0]}  \\&=((ca')\underline{\mathrm{d}}b')_{[-1]} \otimes (ca'\underline{\mathrm{d}}b')_{[0]} + ((a\underline{\mathrm{d}}b)\cdot c')_{[-1]} \otimes ((a\underline{\mathrm{d}}b) \cdot c')_{[0]}\\
            &=((ca')\underline{\mathrm{d}}b')_{[-1]} \otimes (ca'\underline{\mathrm{d}}b')_{[0]} + (a\underline{\mathrm{d}}(bc'))_{[-1]} \otimes ((a\underline{\mathrm{d}}(bc'))_{[0]} - ((ab)\underline{\mathrm{d}}c')_{[-1]} \otimes ((ab)\underline{\mathrm{d}}c')_{[0]}\\
            &= c_{-1}a'_{-1}\mathrm{d}b'_{-1} \otimes c_0a'_0b'_0 + (a_{-1}\mathrm{d}(b_{-1}c'_{-1}) - (a_{-1}b_{-1})\mathrm{d}c'_{-1}) \otimes a_0b_0c'_0 \\
            &=c_{-1} \cdot (a'_{-1}\mathrm{d}b'_{-1}) \otimes c_0a'_0b'_0 + (a_{-1}\mathrm{d}b_{-1}) \cdot c'_{-1} \otimes a_0b_0c'_0\\
            &=(\underline{\omega})_0 \cdot c'_0 + c_{-1} \cdot (\underline{\omega}')_{[-1]} \otimes c_0(\underline{\omega}')_{[0]} + (\underline{\omega})_{[-1]}\cdot c'_{-1} \otimes (\underline{\omega})_{[0]}c'_0,
        \end{split}
    \]
    where we used well-definedness of $\mathrm{ver}$ as a linear map. \\
    Conversely, if $\uB^\bullet$ is a left $H^\bullet$-comodule algebra object in the category of cochain complexes, we have in particular a well-defined left $H^\bullet$-coaction ${}_{\uB^\bullet}\Delta^\bullet$ on $\uB^\bullet$. Thus, the projection of such coaction onto $\Omega^1(H) \otimes \uB$, that is,
    \[
        f := \mathrm{pr}_{\Omega^1(H) \otimes \uB} \circ {}_{\uB^\bullet}\Delta^\bullet|_{\Omega^1(\uB)}
    \]
    is well-defined. Since ${}_{\uB^\bullet}\Delta^\bullet$ is a morphism of cochain complexes by assumption, the map $f$ coincides with $\mathrm{ver}$ in Equation \ref{eq:vermap} on exact forms. Finally, since in addition $\uB^\bullet$ is an $H^\bullet$-comodule algebra, we can write, on a general element $b\underline{\mathrm{d}}b' \in \Omega^1(\uB)$
    \[
        \begin{split}
            f(b\underline{\mathrm{d}}b') :&= \mathrm{pr}_{\Omega^1(H) \otimes \uB} \circ {}_{\uB^\bullet}\Delta^\bullet(b\underline{\mathrm{d}}b') = \mathrm{pr}_{\Omega^1(H) \otimes \uB} \circ {}_{\uB^\bullet}\Delta^\bullet(b){}_{\uB^\bullet}\Delta^\bullet(\underline{\mathrm{d}}b')\\
            &=\mathrm{pr}_{\Omega^1(H) \otimes \uB} \circ {}_{\uB^\bullet}\Delta^\bullet(b)(\mathrm{id} \otimes d_B + \mathrm{d} \otimes \mathrm{id}){}_{\uB^\bullet}\Delta^\bullet(b') \\
            &=b_{-1}\mathrm{d}b'_{-1} \otimes b_0b'_0,
        \end{split}
    \]
    which is precisely $\mathrm{ver}(b\underline{\mathrm{d}}b')$. Thus, the map in Equation \ref{eq:vermap} is well-defined.
\end{proof}
\begin{corollary}
    If one, and hence all, of the three conditions in Theorem \ref{thm:onbicovarianceofsmash} holds, then $\uB \oplus \Omega^1(\uB)$ is in addition both a differential graded algebra and a differential graded coalgebra in the category of left $H \oplus \Omega^1(H)$-modules, where the $\Omega^1(H)$-component of the $H \oplus \Omega^1(H)$-action on $\uB$ is trivial.
\end{corollary}
\begin{proof}
    Let $\uB^\bullet := \uB \oplus \Omega^1(\uB)$ and $H^\bullet := H \oplus \Omega^1(H)$. Consider the following graded map of degree $0$:
    \[
        \begin{split}
            \triangleright^\bullet \colon (H \oplus \Omega^1(H)) \otimes (\uB \oplus \Omega^1(\uB)) &\to \uB \oplus \Omega^1(\uB),\\
        h \otimes b + h \otimes \underline{\omega} +  \omega \otimes b &\mapsto h \triangleright b + h \triangleright \underline{\omega},
        \end{split}
    \]
    where we defined the $\Omega^1(H)$ component of the $H^\bullet$-action on $\uB$ to identically vanish. The above defines a left $H^\bullet$-action on $\uB^\bullet$, as it is written in terms of the actions of $H$ on $\uB$ and $\Omega^1(\uB)$. Moreover, the following diagram
    \begin{center}
        \begin{tikzcd}
\Omega^1(H) \otimes \underline{B} \oplus H \otimes \Omega^1(\underline{B}) \arrow[rrr, "\triangleright^\bullet"]                                     &  &  & \Omega^1(\underline{B})                            \\
                                                                                                                                             &  &  &                                                                         \\
H \otimes \uB \arrow[uu, "\mathrm{d}_\otimes"] \arrow[rrr, "\triangleright^\bullet"] &  &  & \uB \arrow[uu, "\underline{\mathrm{d}}"']
\end{tikzcd}
    \end{center}
    commutes precisely when the $\Omega^1(H)$ component of the action on $\uB$ identically vanishes, since $\Omega^1(\uB) \in {}_H^H\mathcal{YD}$ and thus $\underline{\mathrm{d}}$ is in particular left $H$-linear. Hence, $\uB^\bullet$ is a left $H^\bullet$-module in the category of cochain complexes with respect to the action $\triangleright^\bullet$. Moreover, by Lemma \ref{lem:BbulletDGHA}, the multiplication $\underline{\mu}^\bullet$ in Equation \ref{eq:mu_bullet} is left $H^\bullet$-linear. Thus $\uB^\bullet$ is a $H^\bullet$-module algebra in ${}_{\Bbbk}\mathrm{Vec}^{\bullet\leq 1}$. Additionally, the comultiplication in Equation \ref{eq:delta_bullet} is $H^\bullet$-linear and $\uB^\bullet$ is a left $H^\bullet$-module coalgebra in ${}_\Bbbk\mathrm{Vec}^\bullet$. Finally, the differential $\underline{\mathrm{d}}$ is clearly left $H^\bullet$-linear.
\end{proof}

\begin{remark}
    In this remark we highlight some of the key differences between the construction presented here and the approach developed in \cite{AzizMajid}. In our setting, the starting datum is a braided Hopf algebra $\uB\in {}^{H}_{H}\mathcal{YD}$, together with a braided covariant FODC in the same category, from which we build a smash product FODC on the corresponding Radford--Majid biproduct. In Theorem \ref{thm:onbicovarianceofsmash}, we showed that in order for the smash product calculus to be bicovariant it is enough that the $H^\bullet$-coaction ${}_{\uB^\bullet}\Delta^\bullet$ in Equation \ref{eq:Hbulletcoaction} is an algebra morphism in ${}_{\Bbbk}\mathrm{Vec}^{\bullet\leq 1}$, while it is not needed for it to be a coalgebra morphism in general. In contrast, in \cite{AzizMajid}, $\uB^\bullet$ is required to be a braided graded Hopf algebra in ${}^{H^{\bullet}}_{H^{\bullet}}\mathcal{YD}$, while on the other hand $(\Omega^1(\uB),\underline{\mathrm{d}})$ is not assumed to be a braided covariant FODC in ${}_H^H\mathcal{YD}$, meaning in particular that $\underline{\mathrm{d}}$ is not $H$-linear. Thus, in \cite{AzizMajid}, the bimodule structure of the smash product FODC is different from the one described in Theorem \ref{thm:thesmashproductcalculus}, being the $\Omega^1(H)$ component of the $H^\bullet$-action on $\uB^\bullet$ is different from zero in general. 
    \qed
\end{remark}

\subsection{The Radford--Majid Maurer--Cartan form }
\label{subsection:Maurer_Cartan}
In this Section we study properties of the quantum Maurer--Cartan form associated to a smash product calculus on the Radford--Majid biproduct. 
We start by observing how the kernel of the counit $\varepsilon_{\#}$, denoted by $(\uB\# H)^{+}$, decomposes. 

Let $\uB\#H$ be a Radford--Majid biproduct, and let
    \[
    \begin{aligned}
    \pi_{+}\colon \uB\# H & \to (\uB\# H)^{+},\\ 
     b\#h & \mapsto (b\# h)^{+}:= b\# h - \varepsilon_{\#}(b\#h)1_{\#}
    \end{aligned}
    \]
    denote the projection to the kernel of the counit of $\uB\# H$. For any $b\# h\in \uB \# H$, it holds that 
    \begin{equation}
    \label{eq:counit_kernel_biproduct}
    (b\# h)^{+} = b^{+}\# h + \underline{\varepsilon}(b)1\# h^{+},
    \end{equation}
    where $b^+ = b- \underline{\varepsilon}(b)1$ and $h^+ = h-\varepsilon(h)1$.

We now prove a lemma showing how the Maurer--Cartan form $\varpi_\#$ decomposes accordingly.  

\begin{lemma}\label{lemma:splitting_formula}
Let $(\Omega^1(\uB),\underline{\mathrm{d}})$ and $(\Omega^1(H),\mathrm{d})$ be FODCi on the braided Hopf algebra $\uB \in {}_H^H\mathcal{YD}$ and on the Hopf algebra $H$ respectively, such that the smash product calculus $(\Omega^1_\#(\uB\#H),\mathrm{d}_\#)$ on the Radford--Majid biporduct $\uB\#H$ is left $\uB\#H$-covariant according to Proposition \ref{pro:leftBsmashHcovariance}. Consider 
\begin{equation}
    \begin{split}
            \varpi_{\#} \colon (\underline{B}\#H)^+ &\to \Lambda^1_{\#}(\underline{B}\#H),\\
            b\#h &\mapsto S_\#((b\#h)_1)\mathrm{d}_\#(b\#h)_2,
    \end{split}
    \label{eq:varpismash}
\end{equation}
the corresponding quantum Maurer--Cartan form. Then, the following formula
    \begin{equation}
        \varpi_{\#}((b\#h)^+) = \underline{\varpi}( (S(b_{-1}h_1))_1 \triangleright b^+_0) \otimes (S(b_{-1}h_1))_2h_2 + \underline{\varepsilon}(b)(1 \otimes \varpi_H(h^+)),
        \label{eq:splitting}
    \end{equation}
    holds for any $b\#h \in \underline{B}\#H$.
    \label{lem:splittingformula}
\end{lemma}
\begin{proof}
    Let $b\#h \in \uB\#H$. Then
    \[
        \begin{split}
        S_\#((b\#h)_1)\cdot \mathrm{d}_\#(b\#h)_2 & = S_\#(b^1\# (b^2)_{-1}h_1)\cdot\mathrm{d}_\#((b^2)_0\#h_2)\\
            & = (1\#S((b^1)_{-1}(b^2)_{-1}h_1))(\underline{S}((b^1)_0)\#1)\cdot \mathrm{d}_\#((b^2)_0\#h_2)\\
            & = \big(S((b^1)_{-1}(b^2)_{-1}h_1)_1\triangleright \underline{S}((b^1)_0)\#S((b^1)_{-1}(b^2)_{-1}h_1)_2 \big)\cdot (\underline{\mathrm{d}}(b^2)_0 \otimes h_2 + (b^2)_0 \otimes \mathrm{d}h_2)\\
            & =\big(S(b_{-1}h_1)_1 \triangleright \underline{S}((b_0)^1)\#S(b_{-1}h_1)_2\big) \cdot (\underline{\mathrm{d}}(b_0)^2 \otimes h_2 + (b_0)^2 \otimes \mathrm{d}h_2) \\
            & =\big(S(b_{-1}h_1)_1\triangleright \underline{S}((b_0)^1)\big)\underline{\mathrm{d}}\big(S(b_{-1}h_1)_2 \triangleright (b_0)^2\big) \otimes S(b_{-1}h_1)_3h_2 + \\
            &\quad +\big(S(b_{-1}h_1)_1\triangleright \underline{S}((b_0)^1)\big)\big(S(b_{-1}h_1)_2 \triangleright (b_0)^2\big) \otimes S(b_{-1}h_1)_3\mathrm{d}h_2\\
            & = S(b_{-1}h_1)_1 \triangleright \big(\underline{S}((b_0)^1)\underline{\mathrm{d}}(b_0)^2\big) \otimes S(b_{-1}h_1)_2h_2 \\
            &\quad + S(b_{-1}h_1)_1 \triangleright (\underline{S}((b_0)^1)(b_0)^2) \otimes S(b_{-1}h_1)_2\mathrm{d}h_2 \\
            & = S(b_{-1}h_1)_1 \triangleright \big(\underline{S}((b_0)^1)\underline{\mathrm{d}}(b_0)^2\big) \otimes S(b_{-1}h_1)_2h_2 + \underline{\varepsilon}(b) \otimes S(h_1)\mathrm{d}h_2,
        \end{split}
    \]
    where we used the definition of the antipode $S_\#$ (Equation \ref{eq:Hopfbosonisation}) and that of the left $\uB\#H$-action on $\Omega^1(\uB\#H)$ (Equation \ref{eq:leftBsmashHactiononGammasmash}), and we exploited the fact that $\underline{\mathrm{d}}$ and $\underline{\varepsilon}$ are morphisms in ${}_H^H\mathcal{YD}$ by assumption. Restricting our attention now to elements in $(\uB\#H)^+$ which, according to Equation \ref{eq:counit_kernel_biproduct}, are of the form $b^+\#h + \underline{\varepsilon}(b)1\#h^+$, we conclude
    \[
        \begin{split}
           \varpi_{\#}((b\# h)^{+})& =\underline{\varpi}( (S(b_{-1}h_1))_1 \triangleright b^+_0) \otimes (S(b_{-1}h_1))_2h_2 + \underline{\varepsilon}(b)(1 \otimes \varpi_H(h^+)).
        \end{split}
    \]
\end{proof}

Note that in classical differential geometry the Maurer--Cartan form on the affine extension $\mathbb{R}^n \rtimes \textrm{G}$ splits as the direct sum of two $1$-forms taking values on the Lie algebras $\mathbb{R}^n$ and $\mathfrak{g}$.  We now show that in our setting an analogous result holds.
\begin{proposition}
    Let $H$ be a Hopf algebra and $\underline{B} \in {}_H^H\mathcal{YD}$ a braided Hopf algebra. Then, for any left $\underline{B}\#H$-covariant smash product calculus $(\Omega^1_{\#}(\underline{B}\#H),\mathrm{d}_{\#})$ on $\underline{B}\#H$, we have
    \begin{equation}
    \label{eq:splitting_formula}
        \varpi_{\#}((b\#h)^+) = \varpi_{\#}( S^{-1}(h) \triangleright b^+ \#1) + \underline{\varepsilon}(b)(1 \otimes \varpi_H(h^+)).
    \end{equation}
    It follows that the quantum Maurer--Cartan form splits as
    \begin{equation}
        \varpi_{\#} = \varpi_{\#}|_{\underline{B}^+\#1} \oplus \varpi_{\#}|_{1\#H^+}.
        \label{eq:okokok}
    \end{equation}
\end{proposition}
\begin{proof}
    To prove this, it is enough to show that
    \[
        \varpi_{\#}(S^{-1}(h) \triangleright b\#1) = \underline{\varpi}( S(b_{-1}h_1)_1 \triangleright b_0) \otimes S(b_{-1}h_1)_2h_2
    \]
    for any $b\#h \in \uB^+\#H$. By Eq. \ref{eq:splitting}, we have that 
    $$
        \varpi_{\#}(b\#1) = \underline{\varpi}(S(b_{-1})_1\triangleright b_0) \otimes S(b_{-1})_2
    $$
    for all $b \in B$. Thus,
    \[
        \varpi_{\#}(S^{-1}(h) \triangleright b\#1) = \underline{\varpi}\big( S(S^{-1}(h) \triangleright b)_{-1})_1 \triangleright (S^{-1}(h) \triangleright b)_0\big) \otimes S(S^{-1}(h) \triangleright b)_{-1})_2.
    \]
    By the Yetter--Drinfeld compatibility we have that
    \[
    \begin{split}
        (S^{-1}(h) \triangleright b)_{-1} \otimes (S^{-1}(h) \triangleright b)_0 = &S^{-1}(h)_1b_{-1}S(S^{-1}(h)_3) \otimes S^{-1}(h)_2 \triangleright b_0 \\
        = &S^{-1}(h_3)b_{-1}S(S^{-1}(h_1)) \otimes S^{-1}(h_2) \triangleright b_0 \\
        = &S^{-1}(h_3)b_{-1}h_1 \otimes S^{-1}(h_2) \triangleright b_0,
    \end{split}
    \]
    thus
    \[
    \begin{aligned}
        \varpi_{\#}(S^{-1}(h) \triangleright b\#1) & = \underline{\varpi}\big( S(S^{-1}(h_3)b_{-1}h_1)_1 \triangleright S^{-1}(h_2) \triangleright b_0 \big) \otimes S(S^{-1}(h_3)b_{-1}h_1)_2 \\
        & = \underline{\varpi}\big( (S(b_{-1}h_1)h_3)_1S^{-1}(h_2) \triangleright b_0 \big) \otimes (S(b_{-1}h_1)h_3)_2 \\
        & = \underline{\varpi}\big( S(b_{-1}h_1)_1h_{31}S^{-1}(h_2) \triangleright b_0 \big) \otimes S(b_{-1}h_1)_2h_{32} \\
        & = \underline{\varpi}\big( S(b_{-1}h_1)_1h_{3}S^{-1}(h_2) \triangleright b_0 \big) \otimes S(b_{-1}h_1)_2h_{4} \\
        & = \underline{\varpi}\big( S(b_{-1}h_1)_1 \triangleright b_0 \big) \otimes S(b_{-1}h_1)_2h_2.
    \end{aligned}
    \]
\end{proof}

\begin{example}
    We spell out the quantum Maurer--Cartan form for the example of the quantum affine extension of $\mathcal{O}_{q}(\textrm{GL}_{2})$, that is the Radford--Majid biproduct $\underline{\mathbb{C}}_q^2\#\mathcal{O}_{q}(\textrm{GL}_{2})$. For example, considering the element $x\# 1 \in \ker\varepsilon_{\#}$, and exploiting the decomposition of the Maurer--Cartan form as explained in Equation \eqref{eq:splitting}, we get 
    \begin{align*}
    \varpi_{\#}(x\# 1) & = S(x_{-1})\cdot(\underline{\varpi}(x_{0})\#1)\\ 
    & = S(a)\cdot(\underline{\varpi}(x)\#1) + S(b)\cdot(\underline{\varpi}(y)\#1) \\ 
    & = D^{-1}d \cdot(\underline{\varpi}(x)\# 1) - q^{-1}D^{-1}b\cdot(\underline{\varpi}(y)\# 1) \\ 
    & = -q^{3}\underline{\varpi}(y)\# D^{-1}b + q^{2}\underline{\varpi}(x)\# D^{-1}d.
    \end{align*}
   Similarly, one may evaluate the remaining combinations in $\ker\varepsilon_{\#}$. We only list the other non-vanishing contributions:
    \begin{align*}
        \varpi_{\#}(y\#1)& = q^4 \underline{\varpi}(y) \otimes D^{-1}a - q^3\underline{\varpi}(x) \otimes D^{-1}c\,;\\
        \varpi_{\#}(x\#a)& = q^2\varpi_{\#}(x\#1)\,;\\
        \varpi_{\#}(y\#a)& = q\,\varpi_{\#}(y\#1)\,;\\
        \varpi_{\#}(x\#c)& = (q^2-1)\varpi_{\#}(y\#1)\,;\\
        \varpi_{\#}(x\#d)& = q\, \varpi_{\#}(x\#1)\,;\\
        \varpi_{\#}(y\#d)& = q^2\varpi_{\#}(y\#1).
    \end{align*}
    \qed
\end{example}

We now move to the study of the left coinvariant $1$-forms on the trivial Hopf--Galois extension $\uB\subseteq \uB\# H$. The main result we prove here is that, given a left covariant smash product calculus on the total space Hopf algebra $A:=\uB\# H$, the left coinvariant forms of the bundle decompose, as right $H$-covariant right $H$-modules, into the direct sum of the left coinvariant forms of the base space and the left coinvariant forms of the structure Hopf algebra.

We first recall some known facts on differential calculi on a general quantum principal bundle $B := A^{\mathrm{co}H} \subseteq A$ (see Definition \ref{def:QPB}). Let $(\Omega^1(A), \rm d_A)$ be a FODC on a right $H$-comodule algebra
$A$. We call the left $A$ and right $B$-bimodule
$\Omega^1_\mathrm{hor}(A):=A\Omega^1(B)$ \emph{horizontal forms}. Moreover, when $\Omega^1(A)$ is right $H$-covariant, we have that $\Omega^1(B)=\Omega^1_\mathrm{hor}(A)\cap \Omega^1(A)^{\mathrm{co}H}$ (see for example \cite{aflw} for more details).
When the following sequence is exact
\begin{equation}\label{eq134}
    0\longrightarrow \Omega^{1}_{\mathrm{hor}}\stackrel{\iota}{\longrightarrow}\Omega^1(A)\stackrel{\pi_{\mathrm{ver}}\ \  }{\longrightarrow}
    A\otimes\Lambda^1_{H}\longrightarrow 0
\end{equation}
where the \emph{vertical map}
\begin{align*} \pi_{\mathrm{ver}}:\Omega^1(A)&\to A\otimes \Omega^1(H), \\a\mathrm{d}_Aa'&\mapsto a_0a'_0\otimes a_1\mathrm{d}a'_1
\end{align*}
is assumed to be well-defined, we name the calculus \textit{principal}. 

As a trivial Hopf--Galois extension, $\uB\subseteq\uB\#H$ is a particular case of a quantum principal bundle. From this point on we consider:
 \begin{enumerate}
     \item $\Omega^{1}(\uB)$ a braided left $\uB$-covariant FODC on a braided Hopf algebra $\uB\in {}^{H}_{H}\mathcal{YD}$, whose classifying ideal we denote by $I_{\uB}$ according to Theorem \ref{thm:classification_theorem_braided}, and whose left coinvariant $1$-forms are $\Lambda^{1}_{\uB}:=\uB^{+}/I_{\uB}$. 
     \item $\Omega^{1}(H)$ a bicovariant FODC on the Hopf algebra $H\in{}_{\Bbbk}\mathrm{Vec}$, whose classyfing ideal we denote by $I_{H}$, and whose left coinvariant $1$-forms are $\Lambda^{1}_{H}:=H^{+}/I_{H}$.
     \item $\Omega^{1}_{\#}(A)$ the corresponding right $H$-covariant (Corollary \ref{cor:thesmashproductcalculusisrightHcovariant}) and left $A$-covariant (Proposition \ref{pro:leftBsmashHcovariance}) smash product FODC on the Hopf algebra in ${}_{\Bbbk}\mathrm{Vec}$ given by the Radford--Majid biproduct $A:=\uB\# H$, whose classifying ideal we denote by $I_{\#}$ and whose left coinvariant $1$-forms are $\Lambda^{1}_{A}:=A^{+}/I_{\#}$. 
     \item Also, $\Omega_{u}^{1}(-)$ denotes the universal FODCi, with $\mathcal{N}_{H}$, $\mathcal{N}_{\uB}$ and $\mathcal{N}_{\#}$ the corresponding subbimodules realising $\Omega^{1}(H)$, $\Omega^{1}(\uB)$ and $\Omega^{1}_{\#}(A)$ as quotients. 
\end{enumerate}
One easily verifies that for the setting at hand we have $\Omega^{1}_{\mathrm{hor}}(A)\cong \Omega^{1}(\uB)\otimes H$. Moreover, any such smash product calculus is principal (\cite{aflw}, Example 2.20). Therefore, by means of the exactness of the Atiyah sequence and by Theorem \ref{section:classification_theorem}, we have the following commuting diagram, which gives a clear interpretation of the horizontal and vertical forms of the trivial Hopf--Galois extension $\uB\subseteq \uB\#H$. 

\begin{equation}
\label{equation:diagram_smash}
\begin{tikzcd}
0 \arrow[r] & (\Omega^{1}_{u}(\uB)/\mathcal{N}_{\uB})\otimes H \arrow[r,"\iota^{u}"] \arrow[d,"\cong"]                       & \Omega^{1}_{u}(\uB\# H)/\mathcal{N}_{\#} \arrow[r,"\pi_{ver}^{u}"] \arrow[d,"\cong"]                                  & \uB\otimes(\Omega^{1}_{u}(H)/\mathcal{N}_{H}) \arrow[r] \arrow[d,"\cong"]      & 0 \\
0 \arrow[r] & \Omega^{1}(\uB)\otimes H \arrow[r, "\iota"] \arrow[d,"\cong"]         & \Omega^{1}_{\#}(\uB\# H) \arrow[r, "\pi_{\mathrm{ver}}"] \arrow[d,"\cong"]       & \uB\otimes \Omega^{1}(H) \arrow[r] \arrow[d,"\cong"] & 0 \\
0 \arrow[r] & \Lambda^{1}_{\uB}\otimes \uB\otimes H \arrow[r, , dashed] & (\uB\# H)\otimes (\uB\# H)^{+}/I_{\#} \arrow[r, dashed] & \uB\otimes H  \otimes\Lambda^{1}_{H} \arrow[r]      & 0.
\end{tikzcd}
\end{equation}
Moreover, according to Equation \eqref{eq:splitting_formula}, we have 
\[
\begin{tikzcd}
\uB^{+} \# H   \oplus  1\# H^{+}\cong(\uB\# H)^{+} \arrow[rrr, "\varpi_{\#}"] &  &  & {}^{\mathrm{co}A}(\Lambda^{1}_{\uB}\otimes H) \oplus  1\otimes \Lambda^{1}_{H}.
\end{tikzcd}
\]

Although the Maurer--Cartan form of the bundle provides a clear identification of the coinvariant 1-forms on the structure Hopf algebra, the corresponding coinvariant 1-forms on the base space are less transparent as a consequence of the product structure with which the Radford--Majid biproduct is endowed. The following proposition clarifies this point.  

\begin{proposition}
\label{prop:iso_of_base_forms_with_horizontals}
We have $\Lambda^{1}_{\uB} \cong {}^{\mathrm{co}A}(\Lambda^{1}_{\uB}\otimes H)$, as vector spaces. Explicitly, the isomorphism is given by the assignments 
\begin{equation}
	\begin{aligned}
    \label{eq:horizontal_iso}
		\psi\colon\Lambda^{1}_{\uB} &\to {}^{\mathrm{co}A}(\Lambda^{1}_{\uB}\otimes H) \\ 
		\underline{\varpi}(b)&\mapsto \underline{\varpi}(b_{0})\otimes S^{-1}(b_{-1}),\\ 
		\varepsilon(h)\underline{\varpi}(b) &\mapsfrom \underline{\varpi}(b)\otimes h.
	\end{aligned}
\end{equation}

\begin{proof}
	Let $A=\uB\# H$. First of all, notice that 
	\[
	\Omega^{1}(\uB)\otimes H \cong \Lambda^{1}_{\uB}\otimes \uB\otimes H \cong \Lambda^{1}_{\uB}\otimes A.
	\] 
    By the fundamental theorem of Hopf modules, together with Equation \eqref{eq:splitting_formula}, we have 
	\begin{align*}
	\Omega^{1}_{\#}(\uB\# H)& \cong \Lambda^{1}_{A}\otimes A\\ 
    &\cong \left({}^{\mathrm{co}A}(\Lambda^{1}_{\uB}\otimes H)\otimes A \right) \oplus \left((\Bbbk\otimes \Lambda^{1}_{H})\otimes A\right)\\ 
    &\cong \left({}^{\mathrm{co}A}(\Lambda^{1}_{\uB}\otimes H)\otimes A \right) \oplus \left(\Lambda^{1}_{H}\otimes A\right). 
	\end{align*}
    
\noindent 	As it can be easily deduced, the isomorphism of vector spaces $\Bbbk\otimes\Lambda^{1}_{H}\otimes A\cong \uB\otimes \Omega^{1}(H)$ holds. Moreover, as by any such choice of smash product FODC we have split-exactness of the sequences in Equation \eqref{equation:diagram_smash}, we have 

	  \[
	{}^{\mathrm{co}A}(\Lambda^{1}_{\uB}\otimes H)\otimes A \cong \Lambda^{1}_{\uB} \otimes A.
	\] 
    
  \noindent  Finally, as any Hopf algebra realises a quantum principal bundle  (see Section \ref{subsection:Hopf--Galois_extensions_and_smash_product_algebras}) on the underlying ground field,  $A$ is faithfully flat as a $(A^{\mathrm{co}A}=\Bbbk)$-module, and thus the claim holds. 

    More explicitly, we may show that the maps in Equation \eqref{eq:horizontal_iso} indeed gives an isomorphism of vector spaces. First of all we find  
        \begin{align*}
    {}_{\#}\Delta(\underline{\varpi}(b)\otimes 1) & = (\underline{S}(b^{1})^{1} \# \underline{S}(b^{1})^{2}_{-1})\cdot_{\#}(b^{21}\# b^{22}_{-1})\otimes (\underline{S}(b^{1})^{2}_{0}\underline{\mathrm{d}}b^{22}_{0}\otimes 1) \\ 
    & = (\underline{S}(b^{1})_{-2}\triangleright \underline{S}(b^{2})\# \underline{S}(b^{1})_{-1})\cdot_{\#}(b^{3}\# b^{4}_{-1})\otimes (\underline{S}(b^{1})_{0}\underline{\mathrm{d}}b^{4}_{0}\otimes 1)\\ 
    & = \left((\underline{S}(b^{1})_{-3}\triangleright S(b^{2}))\, (\underline{S}(b^{1})_{-2}\triangleright b^{3})\# \underline{S}(b^{1})_{-1}b^{4}_{-1}\right) \otimes (\underline{S}(b^{1})_{0}\underline{\mathrm{d}}b^{4}_{0}\otimes 1)\\ 
    & = \left( \underline{S}(b^{1})_{-2}\triangleright(\underline{S}(b^{2})b^{3})\# \underline{S}(b^{1})_{-1}b^{4}_{-1}\right) \otimes (\underline{S}(b^{1})_{0}\underline{\mathrm{d}}b^{4}_{0}\otimes 1)\\ 
    & = (\underline{S}(b^{1})_{-2}\triangleright 1_{B})\# \underline{S}(b^{1})_{-1}b^{2}_{-1})\otimes \underline{S}(b^{1})_{0}\underline{\mathrm{d}}b^{2}_{0}\otimes 1 \\ 
    & = (1\# \underline{S}(b^{1})_{-1}b^{2}_{-1})\otimes \underline{S}(b^{1})_{0}\underline{\mathrm{d}}b^{2}_{0}\otimes 1\\
    & = (1\# b^{1}_{-1}b^{2}_{-1})\otimes \underline{S}(b^{1}_{0})\underline{\mathrm{d}}b^{2}_{0}\\ 
    & =(1\# b_{-1})\otimes (\underline{S}(b_{0}^{1})\underline{\mathrm{d}}b_{0}^{2}\otimes 1)\\
    & = (1\# b_{-1})\otimes (\underline{\varpi}(b_{0})\otimes 1),
    \end{align*}
    from which we deduce
    \begin{align*}
    {}_{\#}\Delta(\underline{\varpi}(b_{0})\otimes S^{-1}(b_{-1})) & = {}_{\#}\Delta((\underline{\varpi}(b_{0})\otimes 1)(1\# S^{-1}(b_{-1}))   \\
    & = {}_{\#}\Delta((\underline{\varpi}(b_{0})\otimes 1)) \Delta_{\#}(1\# S^{-1}(b_{-1}))  \\
    & = \left((1\# b_{-1})\otimes (\underline{\varpi}(b_{0})\otimes 1)\right)\left((1\# S^{-1}(b_{-2})_{1})\otimes (1\# S^{-1}(b_{-2})_{2})\right) \\ 
    & = \left((1\# b_{-1}S^{-1}(b_{-2})_{1})\otimes (\underline{\varpi}(b_{0})\otimes S^{-1}(b_{-2})_{2})\right) \\ 
    & = (1\# 1)\otimes (\underline{\varpi}(b_{0})\otimes S^{-1}(b_{-1})).
    \end{align*}
    Thus, the map $\psi\colon \Lambda^{1}_{\uB}\otimes \Bbbk \to {}^{\mathrm{co}A}(\Lambda^{1}_{\uB}\otimes H)$ is well defined as a $\Bbbk$-linear map. It is immediate to check that $\psi^{-1}\circ \psi=\mathrm{id}_{\Lambda^{1}_{\uB}}$. 
    
    On the other hand, let  $\gamma\otimes k \in {}^{\mathrm{co}A}(\Lambda^{1}_{\uB}\otimes H)$.  We may write 
    \[
    \gamma\otimes k = \underline{\varpi}(S(h_{2})S(b_{-1})\triangleright b_{0})\otimes S(h_{1})S(b_{-2})h_{3},
    \]
    according to Equation \eqref{eq:splitting_formula}. We have
    \begin{align*}
   \psi\circ \psi^{-1}\circ \left(S(b_{-1}h_{1})\triangleright(\underline{\varpi}(b_0)\otimes h_{2})\right) & = \psi \circ \underline{\varpi}(S(b_{-1}h)\triangleright b_{0})\\ 
   & = \underline{\varpi}\left((S(b_{-1}h)\triangleright b_{0})_{0}\right) \otimes S^{-1}\left(S(b_{-1}h)\triangleright b_{0}\right)_{-1}\\ 
   & =\underline{\varpi}(S(b_{-3}h_{2})\triangleright b_{0}) \otimes S^{-1}\left(S(b_{-2}h_{3})\,b_{-1}\, S^{2}(b_{-4}h_{1})\right) \\ 
   & = \underline{\varpi}(S(b_{-3}h_{2})\triangleright b_{0}) \otimes S(b_{-4}h_{1})\, S^{-1}(b_{-1})b_{-2}\, h_{3} \\ 
   & = \underline{\varpi}(S(h_{2})S(b_{-1})\triangleright b_{0})\otimes S(h_{1})S(b_{-2})h_{3},
    \end{align*}
    and thus $\psi\circ\psi^{-1}=\mathrm{id}_{{}^{\,\mathrm{co}A}(\Lambda^{1}_{\uB}\otimes H)}$. 
\end{proof}
\end{proposition}

Together with Equation \eqref{equation:diagram_smash}, this result recovers the classical geometric picture: the coinvariant 1-forms on the total space $A := \uB \# H$ decompose into horizontal and vertical components. The former are identified with 1-forms on the base space $\uB$, while the latter correspond to 1-forms on the structure Hopf algebra $H$. This decomposition faithfully reproduces the classical differential geometric splitting of 1-forms on a principal bundle. Geometrically, the vertical forms encode the infinitesimal directions along the fibers, whereas the horizontal forms arise as pullbacks of forms from the base manifold. Ultimately, the interplay between the Radford–Majid biproduct structure and smash product calculus provides a genuine noncommutative analogue of this classical decomposition,  separating the bundle's coinvariant 1-forms into contributions from the homogeneous space and the group itself. Let us point out again that, classically, the homogeneous space obtained by extension of a Lie group $\textrm{G}$, is an example of reductive Klein geometry, meaning that the corresponding Lie algebra splits as a direct sum of $\mathrm{G}$-modules. Inspired by this observation we describe how the splitting obtained above is consistent with the category of (co)modules in which the braided Hopf algebra $\uB$ is considered and, in particular, showing that the Maurer--Cartan form $\varpi_{\#}\colon (B\#H)^{+}\to \Lambda^{1}_{A}$ is a morphism of right $H$-covariant right $H$-modules as expected. 

 In the next lemma we show the module and comodule structure on the horizontal forms of the bundle are compatible with the splitting in Proposition \ref{prop:iso_of_base_forms_with_horizontals}.
    
\begin{lemma}\label{lemma:right_comodule_and_module_structures_on_coinvariants}
The following statements hold.
\begin{enumerate}
    \item ${}^{\mathrm{co}A}(\Lambda^{1}_{\uB}\otimes H)$ is an object in $\mathrm{Mod}^{H}\cap\mathrm{Mod}_H$, with right action 
    \begin{equation}  \label{eq:action_on_coA} 
        \begin{aligned}
    \triangleleft_{\mathrm{co}A}\colon {}^{\mathrm{co}A}(\Lambda^{1}_{\uB}\otimes H)\otimes H & \to {}^{\mathrm{co}A}(\Lambda^{1}_{\uB}\otimes H),\\
            \gamma\otimes h\otimes k & \mapsto (1\# S(k_{1}))(\gamma\otimes h)(1\# k_{2}).
        \end{aligned}
    \end{equation}
    and with right coaction
    \begin{equation}
        \begin{aligned}
        \label{eq:coaction_on_coA} \Delta_{H,\mathrm{co}A}\colon {}^{\mathrm{co}A}(\Lambda^{1}_{\uB}\otimes H)& \mapsto {}^{\mathrm{co}A}(\Lambda^{1}_{\uB}\otimes H)\otimes H,\\ 
        \gamma\otimes h &\mapsto \gamma\otimes h_{1}\otimes h_{2}.
        \end{aligned}
    \end{equation}
    \item $\Lambda^{1}_{\uB}$ is an object in $\mathcal{YD}^{H}_{H}$, with right action 
    \begin{equation}\label{eq:action_on_Lambda1B}
        \begin{aligned}
\triangleleft_{\Lambda^{1}_{\uB}}\colon  \Lambda^{1}_{\uB}\otimes H & \to \Lambda^{1}_{\uB},\\ 
        \gamma \otimes h & \mapsto S(h)\triangleright \gamma. 
      \end{aligned}
    \end{equation}
    and right coaction 
    \begin{equation}
    \label{eq:coaction_on_Lambda1B} 
        \begin{aligned}
\Delta_{H,\Lambda^{1}_{\uB}}\colon  \Lambda^{1}_{\uB} & \to  \Lambda^{1}_{\uB}\otimes H,\\ 
        \underline{\varpi}(b)&\mapsto \underline{\varpi}(b_{0})\otimes S^{-1}(b_{-1}).
        \end{aligned}
    \end{equation}
\end{enumerate}
\begin{proof} In order:
\begin{enumerate}
\item Considering an element  of ${}^{\mathrm{co}A}(\Lambda^{1}_{\uB}\otimes H)$ (according to Equation \eqref{eq:splitting_formula}) we have 
\begin{align*}
\left(\varpi_{\#}|_{{}^{\mathrm{co}A(\Lambda^{1}_{\uB}\otimes H)}}(b\# h)\right) \triangleleft_{\mathrm{co}A}(1\# k)& =  \left(\varpi_{\#}|_{{}^{\mathrm{co}A(\Lambda^{1}_{\uB}\otimes H)}}\left((b\# h)\cdot_{\#}(1\# k)\right)\right) \\ 
& = \left(\varpi_{\#}|_{{}^{\mathrm{co}A(\Lambda^{1}_{\uB}\otimes H)}}(b\# hk)\right). 
\end{align*}
The axioms of an action are understood. Thus, Equation \eqref{eq:action_on_coA} provides a well defined action.  Moreover, the assignment in Equation \eqref{eq:coaction_on_coA} gives a well-defined $H$-coaction. Indeed,  ${}^{\mathrm{co}A}(\Lambda^{1}_{\uB}\otimes H)$ closes as a right $H$-comodule under $\Delta_{H,\mathrm{co}A}$, as 
\begin{align*}
\Delta_{H,\mathrm{co}A}(S(b_{-1}h_{2})\triangleright\underline{\varpi}(b_{0})\otimes S(b_{-2}h_{1})h_{3}) & = (S(b_{-1}h_{2})\triangleright\underline{\varpi}(b_{0})\otimes (S(b_{-2}h_{1})h_{3})_{1}\otimes (S(b_{-2}h_{1})h_{3})_{2} \\ 
& = (S(b_{-1}h_{3})\triangleright\underline{\varpi}(b_{0})\otimes S(b_{-2}h_{2})h_{4}\otimes S(b_{-3}h_{1})h_{5} \\ 
& = S(b_{-1}h_{2})\triangleright(\underline{\varpi}(b_{0})\otimes h_{3})\otimes S(b_{-2}h_{1})h_{4} \\ 
& = \varpi_{\#}|_{{}^{\mathrm{co}A}(\Lambda^{1}_{\uB}\otimes H)}(b_{0}\#h_{2}) \otimes S(b _{-1}h_{1})h_{3},
\end{align*}
and as coassociativity and counitality of $\Delta_{H,\mathrm{co}A}$ are obvious we are done. 
\item All the different flavors of categories of Yetter--Drinfeld modules are equivalent, and in particular the assignments in Equations \eqref{eq:action_on_Lambda1B} and \eqref{eq:coaction_on_Lambda1B} are precisely the module and comodule structures that realise $\Lambda^{1}_{\uB}$ as an object of $\mathcal{YD}^{H}_{H}$, starting from its given structure in ${}^{H}_{H}\mathcal{YD}$, and thus we are done. However, we still find it instructive to present an explicit realization of the $\mathcal{YD}^{H}_{H}$ structure on $\Lambda^{1}_{\uB}$. It is an immediate check that the assignment in Equation \eqref{eq:action_on_Lambda1B} defines an action on $\Lambda^{1}_{\uB}$. Moreover, the assignment in Equation \eqref{eq:coaction_on_Lambda1B} is in fact induced by the diagram 
\[
\begin{tikzcd}
{}^{\mathrm{co}A}(\Lambda^{1}_{\uB}\otimes H) \arrow[rr, "{\Delta_{H,\mathrm{co}A}}"] &  & {}^{\mathrm{co}A}(\Lambda^{1}_{\uB}\otimes H)\otimes H \arrow[dd, "(\psi^{-1}\otimes \mathrm{id})"] \\
                                                                                    &  &                                                                                                   \\
\Lambda^{1}_{\uB} \arrow[rr, "{\Delta_{H,\Lambda^{1}_{\uB}}}",dashed] \arrow[uu, "\psi"]       &  & \Lambda^{1}_{\uB}\otimes H                                                                         
\end{tikzcd},
\]
indeed 
\begin{align*}
\Delta_{H,\Lambda^{1}_{\uB}}(\underline{\varpi}(b)) & := (\psi^{-1}\otimes \mathrm{id})\circ \Delta_{H,\mathrm{co}A}\circ \psi (\underline{\varpi}(b))\\ 
& = (\psi^{-1}\otimes \mathrm{id})\circ \Delta_{H,\mathrm{co}A}\left(\underline{\varpi}(b_{0})\otimes S^{-1}(b_{-1})\right) \\ 
& = \psi^{-1}\left(\underline{\varpi}(b_{0})\otimes S^{-1}(b_{-1})\right)\otimes S^{-1}(b_{-2})\\ 
& = \underline{\varpi}(b_{0}) \otimes S^{-1}(b_{-1}).
\end{align*}

\noindent Thus, we have a well-defined right $H$-coaction on $\Lambda^{1}_{\uB}$.  Finally, we have 

\begin{align*}
\Delta_{H,\Lambda^{1}_{\uB}}(\gamma\triangleleft h) & = \Delta_{H,\Lambda^{1}_{\uB}}(S(h)\triangleright \underline{\varpi}(b)) \\ 
& = \Delta_{H,\Lambda^{1}_{\uB}}( \underline{\varpi}(S(h)\triangleright b)) \\ 
& = \underline{\varpi}((S(h)\triangleright b)_{0}) \otimes S^{-1}(S(h)\triangleright b)_{-1} \\ 
& = \underline{\varpi}(S(h)_{2}\triangleright b_{0}) \otimes S^{-1}(S(h)_{1}b_{-1}S(S(h)_{3})) \\ 
& = S(h_{2})\triangleright\underline{\varpi}( b_{0}) \otimes S^{-1}(S(h_{3})b_{-1}S^{2}(h_1)) \\
& = \underline{\varpi}(S(h_{2})\triangleright b_{0}) \otimes S(h_{1})S^{-1}(b_{-1})h_3 \\ 
& = \underline{\varpi}(b_{0})\triangleleft h_{2}\otimes S(h_{1})S^{-1}(b_{-1})h_{3} \\ 
& = \gamma_{0}\triangleleft h_{2}\otimes S(h_{1})\gamma_{-1}h_{3}, 
\end{align*}
and thus the Yetter--Drinfeld property holds. 
\end{enumerate}
\end{proof}
\end{lemma}

The last result turns out to be crucial to better understand the properties of the Maurer--Cartan form for the Radford--Majid biproduct.

\begin{proposition}\label{prop:varpi_smash_is_H_comodule_map}
   The map $\varpi_{\#}\colon (\uB\# H)^{+} \to \Lambda^{1}_{A}$ is a morphism in $\mathrm{Mod}^{H}\cap\mathrm{Mod}_{H}$, and in particular both the restrictions of $\varpi_{\#}|_{\uB^{+}\#H}$ and $\varpi_{\#}|_{1\#H^{+}}$ are. 
    \begin{proof}
    First of all, the map $\varpi_{\#}\big|_{\uB^{+}\# H}\colon \uB^{+}\# H \to {}^{\mathrm{co}A}(\Lambda^{1}_{\uB}\#H)$ is a morphism in $\mathrm{Mod}^{H}$. Indeed, we consider the relative right adjoint $H$-coaction on $\uB^{+}\#H$ as induced by the adjoint $A$-coaction via the Hopf algebra surjection $\pi\colon \uB \#H\to H$, namely 
        \begin{equation}
            \begin{aligned}
                \mathrm{coAd}_{H,\uB^{+}\#H}\colon \uB^{+}\# H &\to \uB^{+}\# H \otimes H, \\
                b\# h & \mapsto (b_{0}\# h_{2}) \otimes S(b_{-1}h_{1})h_{3},
            \end{aligned}
        \end{equation}
        which in particular restricts to a well-defined right $H$-coaction on $\uB^{+}\#H$. By a direct calculation we now find

        \begin{align*}
          (\varpi_{\#}\big|_{\uB^{+}\#H}\otimes \mathrm{id}_{H})\circ   \mathrm{coAd}_{H,\uB^{+}\#H}(b\#h)& = \varpi_{\#}\big|_{\uB^{+}\#H}(b_{0}\# h_{2})\otimes S(b_{-1}h_{1})h_{3} \\ 
          & = S(b_{-1}h_{3})\triangleright \underline{\varpi}(b_{0})\otimes S(b_{-2}h_{2})h_{4} \otimes S(b_{-3}h_{1})h_{5}\\ 
         & = \Delta_{\mathrm{co}A,H}\left(S(b_{-1}h_{1})_{1}\triangleright \underline{\varpi}(b_{0}) \otimes S(b_{-1}h_{1})_{2}h_{2}\right) \\ 
          & = \Delta_{\mathrm{coA},H} \circ\varpi_{\#}\big|_{\uB^{+}\#H}(b\#h).
        \end{align*}
        Arguing similarly, one can promptly show that $\varpi_{\#}|_{1\# H^{+}}\colon 1\# H^{+} \to 1\# \Lambda^{1}_{H}$ is a right $H$-comodule map, where $1\# H^{+}$ and $1\# \Lambda^{1}_{H}$ are endowed with the same comodule structure outlined above. Moreover, right $\uB\# H$-linearity of $\varpi_{\#}$ is automatic for any left covariant FODC on $\uB\#H$ and thus, in particular, right $H$-linearity follows.
    \end{proof}
    \end{proposition}

\noindent Finally, we present the main result of the section. Given the module and comodule structures outlined throughout the section, we have a splitting of the coinvariant $1$-forms on $A:=B\# H$ in the category $\mathrm{Mod}^{H}\cap\mathrm{Mod}_{H}$. 
\begin{theorem}
The space of left coinvariant $1$-forms $\Lambda^{1}_{A}:={}^{\mathrm{co}A}\Omega^{1}_{\#}(A)$ splits, in $\mathrm{Mod}^{H}\cap\mathrm{Mod}_{H}$, into the direct sum 

    \begin{equation}
    \Lambda^{1}_{A}\cong (\Lambda^{1}_{\uB}\otimes 1) \oplus (1\otimes \Lambda^{1}_{H}).
    \end{equation}
    
    \begin{proof}
           First, we show $\psi^{-1}\colon {}^{\mathrm{co}A}(\Lambda^{1}_{\uB}\otimes H) \to \Lambda^{1}_{\uB}$ is a morphism of right $H$-modules with respect to the right $H$-actions outlined in Lemma \ref{lemma:right_comodule_and_module_structures_on_coinvariants}. Indeed, let 
    \[
    \gamma\otimes c =S(b_{-1}h_{1})\triangleright(\underline{\varpi}(b_{0})\otimes h_{2})=\varpi_{\#}|_{{}^{\mathrm{co}A}(\Lambda^{1}_{\uB}\otimes H)}(b\#h)\in {}^{\mathrm{co}A}(\Lambda^{1}_{\uB}\otimes H)
    \] 
    for $(b\# h)\in \ker\varepsilon_{\#}$, and let $k\in H$. We have 
    \begin{align*}
    \psi^{-1}((\gamma\otimes c)\triangleleft_{\mathrm{co}A} k)) & =  \psi^{-1}((\gamma\otimes c)\triangleleft_{\mathrm{ad}}  (1\# k)) \\ 
    & = \psi^{-1} \left(\varpi_{\#}|_{{}^{\mathrm{co}A}(\Lambda^{1}_{\uB}\otimes H)}(b\#h)\triangleleft_{\mathrm{ad}}  (1\# k)\right) \\ 
    & = \psi^{-1}\left(\varpi_{\#}|_{{}^{\mathrm{co}A}(\Lambda^{1}_{\uB}\otimes H)}(b\#hk)\right)\\ 
    & = \psi^{-1}\left( S(b_{-1}h_{1}k_{1})\triangleright\left(\underline{\varpi}(b_{0})\otimes h_{2}k_{2} \right)\right) \\ 
    & = S(b_{-1}h\, k)\triangleright \underline{\varpi}(b_{0})\\ 
    & = S(k)\triangleright \left(S(b_{-1}h)\triangleright \underline{\varpi}(b_{0})\right)\\ 
    & = \left(S(b_{-1}h)\triangleright \underline{\varpi}(b_{0})\right)\triangleleft_{\mathrm{\Lambda^1}}k\\ 
     & = \psi^{-1}(\gamma\otimes c)\triangleleft_{\Lambda^1}k.
    \end{align*}
     Moreover, $\psi^{-1}\colon {}^{\mathrm{co}A}(\Lambda^{1}_{\uB}\otimes H) \to \Lambda^{1}_{\uB}$ is also a morphism of right $H$-comodules. Indeed the  diagram  
     \[
     \begin{tikzcd}
{}^{\mathrm{co}A}(\Lambda^{1}_{\uB}\otimes H) \arrow[dd, "\psi^{-1}"] \arrow[rr, "\Delta_{H,\mathrm{co}A}"] &  & {}^{\mathrm{co}A}(\Lambda^{1}_{\uB}\otimes H)\otimes H \arrow[dd, "(\psi^{-1}\otimes \mathrm{id})"] \\
                                                                                                                                           &  &                                                                                                   \\
\Lambda^{1}_{\uB} \arrow[rr, "{\Delta_{H,\Lambda^{1}_{\uB}}}"]                                                                                 &  & \Lambda^{1}_{\uB}\otimes H                                                                         
\end{tikzcd}
\]
commutes by the very definitions of the coactions $\Delta_{H,\Lambda^{1}_{\uB}}$ and $\Delta_{H,\mathrm{co}A}$ supplied in Lemma \ref{lemma:right_comodule_and_module_structures_on_coinvariants}. 
Finally, the claim now follows from Proposition \ref{prop:varpi_smash_is_H_comodule_map}. 
    \end{proof}
\end{theorem}

As observed before, we have that the domain of the Maurer--Cartan form $\varpi_\#$ decomposes as $A^+=B^+\# H\oplus 1\# H^+$. The first component is of particular interest, since it encodes the non-trivial interplay between $\underline{B}$ and $H$ induced by the smash product structure. 

\begin{proposition}\label{prop:chi_is_H_comodule}
We have a surjective morphism in $\mathrm{Mod}^{H}\cap\mathrm{Mod}_{H}$ given by
\begin{equation}\label{eq:definition_of_chi}
   \begin{aligned}
       \chi\colon \underline{B}^{+}\#H & \to \underline{B}^{+},\\ 
       b\# h &\mapsto S(b_{-1}h)\triangleright b_{0}. 
       \end{aligned}
   \end{equation}
    Moreover, the diagram 
    \[
        \begin{tikzcd}
        \underline{B}^{+}\# H \arrow[rr, "\varpi_{\#}|_{\underline{B}^{+}\# H}"] \arrow[dd, "\chi"] &  & {}^{\mathrm{co}A}(\Lambda^{1}_{\underline{B}}\otimes H) \arrow[dd, "\psi^{-1}"] \\
                                                                                                &  &                                                                            \\
        \underline{B}^{+} \arrow[rr, "\underline{\varpi}"]                                                      &  & \Lambda^{1}_{\underline{B}}                                               
        \end{tikzcd}
        \] 
        commutes in $\mathrm{Mod}^{H}\cap \mathrm{Mod}_{H}$. 
    \begin{proof}
        The above diagram commutes in ${}_{\Bbbk}\mathrm{Vec}$, as it can readily be checked, and thus $\chi$ is surjective. More explicitly we have
        \[
        \chi\left(b_{0}\# S^{-1}(b_{-1})\right)  =  S\left(b_{-1}S^{-1}(b_{-2})\right)\triangleright b_{0} =b 
        \]
        for any $b\in \uB^{+}$. For the equivariance and covariance properties of $\chi$, consider the module and comodule structures on $\underline{B}^{+}$ and $\underline{B}^{+}\# H$, respectively, which follow readily from Lemma \ref{lemma:right_comodule_and_module_structures_on_coinvariants} and Proposition \ref{prop:varpi_smash_is_H_comodule_map}. We have 
        \begin{align*}
        \chi((b\#h)\cdot_{\#}(1\# k)) & =\chi(b\#hk) \\ 
        & = S(b_{-1}hk)\triangleright b_{0} \\ 
        & = S(k)(S(b_{-1}h)\triangleright b_{0})\\ 
        & = (S(b_{-1}h)\triangleright b_{0})\triangleleft k,
        \end{align*}
        and thus $\chi$ is right $H$-linear. Right $H$-colinearity follows as well, since 
        \begin{align*}
        \Delta_{\uB^{+}}\circ \chi(b\#h) & =\Delta_{\uB^{+}}(S(b_{-1}h)\triangleright b_{0}) \\ 
        & = (S(b_{-1}h)\triangleright b_{0})_{0}\otimes S^{-1}((S(b_{-1}h)\triangleright b_{0})_{-1}) \\ 
        & = S(b_{-3}h_{2})\triangleright b_{0}\otimes S^{-1}(S(b_{-2}h_{3})b_{-1}S^{2}(b_{-4}h_{1})) \\ 
        & = S(b_{-1}h_{2})\triangleright b_{0} \otimes S(b_{-2}h_{1})h_{3}\\ & = (\chi \otimes \mathrm{id})(b_{0}\# h_{2} \otimes S(b_{-1}h_{1})h_{3}) \\ 
        & = (\chi\otimes \mathrm{id})\circ\mathrm{coAd}_{H,\uB^{+}\#H}(b\#h).
        \end{align*} 
        Finally, from Proposition \ref{prop:varpi_smash_is_H_comodule_map} and Lemma \ref{lemma:right_comodule_and_module_structures_on_coinvariants}, it follows that the above is a commuting diagram in $\mathrm{Mod}^{H}\cap\mathrm{Mod}_{H}$.
        \end{proof}
    \end{proposition}

\noindent We now give an explicit presentation of the ideal $I_\#$ in terms of the ideals $I_{\uB}$ and $I_{H}$. 

\begin{lemma}
 The classifying ideal $I_{\#}$ of any right $H$-covariant and left $\underline{B}\#H$-covariant smash product calculus is given by 
 \begin{equation}\label{eq:ideal_smash}
        I_{\#}:=\left\{b\#h \in (\uB\#H)^+ \mid S(b_{-1}h)\triangleright b_{0}^{+}\in I_{\uB}, \, \underline{\varepsilon}(b)h^{+}\in I_{H}\right\}.
    \end{equation}
    \begin{proof}
        Consider the commutative diagram in Proposition \ref{prop:chi_is_H_comodule}. The ideal $I_{\#}$ is the kernel of $\varpi_{\#}\colon (\uB\# H)^{+} \to \Lambda^{1}_{\#}$, and as  $\psi$ is an isomorphism in $\mathrm{Mod}^{H}$ we have $I_{\#}:=\ker\varpi_{\#}=\ker(\psi^{-1}\circ\varpi_{\#})$. The condition in Equation \eqref{eq:ideal_smash} can indeed be obtained by Equation \eqref{eq:splitting_formula} together with the definition of $\psi$ itself.
        
        We now argue the previous statement more explicitly. Consider a left covariant FODC $\Omega^{1}(\uB)$ on $\uB$ and its corresponding classifying right ideal $I_{\uB}$, which is an object in ${}^{H}_{H}\mathcal{YD}$. We prove that the preimage of $I_{\uB}$ via $\chi\colon \uB^{+}\# H \to \uB^{+}$ is a right $\uB\#H$-ideal in $\uB^{+}\#H$.  Given any element $S(b_{-1}h)\triangleright b_{0}\in I_{\uB}$, we have 
        
        \begin{equation}
            \begin{split}
                {}_{H}\Delta(S(b_{-1}h)\triangleright b_{0}) & = S(b_{-1}h)_1(b_0)_{-1}S(S(b_{-1}h_1)_3) \otimes S(b_{-1}h)_2 \triangleright (b_0)_0\\
                &=S(h_3)S((b_{-1})_3)(b_{-1})_4S^2((b_{-1})_1h_1) \otimes S((b_{-1})_2h_2) \triangleright (b_0)_0 \\
                &= S\left(S(b_{-2}h_{1})h_{3}\right)\otimes S(b_{-1}h_{2})\triangleright b_{0}  
            \in H\otimes I_{\uB}.
            \end{split}
            \label{eq:coactionimageofchi}
        \end{equation}
        Moreover, given any $b\#h\in I_{\#}$ and $c\#k\in \uB\# H$ we find that  
        \begin{align*}
        \chi\left((b\# h)\cdot_{\#}(c\#k)\right) & = \chi( b(h_{1}\triangleright c) \# h_{2}k ) \\ 
        & = S\left( b_{-1}(h_{1}\triangleright c)_{-1}h_{2}k\right)\triangleright b_{0}(h_{1}\triangleright c)_{0}\\ 
        & = S(b_{-1}h_{1}c_{-1}k)\triangleright \left(S(b_{-1}) \triangleright (b_{0}(h_{2}\triangleright c_{0})\right) \\ 
         & = S(b_{-1}h_{1}c_{-1}k)\triangleright\left( \left(S(b_{-1}) \triangleright b_{0}\right)( S(b_{-2})h_{2}\triangleright c_{0})\right) \\ 
         & = S(c_{-1}k)S(h_{1})\triangleright \left(\left(S(b_{-1}) \triangleright b_{0}\right)( S(b_{-2})h_{2}\triangleright c_{0})\right) \\ 
         & = S(c_{-1}k)\triangleright \big(\underbrace{\left(S(b_{-1}h_{2}) \triangleright b_{0}\right)}_{\in I_{\uB}}( S(b_{-2}h_{1})h_{3}\triangleright c_{0})\big). \\ 
        \end{align*}
        Thus, the latter is in $I_{\uB}$, as $I_{\uB}$ is a left $H$-module right $\uB$-ideal.  Moreover, 
        \begin{align*}
       \varpi_{H}\circ(\underline{\varepsilon}\otimes \mathrm{id}_{H}) \left((b\# h)\cdot_{\#} (c\#k)\right) & = (\underbrace{\underline{\varepsilon}(b)h}_{I_{H}}) \,\varepsilon(c)k\in I_{H},
        \end{align*}
        and therefore the claim holds. The other inclusion is trivial.
    \end{proof}
\end{lemma}

Observe, as an immediate consequence, that the universal calculus $\Omega^1_u(\underline{B}\# H)$ on $\underline{B}\# H$ does not coincide with the smash product calculus constructed with the universal calculi on $\underline{B}$ and $H$. Moreover, we observe the following.

\begin{corollary}
    The classifying ideal $I_{\#}$ of any smash product calculus over $\uB\#H$ as above satisfies the following conditions.
    \begin{enumerate}
        \item \begin{equation}
            \pi(I_{\#}) = I_H,
        \end{equation}
        where $\pi \colon \underline{B}\#H \to H, b\#h \mapsto \underline{\varepsilon}(b)h$ and $I_H$ is the classifying ideal of the bicovariant calculus $\Omega^1(H) \cong H \otimes H^+/I_H$ over $H$;
        \item \begin{equation}
            \underline{\mathrm{coAd}}_R^H(I_{\#}) \subseteq I_{\#} \otimes H,
        \end{equation}
        where $\underline{\mathrm{coAd}}_R^H(b\#h) := (b\#h)_2 \otimes \pi\left(S_{\#}((b\#h)_1)(b\#h)_2\right)$ for all $b\#h \in \underline{B}\#H$.
    \end{enumerate}
    \label{cor:forMajidLemma546}
\end{corollary}
\begin{proof}
The first point immediately follows from Equation \ref{eq:ideal_smash}. Moreover, notice that by construction $I_{\#}$ appears as the kernel of $\varpi_{\#}\colon (\uB\# H)^{+}\to \Lambda^{1}_{A}$, which is a morphism of right $H$-comodules according to Proposition \ref{prop:varpi_smash_is_H_comodule_map}. In particular, the right $H$-comodule structure on $(\uB \#H)^{+}$ is precisely given in terms of the relative right adjoint $H$-coaction, which implies that the second point holds. 
\end{proof}

 \begin{remark}\label{rmk:soldering_form_and_ideal}
    An explicit description of the Maurer--Cartan form on $A$ is now given by 
    \begin{equation} 
    \varpi_\#= \psi\circ \underline{\varpi}\circ \chi+\underline{\varepsilon}\otimes \varpi_H.
    \end{equation}
    Moreover, considering the commutative diagram in Proposition \ref{prop:chi_is_H_comodule}, since $I_{\uB}$ is the kernel of $\underline{\varpi}$ and $\chi$ is a morphism of right $H$-comodules, we have that $\uB^+/I_{\uB} \cong \uB^+\#H/(I_\#\cap\uB^+\#H)$ as right $H$-comodules. 
    \qed 
\end{remark}

In the last Section we have exploited the rich algebraic structure of $\uB \# H$, which arises from the assumption that $\uB$ is a braided Hopf algebra in the Yetter–Drinfeld category ${}^{H}_{H}\mathcal{YD}$, where $H$ is a Hopf algebra in ${}_{\Bbbk}\mathrm{Vec}$. This structure provides a coherent interplay between the braided and classical aspects of the construction, and is ultimately responsible for the geometric features exhibited by the resulting model. To make these structural relations more transparent, we give a diagrammatic representation of the various algebraic components of the model:

\[
\begin{tikzcd}[column sep=0.7em, row sep=1.6em]
                                                                                                                                        & \Lambda^{1}_{H}                                                                                              &  &                                                                                  & H^{+}/I_{H} \arrow[lll, "\cong" description]                                                    \\
{\Lambda^{1}_{H,u}} \arrow[ru, two heads]                                                                                               &                                                                                                              &  & H^{+} \arrow[ru, two heads] \arrow[lll, "\cong" description]                     &                                                                                                 \\
                                                                                                                                        & {}^{\mathrm{co}A}(\Lambda^{1}_{\uB}\otimes H) \oplus 1\otimes \Lambda^{1}_{H} \arrow[dd, dashed] \arrow[uu, dashed] &  &                                                                                  & (\uB^{+}\# H \oplus 1\# H^{+})/I_{\#} \arrow[dd] \arrow[uu] \arrow[lll, "\cong" description, dashed] \\
{{}^{\mathrm{co}A}(\Lambda^{1}_{\uB,u}\otimes H) \oplus 1\otimes \Lambda^{1}_{H,u}} \arrow[ru, two heads, dashed] \arrow[dd] \arrow[uu] &                                                                                                              &  & \uB^{+}\# H \oplus 1\# H^{+} \arrow[ru, two heads] \arrow[uu] \arrow[dd] \arrow[lll]  &                                                                                                 \\
                                                                                                                                        & \Lambda^{1}_{\underline{B}}                                                                                  &  &                                                                                  & \underline{B}^{+}/I_{\underline{B}} \arrow[lll, "\cong" description, dashed]                    \\
{\Lambda^{1}_{\underline{B},u}} \arrow[ru, two heads, dashed]                                                                           &                                                                                                              &  & \underline{B}^{+} \arrow[lll, "\cong" description] \arrow[ru, two heads] &                                                                                                
\end{tikzcd}
\]
    
\section{Quantum framings and quantum $\textrm{G}$-structures}
\label{section:framings_and_G_structures}
In noncommutative geometry there is, in general, no intrinsic notion of a frame or a frame bundle. However, in classical differential geometry the frame bundle is equivalent to a principal $\rGL(n)$-bundle equipped with a horizontal, $\rGL(n)$-equivariant, and $\Bbbk^n$-valued $1$-form, inducing an isomorphism between the tangent bundle and the associated bundle obtained by the standard representation of $\textrm{GL}(n)$ on $\Bbbk^n$. Similarly, given a $\rG$-structure $\pi\colon P\to M$, the standard representation of $\rGL(n)$ restricts to a representation of $\rG \subset \rGL(n)$ on $\Bbbk^n$. It follows that the tangent bundle can be recovered as the associated vector bundle $\mathrm{T}M \cong P\times_{\rG}\Bbbk^n$, where the isomorphism is induced by the solder form (more details on this classical construction can be found in Appendix \ref{section:classicalpicture}). This observation suggests a more intrinsic point of view, which is the one that we generalise to the noncommutative setting: rather than starting from the tangent bundle and constructing a frame bundle, we considered the principal bundle together with the solder form as the fundamental geometric data and we recovered the tangent bundle as an associated vector bundle. More in detail, we consider the following notion (see e.g. \cite[Example~5.67]{bm}).

\begin{definition}
Let $V$ be a finite-dimensional $\Bbbk$-vector space equipped with a left $\rG$-action, and let $\pi\colon P\to M$ be a principal $\rG$-bundle. A \emph{frame resolution} of the tangent bundle consists of a horizontal, $\rG$-equivariant $V$-valued $1$-form  $\theta\in\Omega^1(P,V)$ inducing an isomorphism of vector bundles $\mathrm{T}M \cong P\times_{\rG}V$. 
\label{def:frameresolution}
\end{definition}

In classical differential geometry, when considering an $n$ dimensional manifold, the canonical choice is $V=\Bbbk^n$ while in the quantum setting the vector space $V$ has to be considered as part of the data of the mathematical structure, even if, as we will discuss in the next section, in certain setting this data is intrinsic in some sense.

In this Section we thus introduce the notion of a quantum framing; we then focus our analysis on the quantum Maurer--Cartan form studied in Section \ref{subsection:Maurer_Cartan}, and use it to define and study quantum $\rG$-structures on the quantum principal bundles arising from Radford--Majid biproducts.
In Section \ref{section:framings} we show that such a quantum principal bundle is naturally equipped with a frame resolution induced by the quantum Maurer--Cartan form and we compare this result with the picture developed in \cite{bm}.
In Section \ref{section:G_structures} we introduce the notion of a quantum $\mathrm{G}$-structure, and study its basic properties in terms of reductions of the underlying quantum principal bundle, also giving explicit examples of $\mathrm{G}$-structures.

\subsection{Quantum frame resolutions}
\label{section:framings}

In order to generalise the notion of frame resolution to the algebraic setting we need then a consistent notion of what the noncommutative analogue of an associated vector bundle is. 

Let $B=A^{\mathrm{co}H} \subseteq A$ be a quantum principal bundle and let $W$ be a left $H$-comodule algebra with coaction ${}_W\Delta \colon W \to H \otimes W$. Recall the definition of the balanced cotensor product: 
 \begin{equation}
    E := A \square_{H} W = \big\{ a \otimes w \in A \otimes W \; \big| \; a_0\otimes a_1 \otimes w = a \otimes  w_{-1}\otimes w_{0} \big\}.
\label{eq:quantumassociatedvectorbundle}
    \end{equation}
  The latter is a left $B$-module with respect to the $B$-action defined by the multiplication on the first factor. 
\noindent It is straightforward to check that $A \square_{H} W$ is a subalgebra of $A \otimes W$, and moreover that $B$ is a subalgebra of $A \square_{H} W$.
Alternatively, we may consider the following construction.
\begin{definition}
    Let $B=A^{\mathrm{co}H} \subseteq A$ be a quantum principal bundle and let $V$ be a right $H$-comodule with coaction $\Delta_V \colon V \to V \otimes H$. The tensor product $A \otimes V$ is naturally endowed with a right $H$-comodule structure given by
   \begin{equation}  
        \label{eq:coactionforPotimesV}
   \begin{aligned}
      \Delta_{A \otimes V} \colon A \otimes V &\to A \otimes V \otimes H,\\
            a \otimes v \,&\mapsto a_0 \otimes v_0 \otimes a_1v_1,
        \end{aligned}
    \end{equation}
    for any $a \in A$ and $v \in V$.
    We then consider the left $B$-module
    \begin{equation}
    \label{eq:quantum_associated_bundle}
        \mathcal{E} := (A \otimes V)^{\mathrm{co}H},
    \end{equation}
    which we call \emph{quantum associated bundle} to the QPB $B:=A^{\mathrm{co}H}\subseteq A$. Sometimes, we denote the latter by $\mathcal{E}(B, V, H)$ when needed, or simply by $\mathcal{E}$. 
    \label{def:associatedquantumvectorbundleHop}
\end{definition}      
As throughout the manuscript we always assume any Hopf algebra $H$ to have an invertible antipode, we can induce a right $H$-comodule structure on any left $H$-comodule $V$ by defining $\Delta_V(v)=v_{0}\otimes S^{-1}(v_{-1}) $.  Under this assumption, the  notions of quantum associated bundle and balanced cotensor product are isomorphic, namely $(A\otimes V)^{\mathrm{co}H}\cong A\square_H V$ as left $B$-modules.

Consider now a right $H$-covariant first order differential calculus $(\Omega^1(A), \mathrm{d})$  on $A$, and denote by $(\Omega^1(B),\mathrm{d})$ the pullback calculus on $B$ induced by the canonical injection. We denote the horizontal forms by $A\Omega^1(B)$, specialising our attention to the cases in which it is an $A$-bimodule. In the quantum setting the notion of horizontal and equivariant $1$-forms is captured by the \emph{right strongly tensorial} $1$-forms, namely right $H$-comodule maps $ V \to A \Omega^1(B)$. Following \cite{BrzMaj} we then state the following result, yielding the noncommutative generalisation of the classical isomorphism between certain associated vector bundles to the frame bundle and the tangent space of the manifold. 
\begin{proposition}We have a bijective correspondence: 
\[
\left\{
\begin{array}{c}
\emph{Right} \\
\emph{strongly tensorial} \\
\emph{1-forms} \ \theta
\end{array}
\right\}
\quad\Longleftrightarrow\quad
\left\{
\begin{array}{c}
\emph{left $B$-module maps} \\
\cdot \circ (\mathrm{id}_A \otimes \theta) \colon \mathcal{E}
\longrightarrow
\Omega^1(B)
\end{array}
\right\}.
\]
    \label{pro:thecorrespondence2}
\end{proposition}

In light of the last proposition, we define a quantum frame resolution of $\Omega^1(B)$, \emph{i.e.}, the quantum analogue of the frame resolution of the tangent bundle, following \cite{Majid_braided_group_riemannian}.
\begin{definition}\label{def:QFR}
    A \emph{quantum frame resolution} of $(B, \Omega^1(B))$ is the datum of:
    \begin{enumerate}
        \item a quantum principal bundle $B := A^{\mathrm{co}H} \subseteq A$, 
        \item a  right $H$-comodule $(V,\Delta_V)$,
        \item a right strongly tensorial form $\theta \colon V \to A\Omega^1(B)$ such that, denoting with $\triangleright\colon A\otimes \Omega^{1}(A)\to \Omega^{1}(A)$ the $A$-action on $\Omega^{1}(A)$, the map 
    \begin{equation}
    \begin{aligned}
        s_{\theta} := \cdot \circ (\mathrm{id}_A \otimes \theta) \colon \mathcal{E}(B,V,H) & \to \Omega^1(B),\\    
        a \otimes v & \mapsto a\triangleright \theta(v),
     \label{eq:toevidentiatethesthetamap}
     \end{aligned}
    \end{equation}
    is an isomorphism in ${}_{B}\mathrm{Mod}$.
    \end{enumerate}
    We denote such a quantum frame resolution by the tuple $(A,H,V,\theta)$.
    \label{def:quanutumframeresolution}
\end{definition}

We now specialise to the setting studied in Section \ref{section:smash_product_calculi}, where $A=\uB\#H$, for $\uB\in{}^{H}_{H}\mathcal{YD}$ a braided Hopf algebra. Independently from the last construction, we observe that, in this framework, a canonical and natural isomorphism of the quantum associated vector bundle $E=A\square_{H}W$ with the base 1-forms always arises as a consequence of choosing as $W$ the object $\uB^{+}/I_{\uB}\in {}^{H}_{H}\mathcal{YD}$. 

\begin{proposition}\label{prop:frame_resolution}
    The map   
    \begin{equation}
    \label{eq:definition_of_frame_resolution}
    \begin{aligned}
        \mathfrak{s}\colon A\square_{H} \left(\uB^{+}/I_{\uB}\right) & \to \Omega^{1}(\uB),\\ 
        (b\# h)\otimes [c] & \mapsto \varepsilon_{H}(h)b\,\underline{\varpi}([c]) .
        \end{aligned}
    \end{equation}
    is an isomorphism in ${}_{\uB}\mathrm{Mod}$.
\begin{proof}
The claim holds as as a consequence of the Takuchi equivalence of relative Hopf modules (see \cite{Takeuchi}, and \cite{Reamonn} for a more recent account). More precisely, given a quantum homogeneous space $B\subseteq A$ with structure group $H$, there is an equivalence of categories ${}_{B}^{A}\mathrm{Mod}_{B}\cong {}^{H}\mathrm{Mod}_{B}$. Specialising to our case of study, we have that the left covariant FODC $\Omega^{1}(\uB)$ is naturally an object in ${}_{\uB}^{\uB\# H}\mathrm{Mod}_{\uB}$ and the inverse of the unit of the adjunction is given by Equation \eqref{eq:definition_of_frame_resolution}. More explicitly, the inverse $\mathfrak{s}^{-1}$ is essentially given by the Hopf--Galois map of the bundle, i.e., 
\begin{equation}
    \begin{aligned}\mathfrak{s}^{-1}\colon \Omega^{1}(\uB) & \to A\square_{H}\uB^{+}/I_{\uB}\\ 
    b\,\mathrm{d}c &\mapsto (b\# 1)\cdot_{\#}(c\# 1)_{1} \otimes (c\# 1)_{2}.
    \end{aligned}
\end{equation}
To see explicitly that the latter is indeed the inverse of $\mathfrak{s}$ in ${}_{\uB}\mathrm{Mod}$, consider the composition and deduce that 
\[  \mathfrak{s}^{-1}\circ \mathfrak{s}(b\#h\otimes c) = \varepsilon_{H}(h)b\# c_{-1}\otimes c_{0},
\]
then use that $\varepsilon_{H}(h)c_{-1}\otimes c_{0} = c\otimes h$, as we started from the balanced cotensor product over $H$. The other direction is obvious, and as $\mathfrak{s}\in{}_{\uB}\mathrm{Mod}$ its inverse is also a morphism in the same category. 
\end{proof}
\end{proposition}

\begin{remark}\label{rem:solderingremark}
    We show that the notion of frame resolution discussed in \cite{BrzMaj} coincides with the one proposed in Equation \eqref{eq:definition_of_frame_resolution}, once a coherent \emph{soldering form} and right $H$-comodule structure is chosen on $\uB^{+}/I_{\uB}$.  
    
    We consider the quantum associated vector bundle $(A\otimes V)^{\mathrm{co}H}$, where $A=\uB\#H$ and the right $H$-comodule structure on $V:=\uB^{+}/I_{\uB}$ is given by $\Delta_{\uB^{+}/I_{\uB}}(b)=b_{0}\otimes S(b_{-1})$. Let $\Omega^1_\#(\uB\#H)$ be a right $H$-covariant and left $\uB\#H$-covariant smash product calculus on $\uB\#H$, with classifying ideal $I_\#$, and let $\Omega^1_\#(\uB\#1)$ be the pullback calculus to the homogeneous space $\uB\#1$. Furthermore, define the map 
    \[
    \begin{aligned}
    \theta \colon \uB^{+}/I_{\uB} &\to A\,\Omega^{1}(\uB), \\ 
     [b] & \mapsto S(b_{-1})\triangleright \left( \underline{\varpi}([b_{0}]) \otimes  1\right),
    \end{aligned}
    \]
    and consider the $\uB$-module map 
    \[ 
    \begin{aligned}
    s_{\theta}\colon (A\otimes (\uB^{+}/I_{\uB}))^{\mathrm{co}H} & \to \Omega^1(\uB), \\ 
    (b\# h,[c]) & \mapsto  (b\# h)\,\theta([c]). 
    \end{aligned}
    \]
    As $b\#h\otimes [c] \in (A\otimes (\uB^{+}/I_{\uB}))^{\mathrm{co}H}$, we have 
    \begin{align*}
    b\# h \otimes c\#1 \otimes 1\#1 & = b\# h_{1} \otimes (c_{0}\# 1) \otimes h_{2}\,S(c_{-1})\\  \varepsilon(h)b\otimes c\# 1 \otimes 1\#1 &= b\otimes c_{0}\# 1 \otimes h\, S(c_{-1})  \\ 
    \varepsilon(h)b\otimes c\# 1 \otimes 1\#1 \otimes 1\#1 &= b\otimes c_{0}\# 1 \otimes h_{1}\, S(c_{-1}) \otimes h_{2}\, S(c_{-2}),
    \end{align*}
    where we first applied the counit on the second tensor factor, and then applied the coproduct on the last tensor factor. We now obtain 
    \begin{align*}
    s_\theta(b\# h\otimes [c]) & = b\, \underline{\varpi}(h_{1}S(c_{-1})\triangleright [c_{0}] )\otimes  h_{2}S(c_{-2}) \\ 
    & = \varepsilon_{H}(h) b\,\underline{\varpi}([c])\\
    & = \mathfrak{s}(b\#h\otimes [c]). 
    \end{align*}
    Finally, notice that in \cite[Theorem 5.77]{bm} there is a canonical choice of the right $H$-comodule $V$, for any quantum homogeneous principal bundle once fixed a differential structure. In particular, in our case, this would be given by the quotient $\uB^+\#1/(I_\# \cap \uB\#1)$. However, one verifies that 
    \[
        \begin{split}
            \uB^+\#1/(I_\#\cap\uB\#1) &\to \uB^+\#H/(I_\#\cap \uB^+\#H),\\
            [b\#1] &\mapsto [S(b_{-1})_1\triangleright b_0\#S(b_{-1})_2],\\
            [S^{-1}(h) \triangleright b\#1] &\mapsfrom [b\#h],
        \end{split}
    \]
    provides an isomorphism in $\mathrm{Mod}^{H}\cap \mathrm{Mod}_{H}$ with respect to the usual $H$-(co)module structures aforementioned. In addition, the regularity conditions in \cite[Lemma 5.46]{bm} are automatically satisfied, as already discussed in Corollary \ref{cor:forMajidLemma546}.
    \qed 
\end{remark}

\subsection{Quantum reductions and quantum $\textrm{G}$-structures}
\label{section:G_structures}
In the classical picture, as discussed in more detail in Appendix \ref{section:classicalpicture}, a G-structure is defined as a reduction of the frame bundle, or, equivalently, as a principal  G-bundle $P \to M$ together with a horizontal and equivariant $1$-form $\Theta \in \Omega^1_{\mathrm{hor}}(P,\mathbb{R}^n)^\mathrm{G}$. In particular, the soldering form $\Theta$ is obtained as the pullback of the canonical form in Equation \ref{eq:definitioncanonicalform} along the bundle reduction. We now present an analogous construction in noncommutative differential geometry. 

Consider a principal $H$-comodule algebra $B\subseteq A$ with right $H$-coaction given by $\Delta_{A,H}\colon A \to A \otimes H$. Given a Hopf ideal $J\subset H$, with corresponding Hopf algebra surjection $\pi \colon H \to H_0 := H/J$, we observe that $A$ becomes an $H_0$-comodule algebra with respect to the right $H_0$-coaction $\Delta_{A,H_0} := (\mathrm{id}_A \otimes \pi) \circ \Delta_{A,H}$.  Moving along the lines of  \cite{pagani,GiovanniThesis,gunther,hajac}, we give the following definition.

\begin{definition}
Let $B := A^{\mathrm{co}H} \subseteq A$ be a principal $H$-comodule algebra and $\pi \colon H \to H_0:= H/J$ a surjective Hopf algebra morphism as before. 
Let $B_0=A_0^{\mathrm{co}H_0}\subseteq A_0$ be a principal $H_{0}$-comodule algebra. 
We say that $A_0$ is an \emph{algebraic reduction} or \emph{quantum reduction} of $A$, if there exists a surjective morphism of $H_0$-comodule algebras $\phi \colon A \to A_0$ such that
\begin{equation}
    \phi|_B \colon B \stackrel{\cong}{\longrightarrow} B_0
    \label{eq:restrictionofphi}
\end{equation}
is an isomorphism. Note that $A$ carries the induced $H_0$-coaction $\Delta_{A,H_{0}} := (\mathrm{id}_A \otimes \pi) \circ \Delta_{A,H} \colon A \to A \otimes H_0$.

\label{def:quantumreduction}
\end{definition}

Consider the algebraic reduction induced by a surjective right $H_0$-comodule algebra morphism $\phi \colon A \to A_0$. Let $(\Omega^1(A),\mathrm{d})$ be a right $H$-covariant FODC on $A$, with corresponding right $H$-coaction $\Delta_{\Omega^1(A)}$, and consider on $A_0$ the quotient calculus $(\Omega^1(A_0), \mathrm{d}_0)$ induced by $\phi$, as in Definition \ref{def:quotient_and_pullback_calculi}. We denote by
\begin{equation}
    \begin{split}
        \Phi \colon &\Omega^1(A) \to \Omega^1({A_0}),\\
        &a\mathrm{d}a^{\prime} \mapsto \phi(a)\mathrm{d}_{0}\phi(a^{\prime}),
    \end{split}
    \label{eq:PhiGamma}
\end{equation}
the corresponding morphism of FODCi. Note that if $(\Omega^1(A),\mathrm{d})$ is $H$-covariant, the quotient calculus $(\Omega^1(A_0), \mathrm{d}_0)$ is $H_0$-covariant by construction. We denote the corresponding $H_0$-coaction by $\Delta_{\Omega^1(A_0)}$. Moreover, we consider the pullback FODCi $(\Omega^1(B),\mathrm{d})$ on $B$ and $(\Omega^1(B_0),\mathrm{d}_0)$ on $B_0$, induced by the inclusions $\iota \colon B \hookrightarrow A$ and 
$\iota_0 \colon B_0 \hookrightarrow A_0$.
Since $\phi$ is an isomorphism between the base spaces $B$ and $B_0$, the restriction of $\Phi$ to $\Omega^1(B) \subseteq \Omega^1(A)$ is an isomorphism of FODCi.
\begin{definition}
    Let $(A,H,V,\theta)$ be a quantum frame resolution and $A_0$ be a quantum reduction of $A$ to the Hopf algebra $H_0 := H/J$, with $J$ a Hopf ideal. We call the tuple $(A_0,H_0,V,\theta_0)$, with
    \begin{equation}
        \theta_0 := \Phi \circ \theta,
        \label{eq:theta0definition}
    \end{equation}
    a \emph{quantum} G\emph{-structure} of $(A,H,V,\theta)$.
    \label{def:quantumGstructure}
\end{definition}
Note that the above definition is well-posed only if $\theta_0$ is right $H_0$-strongly tensorial. Let us prove in the following lemma that this is the case.
\begin{lemma}
    Let $\theta \colon V \to A\Omega^1(B)$ a right $H$-strongly tensorial form. Then $\theta_0 := \Phi \circ \theta$ is right $H_0$-strongly tensorial, that is, it has values in $A_0\Omega^1(B_0)$ and it satisfies
    \begin{equation}
        \Delta_{\Omega^1(A_0)} \circ \theta_0 = (\theta_0 \otimes \mathrm{id}_{H_0}) \circ \Delta_{V,H_0},
    \end{equation}
    where $\Delta_{V,H_0} := (\mathrm{id}_V \otimes \pi) \circ \Delta_V$.
    \label{lem:theta0isequivariant}
\end{lemma}
\begin{proof}
    By construction, the restriction of $\Phi$ to $A\Omega^1(B)$ has values in $A_0\Omega^1(B_0)$, and thus $\theta_0$ is horizontal. Moreover, $H$-equivariance of $\theta$ immediately implies $H_0$-equivariance of $\theta$ with respect to the coaction $\Delta_{\Omega^1(A),H_0}:=(\mathrm{id}_{\Omega^1(A)} \otimes \pi) \circ \Delta_{\Omega^1(A)}$, since
    \[
        \Delta_{\Omega^1(A),H_0} \circ \theta = (\mathrm{id}_V \otimes \pi) \circ \Delta_{\Omega^1(A)} \circ \theta = (\theta \otimes \pi) \circ \Delta_V = (\theta \otimes \mathrm{id}_{H_0}) \circ \Delta_{V,H_0}.
    \]
    Finally, we have that
    \[
        \begin{split}
            \Delta_{\Omega^1(A_0)} \circ \theta_0& = \Delta_{\Omega^1(A_0)} \circ \Phi \circ \theta =  (\Phi \otimes \mathrm{id}_{H_0}) \circ \Delta_{\Omega^1(A),H_0} \circ \theta \\
            &= (\Phi \otimes \mathrm{id}_{H_0}) \circ (\theta \otimes \mathrm{id}_{H_0}) \circ \Delta_{V,H_0} = (\theta_0 \otimes \mathrm{id}_{H_0}) \circ \Delta_{V,H_0},
        \end{split}
    \]
    since $\Phi$ is clearly a morphism of right $H_0$-comodules.

\end{proof}
Motivated by classical differential geometry, where reductions preserve the structure of the frame bundle, it becomes natural to ask whether a quantum G-structure is again a quantum frame resolution. The following theorem provides an affirmative answer.

\begin{theorem}\label{th:Gstructure}
    Let $(A,H,V,\theta)$ be a quantum frame resolution of $(B,\Omega^1(B))$ and let $(A_0, H_0, V, \theta_0)$ be a quantum G-structure of $(A,H,V,\theta)$. Then $(A_0,H_0,V,\theta_0)$ is a quantum frame resolution of $(B_0,\Omega^1(B_0))$.
\end{theorem}
\begin{proof}
    Let 
    \[
        \mathcal{E}:=(A \otimes V)^{\mathrm{co}H}, \quad \mathcal{E}_0 := (A_0 \otimes V)^{\mathrm{co}H_0}\,
    \]
    be the quantum associated bundles to the quantum principal bundles $B \subseteq A$ and $B_0 \subseteq A_0$ respectively, where the right $H$-comodule $(V,\Delta_V)$ is regarded as a right $H_0$-comodule via the induced coaction $\Delta_{V,H_0} := (\mathrm{id}_V \otimes \pi) \circ \Delta_V$. Then, if $\phi \colon A \to A_0$ is the surjective morphism of right $H_0$-comodule algebras defining the reduction, the map $\phi \otimes \mathrm{id}_V$ restricts to a map $\mathcal{E} \to \mathcal{E}_0$. In fact, for any $a \otimes v \in \mathcal{E}$, we have that
    \[
        \begin{split}
            \Delta_{A_0 \otimes V} \circ (\phi(a) \otimes v) &= \phi(a_0) \otimes v_0 \otimes \pi(a_1)\pi(v_1) = (\phi \otimes \pi) \circ \Delta_{A \otimes V}(a \otimes v)\\
            &= \phi(a) \otimes v \otimes 1.
        \end{split}
    \]
    
    Consider now the following diagram:
    \[
    \begin{centering}
    \begin{tikzcd}
    \mathcal{E} \arrow[ddd, "(\phi \otimes \mathrm{id}_V)"] \arrow[rrrrr, "s_{\theta}", shift left] &  &  &  &  & \Omega^1(B) \arrow[lllll, "s_{\theta^{-1}}", shift left=2] \arrow[ddd, "\Phi\big|_{\Omega^1(B)}", shift left] \\
                                                                                         &  &  &  &  &                                                                                                                  \\
                                                                                         &  &  &  &  &                                                                                                                  \\
    \mathcal{E}_0 \arrow[rrrrr, "s_{\theta_0}"]                                          &  &  &  &  & \Omega^1(B_0) \arrow[uuu, "\Phi\big|_{\Omega^1(B)}^{-1}", shift left=2]                                    
    \end{tikzcd},
    \end{centering}
\]
    where, as discussed, the restriction of $\Phi$ to base forms is an isomorphism, and $s_{\theta} := \cdot \circ (\mathrm{id}_A \otimes \theta) \colon \mathcal{E} \to \Omega^1(B)$, where $\cdot$ denotes the $A$-action on $\Omega^1(A)$, is an isomorphism of left $B$-modules, since by assumption $(A,H,V,\theta)$ is a quantum frame resolution. Since $\theta_0$ is right strongly tensorial, the map $s_{\theta_0}:= \cdot_0 \circ (\mathrm{id}_{A_0} \otimes \theta_0)$ is a left $B_0$-module map $\mathcal{E}_0 \to \Omega^1(B_0)$ by Proposition \ref{pro:thecorrespondence2}, where we denote by $\cdot_0$ the $A_0$-action on $\Omega^1(A_0)$. We show that the map
    \[
        \Psi := (\phi \otimes \mathrm{id}_V) \circ s_{\theta}^{-1} \circ \Phi\big|_{\Omega^1(B)}^{-1} \colon \Omega^1(B_0) \to \mathcal{E}_0,
        \label{eq:LAMAPPA}
    \]
    is both a left and a right inverse of $s_{\theta_0}$.

    For all $\bar{a} \otimes v \in \mathcal{E}_0$ with $\bar{a} := \phi(\tilde{a})$ for some $\tilde{a} \in A$, we have that

        \[
            \begin{split}
                \Psi \circ s_{\theta_0} (\bar{a} \otimes v) &= (\phi \otimes \mathrm{id}_V) \circ s_{\theta}^{-1} \circ \Phi\big|_{\Omega^1(B)}^{-1} \circ \cdot_0 \circ \big(\mathrm{id}_{A_0} \otimes (\Phi \circ \theta)\big) (\phi(\tilde{a}) \otimes v) \\
                &=(\phi \otimes \mathrm{id}_V) \circ s_{\theta}^{-1} \circ \Phi\big|_{\Omega^1(B)}^{-1} \circ \cdot_0 \circ (\mathrm{id}_{A_0} \otimes \Phi) \circ (\phi \otimes \theta)(\tilde{a} \otimes v)  \\
                &= (\phi \otimes \mathrm{id}_V) \circ s_{\theta}^{-1} \circ \Phi\big|_{\Omega^1(B)}^{-1} \circ \cdot_0 \circ (\phi \otimes \Phi) \circ (\mathrm{id}_A \otimes \theta)(\tilde{a} \otimes v)\\
                &= (\phi \otimes \mathrm{id}_V) \circ s_{\theta}^{-1} \circ \Phi\big|_{\Omega^1(B)}^{-1} \circ \Phi \circ \cdot \circ (\mathrm{id}_A \otimes \theta)(\tilde{a} \otimes v)  \\
                &= (\phi \otimes \mathrm{id}_V) \circ s_{\theta}^{-1} \circ \Phi\big|_{\Omega^1(B)}^{-1} \circ \Phi\big|_{\Omega^1(B)} \circ s_{\theta}(\tilde{a} \otimes v) \\
                &= (\phi \otimes \mathrm{id}_V)(\tilde{a} \otimes v) = \bar{a} \otimes v,
            \end{split}
        \]
        
       \noindent  where we used that $(\phi, \Phi)$ is a morphism of differential calculi.
       
        On the other hand, similarly,

            \begin{align*}
                s_{\theta_0} \circ \Psi &= s_{\theta_0} \circ (\phi \otimes \mathrm{id}_V) \circ s_{\theta}^{-1} \circ \Phi\big|_{\Omega^1(B)}^{-1}\\
                &=\cdot_0 \circ \big(\mathrm{id}_{A_0} \otimes (\Phi \circ \theta)\big) \circ (\phi \otimes \mathrm{id}_V) \circ s_{\theta}^{-1} \circ \Phi\big|_{\Omega^1(B)}^{-1} \\
                &= \cdot_0 \circ (\phi \otimes \Phi) \circ (\mathrm{id}_A \otimes \theta) \circ s_{\theta}^{-1} \circ \Phi\big|_{\Omega^1(B)}^{-1}\\
                &= \Phi \circ \cdot \circ (\mathrm{id}_A \otimes \theta) \circ s_{\theta}^{-1} \circ \Phi\big|_{\Omega^1(B)}^{-1}\\
                &= \Phi\big|_{\Omega^1(B)} \circ s_{\theta} \circ s_{\theta}^{-1}\circ \Phi\big|_{\Omega^1(B)}^{-1} = \mathrm{id}_{\Omega^1(B_0)},
            \end{align*}
            
    concluding the proof. 
\end{proof}
Our next goal is to construct examples of G-structures within the Radford--Majid biproduct setting. The following result proves to be very useful for this purpose. 
\begin{proposition}
\label{prop:split_hopf_and_YD}
    Let $H,H'$ be Hopf algebras in  ${}_{\Bbbk}\mathrm{Vec}$, let $\uB\in {}^{H}_{H}\mathcal{YD}$ be a braided Hopf algebra, and let $\pi\colon H\to H'$ be a Hopf algebra surjection with a section $\iota\colon H'\to H$ which is also a Hopf algebra map. It follows that $\uB$ is a braided Hopf algebra in ${}^{H'}_{H'}\mathcal{YD}$. 
    \begin{proof}
     Letting $\triangleright\colon H\otimes \uB \to \uB$ and ${}_{H}\Delta\colon \uB \to H\otimes \uB$ respectively denote the $H$-action and $H$-coaction on $\uB$, one defines $H'$-action and $H'$-coaction on $\uB$ as 
      \[
      \begin{aligned}
      \triangleright'\colon H'\otimes \uB & \to \uB, & {}_{H'}\Delta\colon \uB &\to H'\otimes \uB,\\
      h'\otimes b &\mapsto \iota(h')\triangleright b, & b&\mapsto (\pi\otimes \mathrm{id})\circ {}_{H}\Delta,
      \end{aligned}
      \]
      and one verifies that the Yetter--Drinfeld compatibility between the latter holds. The claim now follows because, as is evident from the above assignments, any morphism in ${}^{H}_{H}\mathcal{YD}$ is also a morphism in ${}^{H'}_{H'}\mathcal{YD}$.
    \end{proof}
\end{proposition}

\begin{example} 
Consider Sweedler's Hopf algebra $E_{n}$ as in Example \ref{ex:sweedler_hopf}. There is an $n$-dimensional differential calculus on the latter given by the $\Omega^{1}(E_{n}):=\mathrm{span}_{E_{n}}\left\{\mathrm{d}x_{1},\dots,\mathrm{d}x_{n}\right\}$, with relations 
\begin{equation}
x_{i}\,\mathrm{d}x_{j}=-\mathrm{d}x_{j}\, x_{i},\qquad g\,\mathrm{d}x_{i}=-\mathrm{d}x_{i}\, g\,.
\end{equation}

\noindent This differential calculus is bicovariant: having assumed $\mathrm{d}g=0$ translates in a symmetric lifting of the coproduct at the level of the differential $1$-forms. In particular, for $i=1,\dots,n$, we have 
\begin{align*}
\Delta_{E_{n}}(\mathrm{d}x_{i}) = \mathrm{d}x_{i}\otimes 1 , \qquad {}_{{E_{n}}}\Delta(\mathrm{d}x_{i}) = 1\otimes \mathrm{d}x_{i} . 
\end{align*}
Considering the transmutation of $E_{n}$ in the braided monoidal category  $({}_{E_{n}}\mathrm{Mod},\otimes,\sigma_{\mathcal{R}_{0}})$ of left $E_{n}$-modules, we have a braided $\underline{E_{n}}$-bicovariant first order differential calculus $\Omega^{1}(\underline{E_{n}})$ in ${}_{E_{n}}\mathrm{Mod}$  on $\underline{E_{n}}$, again as in Example \ref{ex:sweedler_hopf}. As for any quasitriangular Hopf algebra in ${}_{\Bbbk}\mathrm{Vec}$, there is a braided monoidal functor 
\[
\mathcal{F}\colon \left({}_{H}\mathrm{Mod},\otimes,\sigma_{\mathcal{R}}\right)\to \left({}^{H}_{H}\mathcal{YD},\otimes,\sigma_{\mathcal{YD}}\right)
\]
enriching any object $M$ in ${}_{H}\mathrm{Mod}$ with the (Yetter--Drinfeld compatible) left $H$-coaction 
\[
\begin{aligned}
{}_{H}\Delta\colon M & \to H\otimes M,\\
m & \mapsto m_{-1}\otimes m_{0}=:\mathcal{R}^{2}\otimes \mathcal{R}^{1}\triangleright m, 
\end{aligned}
\]
we in particular have $\underline{E_{n}},\Omega^{1}(\underline{E_{n}}) \in \left({}^{E_{n}}_{E_{n}}\mathcal{YD},\otimes,\sigma_{\mathcal{YD}}\right)$, with left $H$-coaction given on generators by 
\[
    {}_{\underline{E_{n}}}\Delta(x_{i})  = g\otimes x_{i}, \qquad {}_{\Omega^{1}(\underline{E_{n}})}\Delta(\mathrm{d}x_{i}) = g \otimes \mathrm{d}x_{i}. 
\]
Therefore, as $\underline{E_{n}}$ is now understood as a braided Hopf algebra in ${}^{E_{n}}_{E_{n}}\mathcal{YD}$, and as $\Omega^{1}(\underline{E_{n}})$ is now understood as braided $\underline{E_{n}}$-bicovariant FODC in ${}^{E_{n}}_{E_{n}}\mathcal{YD}$, we may consider the Radford--Majid biproduct $\underline{E_{n}}\# E_{n}$ alongside the first order smash product differential calculus $\left(\Omega^{1}(\underline{E_{n}}\# E_{n}),\mathrm{d}_{\#,n}\right)$. It may be readily checked that for the case at hand the map $\mathrm{ver}\colon \Omega^{1}(\underline{E_{n}})\to \Omega^{1}(E_{n})\otimes \underline{E_{n}}$ in Theorem \ref{thm:onbicovarianceofsmash} is well-defined. In particular, it is identically zero, and therefore we have a bicovariant FODC on $\underline{E_{n}}\# E_{n}$.

In light of Proposition \ref{prop:split_hopf_and_YD}, and by means of the Hopf algebra surjection 
\[
\begin{aligned}
    \pi \colon E_{n}& \twoheadrightarrow E_{n-1}, \\ 
    (x_{1},\dots,x_{n},g,1) & \mapsto (x_{1},\dots,x_{n-1},g,1) \\ 
\end{aligned}
\]
admitting the identity as a section, we may also understand $\underline{E_{n}}$ as a braided Hopf algebra in ${}^{E_{n-1}}_{E_{n-1}}\mathcal{YD}$ and $\Omega^{1}(\underline{E_{n}})$ as a  braided $\underline{E_{n}}$-bicovariant FODC in ${}^{E_{n-1}}_{E_{n-1}}\mathcal{YD}$. 

Consider now the Radford--Majid biproduct $\underline{E_{n}}\# E_{n}$ and the first order smash product differential calculus $\left(\Omega^{1}(\underline{E_{n}}\#E_{n-1}),\mathrm{d}_{\#,n-1}\right)$, which is again bicovariant for similar reasoning. The calculus $\Omega^{1}(\underline{E_{n}}\#E_{n-1})$ is induced as the quotient calculus of $\Omega^{1}(\underline{E_{n}}\#E_{n})$ via the Hopf algebra surjection $\phi:=(\mathrm{id}\otimes\pi)\colon \underline{E_{n}}\# E_{n} \twoheadrightarrow \underline{E_{n}}\# E_{n-1}$ we have a map 
\[
\begin{aligned}
    \Phi\colon \left(\Omega^{1}(\underline{E_{n}}\#E_{n}),\mathrm{d}_{\#,n}\right) &\to \left(\Omega^{1}(\underline{E_{n}}\#E_{n-1}),\mathrm{d}_{\#,n-1}\right), \\ 
    E\,\mathrm{d}_{\#,n}E'& \mapsto \phi(E)\,\mathrm{d}_{\#,n-1}\circ \phi(E').
\end{aligned}
\]
Thus, the construction presented above provides an example of quantum $\mathrm{G}$-structure in agreement with Theorem \ref{th:Gstructure}. A clarifying diagram follows. 
\[
\begin{tikzcd}
\Omega^{1}(\underline{E_{n}}\#E_{n}) \arrow[rrrr, "\Phi"] &                                                                &                                                       &                            & \Omega^{1}(\underline{E_{n}}\# E_{n-1}) \\
                                                                                        & \underline{E_{n}} \# E_{n} \arrow[rr, "\phi "] &                                                       & \underline{E}_n \# E_{n-1} &                                                    \\
                                                                                        &                                                                &                                                       &                            &                                                    \\
                                                                                        &                                                                & \underline{E_{n}} \arrow[luu, hook] \arrow[ruu, hook] &                            &                                                   
\end{tikzcd}
\]
\qed

\end{example}
Classically, on a 2-dimensional manifold, a reduction of the frame bundle to the orthogonal group $\mathrm{O}_2$ is equivalent to the choice of a Riemannian metric, while a reduction to $\mathrm{SO}_2\cong \mathrm{U}_1$ additionally specifies an orientation. 
Building on Examples \ref{ex:C_q(2)_braided} and \ref{example:2D_calculus_on_the_quantum_plane} we now extend this perspective to the frame bundle of the braided quantum plane.

\begin{example}

Fix the notation as above, namely let $\uB:=\underline{\mathbb{C}}_{q}^{2}$, $H:=\mathcal{O}_{q}(\mathrm{GL}_{2})$, $H_{0}:=\mathcal{O}(\mathrm{SO}_{2})\cong \mathcal{O}(\mathrm{U}_1):=\mathbb{C}[\alpha,\alpha^{-1}]$, and consider the Hopf algebra surjection 
\begin{equation}
\pi\colon H  \twoheadrightarrow H_{0}, \qquad  
\begin{pmatrix}
    a & b \\ c & d
\end{pmatrix}  \mapsto \begin{pmatrix}
    \alpha & 0 \\ 
    0 & 1
\end{pmatrix}, \qquad 
D^{-1}  \mapsto \alpha^{-1},
\end{equation}

\noindent whose kernel, which is a Hopf ideal, we denote by $K:=\ker\pi$. Consider the  bicovariant $4$-dimensional differential calculus $(\Omega^{1}(H):=H\otimes H^{+}/I,\mathrm{d})$, whose cotangent space $H^{+}/I\cong\Lambda^{1}_{H}$ is explicitly given in \cite{aflw} in terms of generators $\omega^{i}$. To induce a quotient calculus on $H_{0}$ via the surjection $\pi$, we construct a right $H_{0}$-ideal contained in $H_{0}^{+}$. Define $J:=\pi(I)$. This is a right $H_{0}$-ideal, since

\[
J\cdot_{0}H_{0}=\pi(I)\cdot_{0}\pi(H)=\pi(I\cdot H)=\pi(I)=J.
\]

\noindent Moreover $J\subseteq H_{0}^{+}$: indeed $I\subseteq H^{+}$ by hypothesis, so $J=\pi(I)\subseteq\pi(H^{+})=H_{0}^{+}$. Hence $J$ defines a left covariant differential calculus
\[
\Omega^{1}(H_{0}):=H_{0}\otimes H_{0}^{+}/J
\]
on $H_{0}$, with cotangent space $H_{0}^{+}/J\cong\Lambda^{1}_{H_{0}}$.

\smallskip

We now give an explicit realisation of $\Lambda^{1}_{H_{0}}$.  Since $\varepsilon_{0}\circ\pi=\varepsilon$, for every $k\in K$ we have 
\[
\varepsilon(k)=\varepsilon_{0}(\pi(k))=\varepsilon_{0}(0)=0,
\]
and thus $K\subseteq H^{+}$. Consequently, the restriction $\pi|_{H^{+}}\colon H^{+}\to H_{0}^{+}$ is well defined and, since $\pi$ is surjective and $\varepsilon=\varepsilon_{0}\circ\pi$, it is again surjective, with kernel $H^{+}\cap K=K$. Moreover, the induced map at the level of the cotangent spaces, that we denote by $\widetilde{\pi}\colon\Lambda^{1}_{H}\to\Lambda^{1}_{H_{0}}$, is automatically well defined because $\pi(I)\subseteq J$. Furthermore, since $\pi|_{H^+}$ is surjective, $\widetilde{\pi}$ is surjective as well. Accordingly, the images of the generators $\omega^{i}$ span $\Lambda^{1}_{H_{0}}$. We get
\[
\begin{aligned}
\widetilde{\pi}(\omega^{1})&=Q'(q^{-2}-q^{-1})\,\alpha^{-1}\mathrm{d}\alpha, & \widetilde{\pi}(\omega^{2})&=0,\\
\widetilde{\pi}(\omega^{3})&=0, & \widetilde{\pi}(\omega^{4})&=Q'(q^{-5}-q^{-3}-q^{-2}+q^{-1})\,\alpha^{-1}\mathrm{d}\alpha.
\end{aligned}
\]
Thus, $\Lambda^{1}_{H_{0}}$ is 1-dimensional, spanned by the single element $\varpi(\alpha):=\alpha^{-1}\mathrm{d}\alpha$. Finally, $\Omega^{1}(H_{0})$ is  bicovariant. Indeed, as $\mathrm{coAd}_{\mathrm{R}}(I)\subseteq I\otimes H$, and since $\pi$ is a Hopf algebra map, we have 
\begin{align*}
\mathrm{coAd}_{\mathrm{R}}(J)& =\mathrm{coAd}_{\mathrm{R}}(\pi(I))\\ 
& =(\pi\otimes\pi)\bigl(\mathrm{coAd}_{\mathrm{R}}(I)\bigr)\\ 
& \subseteq(\pi\otimes\pi)(I\otimes H)=J\otimes H_{0}.
\end{align*}

\noindent Hence, $J$ is closed under $\mathrm{coAd}_{\mathrm{R}}$ and $\Omega^{1}(H_{0})$ is bicovariant. 

As the Hopf algebra map $\pi\colon H \twoheadrightarrow H_{0}$ has a Hopf algebra section $\iota \colon H_{0} \to H$ sending $\alpha \mapsto D$, we understand $\uB\in {}^{H_{0}}_{H_{0}}\mathcal{YD}$ in light of Proposition \ref{prop:split_hopf_and_YD}, and the above construction gives an example of a quantum G-structure in accordance with the following diagram: 
\[
\begin{tikzcd}
\Omega^{1}(\underline{\mathbb{C}}_{q}^{2}\# \mathcal{O}_{q}(\mathrm{GL}_{2})) \arrow[rrrr, "\Phi"] &                                                                         &                                                       &                                                     & \Omega^{1}(\underline{\mathbb{C}}_{q}^{2}\# \mathcal{O}(\mathrm{SO}_{2})) \\
                                                                                      & \underline{\mathbb{C}}_{q}^{2}\# \mathcal{O}_{q}(\mathrm{GL}_{2}) \arrow[rr, "\phi"] &                                                       & \underline{\mathbb{C}}_{q}^{2}\# \mathcal{O}(\mathrm{SO}_{2}) &                                                                 \\
                                                                                      &                                                                         &                                                       &                                                     &                                                                 \\
                                                                                      &                                                                         & \underline{\mathbb{C}}_{q}^{2} \arrow[luu, hook] \arrow[ruu, hook] &                                                     &                                                                
\end{tikzcd}
\]
where $\phi:=(\mathrm{id}\otimes \pi)$ and $\Phi$ is the induced morphism of FODCi. 
\qed
\end{example}

\appendix

\section{Cartan geometries, frame bundles, and G-structures} \label{section:classicalpicture}
This Appendix provides a point of comparison between the results obtained in the paper and the classical differential geometric picture. In the following, we work in the category of $n$-dimensional real or complex smooth manifolds, and we denote by $\rGL(n,\Bbbk)$ or simply $\rGL(n)$ the general linear group on $\Bbbk^n$. 

\begin{definition}
    Let $\rH \subset \rG$ be a closed Lie subgroup of a Lie group $\rG$. A \emph{Cartan geometry} of type $(\rG,\rH)$ on a manifold $M$ is a principal $\rH$-bundle $\pi \colon E \to M$ together with a $\mathrm{Lie}(\rG)=:\mathfrak{g}$-valued $1$-form $\omega \in \Omega^1(E,\mathfrak{g})$ called the \emph{Cartan connection}, satisfying
    \begin{equation}
            \label{eq:propertiesofCartanconnection}
        \begin{aligned}
            &\mathrm{R}_h^{\ast} \omega = \mathrm{ad}_{h^{-1}} \circ \omega, & &  \text{for all } h \in \rH,\\
            &\omega\big(\zeta(X)\big) = X, & &  \text{for all } X \in \mathfrak{h},\\
            &\omega_p \colon T_pE \to \mathfrak{g} & & \text{is a linear isomorphism for all } p \in E.
        \end{aligned}
        \end{equation}
    \end{definition}
    Here, $\mathrm{R}_h$ denotes the right action of $\rH$ on fibers, $\mathrm{ad}_\bullet$ denotes the adjoint action of a Lie group on its Lie algebra, and $\zeta(X)$ is the fundamental vector field on $E$ associated with $X \in \mathfrak{h}$. The Cartan connection $\omega$ has \emph{curvature} given by the $\mathfrak{g}$-valued $2$-form
    \begin{equation}
        \label{eq:curvatureCartanconnection}
        \Omega := \mathrm{d}\omega + \tfrac{1}{2}[\omega,\omega] \in \Omega^2(E,\mathfrak{g}),
    \end{equation}
    measuring the failure of $\omega$ to satisfy the Maurer--Cartan equation. We now specialise, in the following definition, to a case of particular interest.
    \begin{definition}
    The \emph{homogeneous model} for a Cartan geometry of type $(\rG,\rH)$ is a principal bundle $\rG \to \rG/\rH$ endowed with a $1$-form $\varpi \in \Omega^1(\rG,\mathfrak{g})$, the \emph{Maurer--Cartan form}, which plays the role of a Cartan connection with vanishing curvature, i.e. $\Omega=0$ in \eqref{eq:curvatureCartanconnection}. 
\label{def:homogeneous_model_cartan}
\end{definition}
Consider the frame bundle $\pi:\mathrm{F}M \to M$ viewed as a principal $\rGL(n)$-bundle. It is canonically equipped with an $\Bbbk^{n}$ valued $1$-form $\theta$ that is horizontal and $\rGL(n)$-equivariant. More in detail, viewing the fiber $\mathrm{F}M_p$ over $p \in M$ as the space of all linear isomorphisms $\varphi$ between $\Bbbk^n$ and $\mathrm{T}_pM$, there exists a $1$-form with values in $\Bbbk^n$, called the \emph{soldering form} (or \emph{canonical form}) defined on  general vector fields $\xi \in \mathfrak{X}(\mathrm{F}M)$ as
    \begin{equation}
    \label{eq:definitioncanonicalform}
        \theta_\varphi (\xi_\varphi) := \varphi^{-1} (\pi_{\ast} \xi_\varphi) \in \Bbbk^n.
    \end{equation}
    The soldering form is horizontal and $\rGL(n)$-equivariant by construction:
    \[
    (\mathrm{R}_g)^*\theta=g^{-1}\theta,
    \]
    with $\mathrm{R}_g:\mathrm{F}M\to \mathrm{F}M$ being the fiber-wise action induced by the right multiplication by an element $g\in \rGL(n)$, while on the right hand side of the previous equation we use the standard action of $ \rGL(n)$ on $\Bbbk^n$. As the tangent bundle of a manifold $M$ is isomorphic to the associated bundle $\mathrm{F}M\times_{\rGL(n)} \Bbbk^n$, we have that connections on the frame bundle can be viewed as  linear connections on the tangent bundle. Moreover  we can then identify a principal $\rGL(n)$-bundle over $M$ equipped with an equivariant, horizontal and $\Bbbk^n$ valued differential form, with the frame bundle.

   An explicit realisation of the above construction would be the affine $n$-space $\mathbb{A}^n$, which can indeed be constructed as the homogeneous model

\begin{equation}
    \label{eq:affinemspace}
    \mathbb{A}^{n} \cong \textrm{Aff}(n) /\rGL(n),
\end{equation}

\noindent with $ \textrm{Aff}(n) = \Bbbk^{n}\rtimes \rGL(n,\Bbbk)$ being the affine group. More explicitly, employing a matrix representation, we have that
\begin{equation}
\label{eq:affineextension}
    \textrm{Aff}(n) = \left\{  \begin{pmatrix} 1 & 0 \\ b & A  \end{pmatrix} \mid A \in \rGL(n,\Bbbk),\, b \in \Bbbk^{n}\right\} \subset \rGL(n+1,\Bbbk),
\end{equation}

\noindent whose Lie algebra is

\begin{equation}
\label{eq.defliealgebraassociatedtoaffinegroup}
    \mathfrak{aff}(n) = \left\{  \begin{pmatrix} 0 & 0 \\ B & X  \end{pmatrix} \mid X \in \mathfrak{gl}(n),\, B \in \Bbbk^{n}\right\} \cong \Bbbk^{n}\oplus \mathfrak{gl}(n).
\end{equation}
For the purpose of the paper, it is of central relevance to observe that the splitting $\mathfrak{aff}(n)=\Bbbk^n\oplus\mathfrak{gl}(n)$ presented above is preserved by the adjoint action of $\mathrm{Aff}(n)$ restricted to $\rGL(n)$, so that $\mathfrak{aff}(n)$ splits as a direct sum of $\rGL(n)$-modules. Considering the principal bundle $\textrm{Aff}(n)\to \textrm{Aff}(n)/  \rGL(n)$, the Maurer--Cartan form on the affine group splits  according to this direct sum decomposition as
$\varpi=\theta +\gamma$, where $\theta\in \Omega^1(\textrm{Aff}(n),\Bbbk^n)$ and $\gamma\in \Omega^1(\textrm{Aff}(n),\mathfrak{gl}(n))$ are both $\rGL(n)$-equivariant forms. We stress that this $\theta$ coincides with the soldering form of $\eqref{eq:definitioncanonicalform}$ once $\mathrm{Aff}(n)$ is identified with $\mathrm{F}\mathbb{A}^n$ (the frame bundle of the affine space $\mathbb{A}^n$); this identification is precisely what justifies referring to the $\Bbbk^n$-component of $\varpi$ by the same symbol and name as the canonical form on a general frame bundle. In particular, since $\theta$ is also a horizontal $1$-form, we can identify $\textrm{Aff}(n)$ with the frame bundle $\mathrm{F}\mathbb{A}^{n}$, and thus $\gamma$ with an affine connection on $\mathbb{A}^{n}$. This construction can be generalised to the case of non-vanishing curvature, in the Cartan sense, by considering a $\rGL(n)$ principal bundle $P\to M$ equipped with a Cartan connection $\omega\in\Omega^1(P,\mathfrak{aff}(n))$. However, as here we deal with the noncommutative generalisation of the affine model, we consider a full description of the latter to be beyond the intent of the paper, and thus we address the reader to the monograph \cite{parabook} for more details.
    
\smallskip
    
We now turn our attention to the reduction of the frame bundle $\mathrm{F}M \to M$ to a closed subgroup $\rH \subset \rGL(n)$ of the general linear group, more commonly referred to as an \emph{$\rH$-structure} on $M$. It is worth mentioning that
many relevant geometrical structures arise in this way. For example, considering $\mathrm{O}_n(\mathbb{R})\subset \rGL(n,\mathbb{R})$, one gets a Riemannian structure on $M$. Similarly, for $n$ an even, positive integer, and for $\rH=\mathrm{Sp}(n,\mathbb{R})\subset \rGL(n,\mathbb{R})$
the reduction of the frame bundle induces an almost symplectic structure. We recall the following standard result to the reader. 

\begin{proposition}
\label{pro:thepowerofGstructures}
An $\rH$-structure $\pi \colon P \to M$, with structure group $\rH \subseteq \rGL(n,\Bbbk)$, is equivalently described by a principal $\rH$-bundle $\pi \colon P \to M$ endowed with a horizontal, $\rH$-equivariant $\Bbbk^n$-valued $1$-form $\Theta \in \Omega^1(P,\Bbbk^n)$ such that, for every $p\in P$, the restriction of $\Theta_p$ to the horizontal tangent space is an isomorphism.
\end{proposition}

In particular, given the reduction $\iota: P\to \mathrm{F}M$, one gets $\Theta\in \Omega^{1}(P,\Bbbk^{n})$ by pulling back $\theta$ along the principal bundle morphism $\iota$. 

Connections on $\rH$-structures can be described in a similar style as in the affine picture. One starts with the affine extension of $\rH$, that is, $\textrm{Aff}_\rH(n)=\Bbbk^{n}\rtimes \rH$, and the principal bundle $\textrm{Aff}_\rH(n) \to \textrm{Aff}_\rH(n)/ \rH$. The Maurer--Cartan form on $\textrm{Aff}_{\rH}(n)$ splits equivariantly, with respect to the $\rH$-action, into a principal connection $\gamma$ and an equivariant horizontal  $\Bbbk^{n}$-valued form $\Theta$, of which we make use to identify the bundle $\textrm{Aff}_\rH(n)\to \textrm{Aff}_\rH(n)/ \rH$ with an $\rH$-structure. For the case in which $\rH = \textrm{O}(p,q) \subset \rGL(n,\mathbb{R})$, the affine extension is given by the pseudo-Euclidean group. Hence, this picture describes a $p+q$-dimensional pseudo-Riemannian manifold as a curved analogue of the pseudo-Euclidean space $\mathbb{E}^n$, the latter being an instance of the general categorical equivalence between between pseudo-Riemannian manifolds and torsion-free Cartan geometries of type $( \mathbb{R}^{p+q} \rtimes \textrm{O}(p,q),\textrm{O}(p,q))$.

\bigskip
\noindent\textbf{Contacts}
\medskip

\noindent
Antonio Del Donno:
\texttt{antonio.deldonno@matfyz.cuni.cz}

\medskip
\noindent
Giovanni Gava: \texttt{giovanni.gava@matfyz.cuni.cz}

\medskip
\noindent
Emanuele Latini: \texttt{emanuele.latini@unibo.it}

\medskip
\noindent
Thomas Weber: \texttt{thomas.weber@matfyz.cuni.cz}


\begin{thebibliography}{99}

\bibitem{Andr-Grana}
\textsc{Andruskiewitsch, N., Graña, M.},
\emph{Braided Hopf algebras over non abelian finite groups}.

\bibitem{paolo}
\textsc{Aschieri, P.},
\emph{Cartan structure equations and Levi-Civita connection in braided geometry}.
J. Geom. Phys. \textbf{224}, 105800 (2026).

\bibitem{afl}
\textsc{Aschieri, P., Fioresi, R., Latini, E.},
\emph{Quantum Principal Bundles on Projective Bases}.
Commun. Math. Phys. \textbf{382} (2021) 1691--1724.

\bibitem{aflw}
\textsc{Aschieri, P., Fioresi, R., Latini, E., Weber, T.},
\emph{Differential Calculi on Quantum Principal Bundles Over Projective Bases}.
Commun. Math. Phys. \textbf{405} (2024) 136.


\bibitem{AzizMajid}
\textsc{Aziz, R., Majid, S.},
\emph{Quantum differentials on cross product Hopf algebras}.
J. Algebra. \textbf{563} (2020)
303-351.


\bibitem{bm}
\textsc{Beggs, E. J., Majid, S.},
\emph{Quantum Riemannian Geometry}.
Springer, 2020.

\bibitem{Julius}
\textsc{Benner, J.},
\emph{Codifferential Calculi on Quantum Homogeneous Spaces}.
Preprint arXiv:2606.09687.

\bibitem{BespalovCrossed}
\textsc{Bespalov, Yu. N.},
\emph{Crossed modules, quantum braided groups, and ribbon structures}.
Theor. Math. Phys. \textbf{103}, 3 (1995) 621–637.

\bibitem{BespalovBraidedFundThm}
\textsc{Bespalov, Yu. N., Drabant, B.},
\emph{Hopf (bi-)modules and crossed modules in braided monoidal categories}.
J. Pure Appl. Algebra \textbf{123} (1998) 105--129.

\bibitem{BespalovCalcs}
\textsc{Bespalov, Y., Drabant, B.}, \emph{Differential calculus in braided Abelian categories.} arXiv preprint q-alg/9703036.

\bibitem{Borowiec}
\textsc{Borowiec, A.},
\emph{Bicovariant Codifferential Calculi}.
Preprint arXiv:2602.12493.

\bibitem{Bor96}
\textsc{Borowiec, A.},
\emph{Cartan pairs}.
Czech. J. Phys. \textbf{46} (1996) 1197--1202.

\bibitem{Bor97}
\textsc{Borowiec, A.},
\emph{Vector Fields and Differential Operators: Noncommutative Case}.
Czechoslov. J. Phys. \textbf{47} (1997) 1093--1100.

\bibitem{Brz}
\textsc{Brzezi\'nski, T.},
\emph{Translation map in quantum principal bundles}.
J. Geom. Phys. \textbf{20} (1996) 349--370.

\bibitem{BrzMaj}
\textsc{Brzezi\'nski, T., Majid, S.},
\emph{Quantum group gauge theory on quantum spaces}.
Commun. Math. Phys. \textbf{157} (1993) 591--638;
Erratum: Commun. Math. Phys. \textbf{167} (1995) 235.

\bibitem{Bottegoni_Sciandra}
\textsc{Bottegoni, L., Renda, F., Sciandra, A.},
\emph{Infinitesimal R-matrices for some families of Hopf algebras.} J. Pure Appl. Algebra \textbf{230}, 3 (2026).

\bibitem{parabook}
\textsc{Čap, A., Slovák, J.}, 
\emph{Parabolic geometries I: Background and general theory}.
AMS Math. Surv. Monogr. \textbf{154} (2024).

\bibitem{delDLW}
\textsc{Del Donno, A., Latini, E., Weber, T.},
\emph{On the Durdević approach to quantum principal bundles}.
J. Geom. Phys. \textbf{216} (2025) 105567.

\bibitem{DoiTak}
\textsc{Doi, Y., Takeuchi, M.},
\emph{Cleft comodule algebras for a bialgebra}.
Commun. Algebra \textbf{14} (1986) 801--817.


\bibitem{pagani}
\textsc{Fioresi, R., Latini, E., Pagani, C.},
\emph{Reduction of Quantum Principal Bundles over Non-affine Bases}.
Preprint, arXiv:2403.06830.


\bibitem{Keegan}
\textsc{Flood, K. J., Lobbia, G., Tendas, G.},
\emph{Canonical Differential Calculi via Functorial Geometrization}.
Preprint arXiv:2512.20742.


\bibitem{GiovanniThesis}
\textsc{Gava, G.},
\emph{On quantum $G$-structures}.
Master Thesis, arXiv:2607.05426.


\bibitem{gunther}
\textsc{Günther, R.},
\emph{Crossed Products for Pointed Hopf Algebras}.
Comm. Algebra \textbf{27} (1999) 4389--4410.

\bibitem{hajac_strong}
\textsc{Hajac, P.M.},
\emph{Strong Connections on Quantum Principal Bundles}.
Commun. Math. Phys. \textbf{182} (1996) 579--617.

\bibitem{Hajac_plane}
\textsc{Hajac, P. M., Matthes, R.},
\emph{Frame, cotangent and tangent bundles of the quantum plane.}
In Proc. II Int. Workshop Lie Theory and its Applications in Physics II, ed. by H.-D. Doebner, V. K. Dobrev, J. Hilgert, World Scientific, Singapore, pp.~377--385.



\bibitem{hajac}
\textsc{Hajac, P. M., Rudnik, J., Zieliński, B.},
\emph{Reductions of Piecewise-Trivial Principal Comodule Algebras}.
Preprint arXiv:1101.0201.


\bibitem{KasselBook}
\textsc{Kassel, C.},
\emph{Quantum Groups}.
Graduate Texts in Mathematics 155, Springer, New York, 1995.

\bibitem{RosenbergKontsevich}
\textsc{Kontsevich, M., Rosenberg, A.}, \emph{Noncommutative smooth spaces.}
In The Gelfand mathematical seminars, 1996-1999 (pp. 85-108). Boston, MA: Birkh\"auser Boston.

\bibitem{MajidBraid}
\textsc{Majid, S.},
\emph{Braided groups}.
J. Pure Appl. Algebra. \textbf{86}, 2 (1993) 187--221.

\bibitem{Majid1991}
\textsc{Majid, S.},
\emph{Braided Groups and Algebraic Quantum Field Theories}.
Lett. Math. Phys. \textbf{22} (1991) 167--175.

\bibitem{MajidBos}
\textsc{Majid, S.}, 
\emph{Cross products by braided groups and bosonization}.
J. Algebra \textbf{163}, 1 (1994) 165-–190.

\bibitem{MajidFoundation}
\textsc{Majid, S.},
\emph{Foundations of Quantum Group Theory}.
Cambridge University Press, Cambridge, 1995.

\bibitem{Majid_braided_group_riemannian} 
\textsc{Majid, S.}, 
\emph{Quantum and braided group Riemannian geometry.} 
J. Geom. Phys. \textbf{30}, 2 (1999) 113-146.

\bibitem{MajidLectureNotes}
\textsc{Majid, S.},
\emph{Rank of Quantum Groups and Braided Groups in Dual Form}.
Lect. Notes Math. \textbf{1510} (1992) 79--89.

\bibitem{Reamonn}
\textsc{\'{O} Buachalla, R.}, \emph{Noncommutative complex structures on quantum homogeneous spaces.} 
J. Geom. Phys. \textbf{99} (2016) 154--173.

\bibitem{PflaumSchauenburg}
\textsc{Pflaum, M., Schauenburg, P.},
\emph{Differential Calculi on Noncommutative Bundles}.
Z. Phys. C \textbf{76} (1997) 733--744.

\bibitem{Radford}
\textsc{Radford, D.E.},
\emph{The Structure of Hopf Algebras with a Projection}.
J. Algebra \textbf{92} (1985) 322--347.

\bibitem{SchauenburgDG}
\textsc{Schauenburg, P.},
\emph{Differential-Graded Hopf Algebras and Quantum Group Differential Calculi}.
J. Algebra \textbf{180} (1996) 239-286.

\bibitem{Schneider}
\textsc{Schneider, H.-J.},
\emph{Principal Homogeneous Spaces for Arbitrary Hopf Algebras}.
Israel J. Math. \textbf{72} (1990) 167--195.


\bibitem{SciWeb}
\textsc{Sciandra, A., Weber, T.},
\emph{Noncommutative Differential Geometry on Crossed Product Algebras}.
J. Algebra \textbf{664} (2025) 129--176.

\bibitem{sharpe}
\textsc{Sharpe, R. W.}, 
\emph{Differential geometry: Cartan's generalization of Klein's Erlangen program}. Graduate Texts in Mathematics (Vol. 166), Springer Science Business Media, 2020.

\bibitem{Takeuchi}
\textsc{Takeuchi, M.},
\emph{Relative Hopf modules --- equivalences and freeness criteria}.
J. Algebra \textbf{60} (1979) 452--471.


\bibitem{Woronowicz1989}
\textsc{Woronowicz, S. L.},
\emph{Differential Calculus on Compact Matrix Pseudogroups (Quantum Groups)}.
Commun. Math. Phys. \textbf{122} (1989) 125--170.

\end{thebibliography}
\end{document}